\documentclass[11pt, reqno]{amsart} 
\usepackage{amsfonts,amssymb,amscd,amsthm,amsbsy,bbm,epsf,calc}
\usepackage{hyperref}
\hypersetup{
    colorlinks = true,
	linkcolor = {blue},
	citecolor = {blue},
}
\usepackage{dsfont} 
\usepackage{color, float}
\usepackage{datetime}
\usepackage{bm}
\usepackage[pdftex]{graphicx}
	
\makeatletter
\def\l@section{\@tocline{1}{0pt}{1pc}{}{}}
\def\l@susection{\@tocline{2}{0pt}{1pc}{4.6em}{}}
\def\l@subsubsection{\@tocline{3}{0pt}{1pc}{7.6em}{}}
\renewcommand{\tocsection}[3]{%
  \indentlabel{\@ifnotempty{#2}{\makebox[2.3em][l]{%
    \ignorespaces#1 #2.\hfill}}}#3}
\renewcommand{\tocsubsection}[3]{%
  \indentlabel{\@ifnotempty{#2}{\hspace*{2.3em}\makebox[2.3em][l]{%
    \ignorespaces#1 #2.\hfill}}}#3}
\renewcommand{\tocsubsubsection}[3]{%
  \indentlabel{\@ifnotempty{#2}{\hspace*{4.6em}\makebox[3em][l]{%
    \ignorespaces#1 #2.\hfill}}}#3}
\makeatother

\numberwithin{equation}{section}

\newcounter{hours}\newcounter{minutes}

\newtheorem{thm}{Theorem}[section]
\newtheorem{lem}[thm]{Lemma}
\newtheorem{cor}[thm]{Corollary}
\newtheorem{prop}[thm]{Proposition}

\theoremstyle{remark}                  
\newtheorem{rmk}[thm]{Remark}

\theoremstyle{definition}
\newtheorem{defin}[thm]{Definition}
\newtheorem{ex}[thm]{Example}

\def\Xbar{\bar{X}}
\def\Ybar{\bar{Y}}

\def\diam{\textnormal{diam}}

\def\det{\textnormal{det}}

\def\Id{\textnormal{Id}}

\newcommand{\abs}[1]{\lvert#1\rvert}
\newcommand{\Norm}[2][]{\lVert#2\rVert_{#1}}
\newcommand{\inner}[2]{\langle #1, #2\rangle}

\DeclareMathOperator{\W}{M\MRkern K}
\newcommand{\MRkern}{%
  \mkern-5.8mu
  \mathchoice{}{}{\mkern0.2mu}{\mkern0.5mu}%
}

\DeclareMathOperator{\spt}{spt}

\DeclareMathOperator{\distop}{d}

\newcommand{\dist}[2]{d{\left(#1, #2\right)}}
\newcommand{\hdist}[2]{d_{\mathcal{H}}{\left(#1, #2\right)}}

\def\H{\mathcal{H}}

\DeclareMathOperator{\ch}{conv}

\def\rhobar{\bar{\rho}}
\def\mubar{\bar{\mu}}

\def\xbar{\bar{x}}

\newcommand{\cExp}[2]{\exp^c_{#1}({#2})}

\newcommand\normal[2]{\mathcal{N}_{#1}(#2)}

\def\ybar{\bar{y}}

\newcommand{\subdiff}[2]{\partial {#1}{(#2)}}
\newcommand{\csubdiff}[3][]{\partial^{#1}_c{#2}{(#3)}}
\newcommand{\locsubdiff}[3][]{\partial^{#1}_{\mathrm{loc}}{#2}{(#3)}}
\newcommand{\locdiff}[3][]{\nabla^{#1}_{\mathrm{loc}}{#2}{(#3)}}

\def\S{\mathbb{S}}

\DeclareMathOperator{\proj}{proj}

\newcommand{\dVol}[1][]{d \textnormal{Vol}_{#1}}
\newcommand{\R}{\mathbb{R}}

\def\Ybar{\bar{Y}}
\newcommand{\outrad}[1][]{R_{\Omega_{#1}}}
\newcommand{\inrad}[1][]{r_{\Omega_{#1}}}
\newcommand{\conerad}[1][]{s_{\Omega_{#1}}}

\DeclareMathOperator*{\essinf}{ess\,inf}
\DeclareMathOperator*{\esssup}{ess\,sup}
\begin{document}
\title{Conditions for existence of single valued optimal transport maps on convex boundaries with nontwisted cost}

\author{Seonghyeon Jeong}
\address[Seonghyeon Jeong]{Department of Applied Mathematics, National Sun Yat-Sen University}
\email{jeongs1466@gmail.com} 

\author{Jun Kitagawa}
\address[Jun Kitagawa]{Department of Mathematics, Michigan State University}
\email{kitagawa@math.msu.edu}
\begin{abstract}
We prove that if $\Omega\subset \mathbb{R}^{n+1}$ is a (not necessarily strictly) convex, $C^1$ domain, and $\mu$ and $\bar{\mu}$ are probability measures absolutely continuous with respect to surface measure on $\partial \Omega$, with densities bounded away from zero and infinity, whose $2$-Monge-Kantorovich distance is sufficiently small, then there exists a continuous Monge solution to the optimal transport problem with cost function given by the quadratic distance on the ambient space $\mathbb{R}^{n+1}$. This result is also shown to be sharp, via a counterexample when $\Omega$ is uniformly convex but not $C^1$. Additionally, if $\Omega$ is $C^{1, \alpha}$ regular for some $\alpha$, then the Monge solution is shown to be H\"older continuous.
\end{abstract}

\date{\today}

\date{}
\maketitle
\markboth{S. Jeong and J. Kitagawa}{}


\section{Introduction}
Recall that the \emph{Monge (optimal transport) problem} is, given Borel probability spaces $(\mathcal{X}, \mu)$ and $(\bar{\mathcal{X}}, \mubar)$, and a Borel \emph{cost function} $c: \mathcal{X}\times\bar{\mathcal{X}}\to \R$, to find a minimizer of the \emph{transport cost} 
\begin{align}\label{eqn: Monge prob}
    T\mapsto \int_{\mathcal{X}}c(X, T(X))d\mu(X) 
\end{align}
over the collection of Borel maps $T: \mathcal{X}\to \bar{\mathcal{X}}$ satisfying $T_\sharp \mu=\mubar$; such minimizers are known as \emph{Monge solutions} if they exist.

In this paper, we will be concerned with the case when $\mathcal{X}=\bar{\mathcal{X}}=\partial \Omega$ for some convex body $\Omega\subset \R^{n+1}$, with the cost $c: \R^{n+1}\times \R^{n+1}\to \R$  defined by 
\begin{align*}
    c(X, \Xbar):=\frac{\lvert X-\Xbar\rvert^2}{2}
\end{align*}
where $\lvert \cdot\rvert$ denotes the usual Euclidean norm in $\R^{n+1}$.

By the work \cite{GangboMcCann00} of Gangbo and McCann, it is known that if $\Omega$ is \emph{uniformly convex}, when $\mu$ and $\mubar$ are absolutely continuous with respect to the surface measure on $\partial\Omega$ with densities bounded away from zero and infinity, there is a unique \emph{Kantorovich solution} to the associated optimal transport problem with the above cost $c$. Recall the \emph{Kantorovich problem} is to find a minimizer of 
\begin{align*}
    \gamma\mapsto \int_{\mathcal{X}\times \bar{\mathcal{X}}}cd\gamma
\end{align*}
over
\begin{align*}
    \gamma\in \Pi(\mu, \mubar):=\{\gamma\in \mathcal{P}(\mathcal{X}\times\bar{\mathcal{X}})\mid \text{the left and right marginals of }\gamma\text{ are }\mu \text{ and }\mubar,\text{ respectively}\}.
\end{align*}
Even for such $\mu$ and $\mubar$, a solution of the Monge problem \eqref{eqn: Monge prob} may fail to exist in this case, as the cost function fails to satisfy the so-called Twist condition as in \cite[P.234]{Villani09}.

By the result \cite{KitagawaWarren12} of the second author and Warren, if $\partial \Omega=\S^{n}$ and the \emph{$2$-Monge-Kantorovich distance} defined by
\begin{align*}
    \W_2(\mu,\mubar):=\left(\inf_{\gamma\in \Pi(\mu, \mubar)}\int_{\partial\Omega\times \partial\Omega}cd\gamma\right)^{\frac{1}{2}}
\end{align*}
is sufficiently small, then there is a Monge solution; by using \cite[Theorem 3.8]{GangboMcCann00}, it is possible to construct an explicit example on $\S^n\times \S^n$ where $\W_2(\mu,\mubar)$ is large and there is no Monge solution.

The main goal of this paper is to show when $\Omega$ is a $C^1$, convex body which \emph{may not be strictly convex}, if $\W_2(\mu,\mubar)$ is small enough depending only on $\Omega$, then there exists a Monge solution minimizing \eqref{eqn: Monge prob}. This result is actually \emph{sharp}, as demonstrated in Example~\ref{ex: non C1} below, where we exhibit a uniformly convex body that is not $C^1$, and pairs of measures whose $\W_2$ distances are arbitrarily close to zero, but a Monge solution fails to exist in each case. Additionally, by exploiting techniques from the regularity theory in \cite{GuillenKitagawa17}, we obtain H\"older regularity of the Monge solution if $\Omega$ has a $C^{1, \alpha}$ boundary.

Our main results are as follows. The quantities $\conerad$, $C_{\rho, \Omega}$, $C_{\partial\Omega}$, and $L_{\pi_\Omega}$ are defined in Lemma~\ref{lem: interior cone}, Proposition~\ref{prop: stay away}, \eqref{eqn: geodesic distance bound}, and Lemma~\ref{lem: rad proj push} respectively. 
\begin{thm}\label{thm: monge solution}
    Let $\mu=\rho \dVol[\partial \Omega]$ and $\mubar=\rhobar \dVol[\partial \Omega]$ be probability measures on $\partial \Omega$, with $\rho$ and $\rhobar$ bounded away from zero and infinity and assume $\Omega$ is $C^1$, bounded, and convex. Then if $\rho_0\in (0, \min(\essinf_{\partial\Omega}\rho,\essinf_{\partial\Omega}\rhobar))$ and 
\begin{align}\label{eqn: W2 small}
\W_2(\mu,\mubar)^{\frac{2}{n+2}} <\max\left(\frac{\conerad(\frac{35}{36})}{64C_{\rho_0,\Omega}}, \frac{\conerad(\frac{35}{36})}{16C_{\rho_0,\Omega}C_{\partial\Omega}L_{\pi_\Omega}^2}\right), 
\end{align}
there exists a unique Kantorovich solution $\gamma$ from $\mu$ to $\mubar$ and a $\mu$-a.e. defined map $T: \partial \Omega\to\partial\Omega$ such that $\gamma=(\Id\times T)_\#\mu$.
\end{thm}
\begin{thm}\label{thm: nonstrictly convex body}
Under the same assumptions as Theorem~\ref{thm: monge solution} above, the map $T$ is continuous on $\partial \Omega$.
\end{thm}

\begin{thm}\label{thm: Holder regularity}
Under the same assumptions as Theorem~\ref{thm: monge solution} above, if $\partial\Omega$ is $C^{1, \alpha}$-smooth for some $\alpha\in (0, 1]$, then $T\in C_{loc}^{0, \alpha'}(\partial \Omega; \partial\Omega)$ for some $\alpha'\in (0, 1]$.
\end{thm}
The approach we follow is fundamentally different from that of  \cite{KitagawaWarren12} in the sphere case. The proofs in \cite{KitagawaWarren12} rely on explicit computations along with the uniform convexity of the sphere, and also utilize a continuity method for the Monge-Amp{\`e}re PDE which is possible because the measures have smooth densities. The setting of the current paper is on arbitrary convex domains, with no uniform or strict convexity assumption, and the transported measures may have discontinuous densities.

\subsection{Literature and applications}
The problem of optimal transport involving measures supported on submanifolds is a problem of natural interest, not covered by the classical Brenier-Gangbo-McCann theory of optimal transport (as in \cite{Brenier91, GangboMcCann96, McCann01}), which as noted by Lott in \cite{Lott17} has suffered from an almost criminal lack of attention. Aside from the above work of Gangbo and McCann (\cite{GangboMcCann00}) and the second author with Warren (\cite{KitagawaWarren12}), there are only a handful of theoretical results on the problem, such as a result by McCann and Sosio (\cite{McCannSosio11}) showing that Kantorovich solutions are supported on the union of the graphs of two H\"older continuous maps. There is a discussion of the corresponding problem on boundaries of a domain when the cost function is the ambient distance, instead of squared distance in \cite{DweikSantambrogio19} (see also the discussion of applications below). Additionally, the papers \cite{ChiapporiMcCannPass19, McCannPass20} contain a treatment of certain cases when the submanifolds are of differing dimensions.

Aside from purely mathematical interest, optimal transport between measures on submanifolds appears in a number of important applications; we note a few here and do not claim this is an exhaustive recollection. In \cite{Fry93}, a such an optimal transport problem is proposed as a method of shape classification. It is shown in \cite{GornyRybkaSabra17} that in the plane, the least gradient problem of minimizing the $L^1$ norm of the gradient of functions with a fixed trace on some domain is equivalent to a Beckman problem. This Beckman problem, in turn, is equivalent to an optimal transport problem between measures supported on the boundary of the domain, where the cost function is the Euclidean distance itself, which is analyzed in \cite{DweikSantambrogio19}. Interestingly, from \cite[Proposition 2.6]{DweikSantambrogio19}, one obtains existence of Monge solutions when the cost is the ambient distance under relatively mild conditions. This is a rare case when the problem with cost function given by squared distance tends to exhibit \emph{worse} behavior than just the distance, as one can construct examples in  the squared distance case where the conclusion of \cite[Proposition 2.6]{DweikSantambrogio19} does not hold. In this least gradient problem, there is a marked difference between the case of strictly convex boundary versus non-strictly convex boundary, see for example \cite{RybkaSabra22, Gorny23}. Finally, there has been success in deriving functional and geometric inequalities (such as the sharp Michael-Simons-Sobolev inequality in codimension 2) via optimal transport techniques, involving measures supported on submanifolds, see \cite{Castillon10, BrendleEichmair23, WangKaiHsiang24}.
\subsection{Outline}
The rough plan of the paper is as follows. First we show that the analogue of condition (QQConv) from \cite{GuillenKitagawa15} holds for pairs of points ``on the same side'' (i.e. $(X, Y)$ such that $Y\in \partial\Omega^+(X)$, see \eqref{eqn: same side def}). (QQConv) is an analogue of the Ma-Trudinger-Wang condition introduced in \cite{MaTrudingerWang05}, which is crucial to regularity theory of the Monge-Amp{\`e}re equation in optimal transport. However because the domain $\partial\Omega$ itself only has $C^1$ regularity, the theory in the lower regularity cases of \cite{GuillenKitagawa15} and \cite{GuillenKitagawa17} play a critical role in this paper. With this condition in hand, we show localized versions of the Aleksandrov estimates which are central to the regularity theory of the classical Monge-Amp{\`e}re equation built up by Caffarelli in \cite{Caffarelli90, Caffarelli91, Caffarelli92}; versions of these estimates have also been shown in \cite{FigalliKimMcCann13, Vetois15} under the (weak) Ma-Trudinger-Wang condition, and under the (QQConv) condition in \cite{GuillenKitagawa15, GuillenKitagawa17}. In order to apply these estimates, we show that a distinguished dual potential function exists, with a key ``stay-away estimate'' which shows that we can control the horizontal distance between points that are optimally coupled in terms of the optimal transport cost $\W_2(\mu,\mubar)$. However, this estimate only applies to pairs of points that lie in the same side, hence to exploit the estimate we must first take a sequence of uniformly convex bodies which approximate our body in Hausdorff distance. This procedure is necessary because the theory of \cite{GangboMcCann00} guarantees the existence of optimally coupled points on the same side but only applies to dual potentials in the uniformly convex case; the quantitative control from our stay-away estimate is stable enough to then say something about the limiting (not necessarily uniformly convex) case. Here it is critical that our stay-away estimate relies only on quantities that are $C^1$ in nature, does not rely on $C^2$ or uniformly convex quantities such as curvature bounds, and in particular is stable under Hausdorff convergence. Thus when the optimal transport cost $\W_2(\mu,\mubar)$ is small enough, we can ensure existence of a dual potential with good properties to which we can apply our localized Aleksandrov estimates, which in turn let us show single-valuedness and regularity of the transport map.

 In Section~\ref{section: preliminary} we recall some definitions related to $c$-convex functions and show some geometric properties of convex bodies. In Section~\ref{section: qqconv} we show that the (QQConv) condition holds between certain points and we obtain nice convexity properties of sections of $c$-convex functions. In Section~\ref{section: aleksandrov} we show the aforementioned localized versions of the Aleksandrov estimates. Section~\ref{section: approximation} gives the key stay-away estimate Proposition~\ref{prop: stay away} along with an approximation procedure, we also show that the dual potential constructed in this way possesses some desirable structure. In the final Section~\ref{section: main proof} we tie everything together and present the proof of our main results.

\section{Preliminaries}\label{section: preliminary}
\subsection{Notation and $c$-convexity}
Throughout this paper, \textbf{$\Omega$ will always be a compact, convex body with $C^1$ boundary, which contains the origin in its interior}, except when explicitly stated otherwise (which will only occur in Example~\ref{ex: non C1}). The notation $B^n_r(x)$ will denote the open ball in $\R^n$ of radius $r>0$ centered at the point $x$.  We will also denote by $\nabla^n$ the Euclidean gradient of a function on $\R^n$, and when the boundary of $\Omega$ is at least $C^1$-smooth, $\nabla^\Omega$ will denote the tangential gradient of a function on $\partial \Omega$. The notation $\inner{\cdot}{\cdot}$ and $\abs{\cdot}$ will denote the Euclidean inner product and norm, in what dimension will be clear from context; we will generally use capital letters to denote points in $\R^{n+1}$ and lower case letters to denote points in $\R^n$. When we write $\essinf$ and $\esssup$, these will be the essential infimum and supremum with respect to the $n$-dimensional Hausdorff measure (on which set will be clear from context). Finally for a set $A\subset \R^{n+1}$,  $X\in \R^{n+1}$, and $r>0$ we write
\begin{align*}
    \dist{X}{A}:&=\inf_{Y\in A}\abs{X-Y}.
\end{align*}
Also note since $\Omega$ contains the origin in its interior, there exist constants $0<\inrad\leq \outrad$ such that 
\begin{align}\label{eqn: uniform balls}
    B^{n+1}_{\inrad}(0)\Subset \Omega\Subset B^{n+1}_{\outrad}(0).
\end{align}
For any convex set $\Lambda\subset \R^{n+1}$ and point $X\in \partial \Lambda$ we write 
\begin{align*}
    \normal{\Lambda}{X}:=\{V\in \S^{n}\mid \inner{V}{Y-X}\leq 0,\ \forall Y\in \Lambda\}
\end{align*}
for the set of outward unit normal vectors to $\Lambda$ at $X$. When $\Lambda$ has $C^1$ boundary, then there is a unique outward unit normal vector at every point, and we can view $X\mapsto \mathcal{N}_{\Lambda}(X)$ as a single valued, (uniformly if $\Lambda$ is compact) continuous map. Since $\Omega$ has $C^1$ boundary, for $\Theta\geq 0$ and $X_0\in \partial \Omega$ we can define the (relatively open) sets
\begin{equation}\label{eqn: same side def}
\partial \Omega^{\pm}_\Theta (X_0) = \{ X \in \partial \Omega \mid \pm \inner{\mathcal{N}_{\Omega}(X)}{\mathcal{N}_{\Omega}(X_0)}> \Theta\},
\end{equation}
if $\Theta=0$ the subscript will be omitted. 
We also denote by $\Pi^\Omega_X$ the plane tangent to $\Omega$ at $X\in\partial\Omega$, and by $\proj^\Omega_X$ orthogonal projection onto $\Pi^\Omega_X$. Also for $X\in\partial\Omega$, we define \begin{align*}
    \Omega_{X}:&=\proj^\Omega_X\left(\partial \Omega^+(X)\right).
\end{align*}
Intuitively, $\partial\Omega^+(X)$ is the collection of points ``on the same side of $\partial \Omega$ as $X$'', as measured via the angle between outward normal vectors, while $\partial\Omega^+_\Theta(X)$ is a quantitative shrinking of $\partial\Omega^+(X)$. For the purposes of our results, it will be convenient to write $\partial\Omega^+(X)$ as the graph of some function over the projected $\Omega_X$ in the tangent plane to $\partial\Omega$ at $X$. We give an explicit function $\beta_X$ and prove some properties of the function along with $\Omega_X$ in the next lemma.
\begin{lem}\label{lem: basic properties}
    The set $\Omega_X$ is open, bounded, and convex. Moreover, defining $\beta_X: \Omega_X\to \R$ by 
    \begin{align*}
    \beta_X(x):=-\dist{(\proj_X^\Omega)^{-1}(x)\cap \partial\Omega^+(X)}{\Pi_X^\Omega},
\end{align*}
we can write $\partial\Omega^+(X)$ (after an appropriate rotation) as the graph of $\beta_X$ over $\Omega_X$. Additionally, $\beta_X\in C^1(\Omega_X)$ and is a concave function.
\end{lem}
\begin{proof}
    $\Omega_X$ is bounded from the compactness of $\Omega$. 
    
    For openness, suppose by contradiction that $y\in \Omega_X$ (hence $\proj_X^\Omega(Y)=y$ for some $Y\in \partial\Omega^+(X)$) but there exists a sequence $\{y_k\}_{k=1}^\infty\subset \Pi^\Omega_X\setminus\Omega_X$ converging to $y$. First suppose there exists a subsequence of $\{y_k\}_{k=1}^\infty$ (not relabeled) such that $(\proj_X^\Omega)^{-1}(y_k)\cap \Omega^\circ=\emptyset$ for all $k$, then taking a limit we see that $(\proj_X^\Omega)^{-1}(y)\cap \Omega^\circ=\emptyset$. Since $\Omega^\circ$ is open and convex, by the separation theorem~\cite[Theorem 11.2]{Rockafellar70} there is a hyperplane which contains $(\proj_X^\Omega)^{-1}(y)$ and is supporting to $\Omega$. In particular it contains $Y$, hence the hyperplane is normal to $\mathcal{N}_\Omega(Y)$. However this implies $\inner{\mathcal{N}_\Omega(Y)}{\mathcal{N}_\Omega(X)}=0$, a contradiction. 
    This implies there must be a subsequence (again not relabeled) such that  $ (\proj_X^\Omega)^{-1}(y_k)\cap\Omega^\circ\neq \emptyset$ but $(\proj_X^\Omega)^{-1}(y_k)\cap \partial\Omega^+(X)=\emptyset$ for all $k$. By compactness of $\Omega$, there exist points $Y^{\max}_k$ and $Y^{\min}_k$, necessarily in $\partial\Omega$, which achieve the maximum and minimum values, respectively, of the function $\inner{\cdot}{\mathcal{N}_\Omega(X)}$ over the set $(\proj_X^\Omega)^{-1}(y_k)\cap\Omega$; by assumption $Y^{\max}_k$, $Y^{\min}_k\not\in \partial\Omega^+(X)$. We may pass to another subsequence and assume $Y^{\max}_k$ and $Y^{\min}_k$ converge to some $Y^{\max}$ and $Y^{\min}\in \partial \Omega$ respectively, by continuity of $\mathcal{N}_\Omega$ we have $\inner{\mathcal{N}_\Omega(X)}{\mathcal{N}_\Omega(Y^{\max})}$, $\inner{\mathcal{N}_\Omega(X)}{\mathcal{N}_\Omega(Y^{\min})}\leq 0$ while $Y^{\max}=Y^{\min}+\lambda\mathcal{N}_\Omega(X)$ for some $\lambda\in \R$. We can then calculate
    \begin{align*}
        0&\geq \inner{Y^{\max}-Y^{\min}}{\mathcal{N}_\Omega(Y^{\min})}
        =\lambda\inner{\mathcal{N}_\Omega(X)}{\mathcal{N}_\Omega(Y^{\min})},\\
        0&\geq \inner{Y^{\min}-Y^{\max}}{\mathcal{N}_\Omega(Y^{\max})}
        =-\lambda\inner{\mathcal{N}_\Omega(X)}{\mathcal{N}_\Omega(Y^{\max})}.
    \end{align*}
    If $\inner{\mathcal{N}_\Omega(X)}{\mathcal{N}_\Omega(Y^{\max})}=0$, the hyperplane through $Y^{\max}$ normal to $\mathcal{N}_\Omega(Y^{\max})$ supporting to $\Omega$ contains all of $(\proj_X^\Omega)^{-1}(y)$, hence $Y$, which implies $\mathcal{N}_\Omega(Y)=\mathcal{N}_\Omega(Y^{\max})$, however this contradicts that $Y\in \partial\Omega^+(X)$. Similarly we find $\inner{\mathcal{N}_\Omega(X)}{\mathcal{N}_\Omega(Y^{\max})}\neq 0$, which implies $\lambda=0$, hence $Y^{\max}=Y^{\min}=Y$, contradicting that $Y\in \partial\Omega^+(X)$. Thus $\Omega_X$ must be open.
    
   To see the convexity, let $x_0:=\proj_X^\Omega(X_0)$, $x_1:=\proj_X^\Omega(X_1)\in \Omega_X$ with $X_0$, $X_1\in \partial\Omega^+(X)$, and for $t\in (0, 1)$ define $x_t:=(1-t)x_0+tx_1$; note $x_t=\proj_X^\Omega(X_t)$ where $X_t:=(1-t)X_0+tX_1$. By convexity of $\Omega$ we have $X_t\in \Omega$; if $X_t\in \partial\Omega$ the segment between $X_0$ and $X_1$ lies in $\partial \Omega$, using a separation theorem again as above we can conclude the supporting hyperplane to $\Omega$ at $X_t$ also contains $X_0$, hence $\normal{\Omega}{X_t}=\normal{\Omega}{X_0}$. If $X_t$ is in the interior of $\Omega$, there is $\lambda>0$ such that $\hat{X}_t:=X_t+\lambda\normal{\Omega}{X}\in \partial \Omega$, and $r>0$ such that $X_t+r\normal{\Omega}{\hat{X}_t}\in \Omega$. Then 
\begin{align*}
0\geq \inner{X_t+r\normal{\Omega}{\hat{X}_t}-\hat{X}_t}{\normal{\Omega}{\hat{X}_t}}=\inner{r\normal{\Omega}{\hat{X}_t}-\lambda\normal{\Omega}{X}}{\normal{\Omega}{\hat{X}_t}}   
\end{align*}
which after rearranging yields $\inner{\normal{\Omega}{X}}{\normal{\Omega}{\hat{X}_t}}\geq r/\lambda>0$. Thus in both cases, $x_t\in \Omega_X$, proving convexity.

Next, note that if $Y_1$, $Y_2\in \partial\Omega^+(X)$ with $Y_1=Y_2+\lambda \normal{\Omega}{X}$ for some $\lambda>0$, we would have
\begin{align*}
    0\geq \inner{ Y_1-Y_2}{\normal{\Omega}{Y_2}}=\lambda \inner{\normal{\Omega}{X}}{\normal{\Omega}{Y_2}},
\end{align*}
which is a contradiction, meaning that $\proj_X^\Omega$ is injective on $\partial\Omega^+(X)$. Thus if we make a rotation so that $\normal{\Omega}{X}=e_{n+1}$ and a translation to identify $\Pi_X^\Omega$ with $\R^n$, we find that $\partial\Omega^+(X)$ is the graph of $\beta_X$ over $\Omega_X$. The $C^1$ regularity of $\partial\Omega$ immediately implies $\beta_X\in C^1(\Omega_X)$, while convexity of $\Omega$ (by the way we orient the outward unit normal) implies $\beta_X$ is a concave function of $n$ variables over $\Omega_X$.
\end{proof}

Also following \cite{GangboMcCann00}, we adopt the following terminology. Note that both notions of convexity below do not require any differentiability of $\partial\Omega$.
\begin{defin}
    A convex set $\Omega$ is said to be \emph{strictly convex} if the intersection of any line segment and $\partial\Omega$ is at most two points. $\Omega$ is \emph{uniformly convex} if there exists $R>0$ such that $\Omega\subset B^{n+1}_R(X-R\normal{\Omega}{X})$ for any $X\in \partial \Omega$ and any outward unit normal vector $\normal{\Omega}{X}$ at $X$.
\end{defin}

Next we recall some notions central to the theory of optimal transport. We comment here, although the cost function $c$ is symmetric in its variables, throughout the paper the reader should think of variables with a bar above as belonging to the ``target domain'' while variables without belong to the ``source domain.''
\begin{defin}\label{def: c-convex}
    For a function $u: B^{n+1}_{\outrad}(0)\to \R\cup \{\infty\}$, its \emph{$c$-transform} is the function $u^c: B^{n+1}_{\outrad}(0)\to \R\cup\{\infty\}$ defined by 
    \begin{align*}
        u^c(\Xbar):=\sup_{X\in B^{n+1}_{\outrad}(0)}(-c(X, \Xbar)-u(X)).
    \end{align*}
    If $u=v^c\not\equiv \infty$ for some function $v: B^{n+1}_{\outrad}(0)\to \R\cup \{\infty\}$, we say that $u$ is a \emph{$c$-convex function}.
\end{defin}
Since $c$ is symmetric in its variables, we do not make a distinction between the $c$-transform and the so-called $c^*$-transform as is done customarily.

\begin{defin}\label{def: local subdiff}
    Suppose $\Omega\subset B^{n+1}_{\outrad}(0)$. 
    If $u: B^{n+1}_{\outrad}(0)\to \R\cup \{\infty\}$ is $c$-convex, then the \emph{$c$-subdifferential (with respect to $\Omega$)} of $u$ at $X\in \partial \Omega$ is defined by
    \begin{align*}
        \csubdiff[\Omega]{u}{X} :=\{\Xbar\in\partial\Omega\mid u(Y)\geq -c(Y, \Xbar)+c(X, \Xbar)+u(X),\ \forall Y\in \partial \Omega\}.
    \end{align*}    
    If $u: B^{n+1}_{\outrad}(0)\to \R\cup \{\infty\}$ is a convex function, then its \emph{subdifferential} at $X\in B^{n+1}_{\outrad}(0)$ is defined by
    \begin{align*}
        \subdiff{u}{X}:=\{P\in \R^{n+1}\mid u(Y)\geq u(X)+\inner{Y-X}{P},\ \forall Y\in B^{n+1}_{\outrad}(0)\}.
    \end{align*}
    We also write 
    \begin{align*}
        \partial_c^\Omega u:&=\{(X, \Xbar)\in \partial \Omega\times \partial \Omega\mid X\in \partial \Omega,\ \Xbar\in \csubdiff[\Omega]{u}{X}\}.
    \end{align*}
\end{defin}

\begin{rmk}\label{rmk: subdifferentials}
Note that if $u=v^c$ is $c$-convex, then
\begin{align*}
    \tilde u(X):=u(X)+\frac{\lvert X\rvert^2}{2}=\sup_{Y\in B^{n+1}_{\outrad}(0)}\left(\inner{X}{Y}-\frac{\lvert Y\rvert^2}{2}-v(Y)\right)
\end{align*}
is seen to be a $\R$-valued, convex function, with
\begin{align*}
   \csubdiff[\Omega]{u}{X}=\partial\Omega\cap \subdiff{\tilde u}{X}
\end{align*}
for $X\in \partial \Omega$. In particular, this shows that a $c$-convex function is Lipschitz. Also (where we abusively identify a tangent vector to the  embedded manifold $\partial \Omega$  with a vector in the ambient space $\R^{n+1}$)
\begin{align*}
    -\nabla^\Omega_Xc(X, \Xbar)=\proj^\Omega_X(\Xbar-X),
\end{align*}
 thus if $\Xbar\in \csubdiff[\Omega]{u}{X}$ and $u$ is tangentially differentiable at $X$ it is easy to see that $\proj^\Omega_X(\Xbar-X)=\nabla^\Omega u(X)$. 
\end{rmk}
\begin{rmk}\label{rmk: kantorovich duality}
By the classical Kantorovich duality (\cite[Theorem 5.10 (iii)]{Villani09}), in our current setting
\begin{align}\label{eqn: kantorovich dual}
    \W_2(\mu, \mubar)^2=\max\left\{-\int_{\partial\Omega} ud\mu-\int_{\partial\Omega}u^cd\mubar\mid u\text{ is }c\text{-convex}\right\}.
\end{align}
\end{rmk}

\subsection{Uniform estimates on convexity}
In Section~\ref{section: approximation}, we will approximate the convex body $\Omega$ by a sequence of uniformly convex bodies in Hausdorff distance. As such, we require control on some quantities which will be uniform along this sequence.

By \cite[Lemma A.1]{GangboMcCann00}, for any point in the boundary of a convex body, there is a sufficiently small ball around the point whose intersection with the body is the epigraph of some convex function over a supporting tangent plane of the boundary. Here, we show that if the boundary is $C^1$, then the radius can be taken to be independent of the center point. Note this radius cannot be chosen uniformly if $\partial \Omega$ is not $C^1$, as can be seen for the domain $\Omega$ in Example~\ref{ex: non C1} below: indeed as $X$ approaches the singular set, the radius $R$ must shrink to $0$.
\begin{defin}
    For a $C^1$ convex body $\Omega$, we denote by $\omega_\Omega: [0, \infty)\to [0, \infty)$ the \emph{modulus of continuity of $\mathcal{N}_{\Omega}$}, that is
    \begin{align*}
        \omega_\Omega(r):=\sup_{\abs{X_1-X_2}<r}\abs{\normal{\Omega}{X_1}-\normal{\Omega}{X_2}}.
    \end{align*}
\end{defin}

\begin{rmk}
Since $\Omega$ is assumed to be compact, if it is $C^1$ then $\mathcal{N}_\Omega$ is a uniformly continuous map from $\partial\Omega$ to $\S^n$ (with respect to the relative topology on $\partial\Omega$). Thus $\omega_\Omega(r)\to 0$ as $r\searrow 0$.
\end{rmk}
\begin{lem}
\label{lem: interior cone}
Suppose $\Omega$ is a convex body with $C^1$ boundary, $\Theta\in [0, 1)$ and $\conerad(\Theta)>0$ is such that 
\begin{align}\label{eqn: modulus < 1}
    \omega_\Omega(\conerad(\Theta))<\sqrt{2-2\Theta}.
\end{align}
Then for any $X_0 \in \partial \Omega$, 
\begin{align}\label{eqn: positive inner product in ball}
     B^{n+1}_{\conerad(\Theta)}(X_0) \cap \partial\Omega\subset \partial\Omega^+_\Theta(X_0),
\end{align}
and 
\begin{equation}\label{eqn: cone inclusion holds}
\mathcal{K}_{\Omega, \Theta}(X_0) := \left\{ X\in\R^{n+1} \mid \inner{X- X_0}{-\mathcal{N}_{\Omega}(X_0)} \geq \sqrt{1-\Theta^2} |X-X_0| \right\} \cap B^{n+1}_{\conerad(\Theta)}(X_0) \subset \Omega.
\end{equation}
\end{lem}
\begin{proof}
Since $\Theta$ is fixed, let us write $\conerad$ for $\conerad(\Theta)$. By the choice of $\conerad$, for any $X_0\in \partial \Omega$ and $X\in B^{n+1}_{\conerad}(X_0)\cap \partial\Omega$, we see $\inner{\normal{\Omega}{X}}{\normal{\Omega}{X_0}}=1-\frac{\abs{\normal{\Omega}{X}-\normal{\Omega}{X_0}}^2}{2}> \Theta$, which proves \eqref{eqn: positive inner product in ball}.

Since $B^{n+1}_{\conerad}(X_0) \cap \Omega$ is convex with the origin in the interior, using \cite[Theorem 2.2.6]{Schneider14},
\begin{align*}
\Omega&\supset B^{n+1}_{\conerad}(X_0) \cap \Omega  
=\bigcap_{Y\in \partial\Omega\cap  B^{n+1}_{\conerad}(X_0)^{\circ}}\{ X \mid \inner{X - Y}{\normal{\Omega}{Y}} \leq 0\}\\
&\cap \bigcap_{Y\in \partial B^{n+1}_{\conerad}(X_0) \cap \Omega^{\circ}}\{ X \mid \inner{X - Y}{\normal{B^{n+1}_{\conerad}(X_0)}{Y}} \leq 0\}\cap B^{n+1}_{\conerad}(X_0)\\
&\supset \bigcap_{Y\in \partial\Omega\cap  B^{n+1}_{\conerad}(X_0)}\{ X \mid \inner{X - Y}{\normal{\Omega}{Y}} \leq 0\}\cap B^{n+1}_{\conerad}(X_0)\\
& \supset \bigcap_{Y\in \partial\Omega\cap  B^{n+1}_{\conerad}(X_0)}\{ X \mid \inner{X - X_0}{\normal{\Omega}{Y}} \leq 0\}\cap B^{n+1}_{\conerad}(X_0).
\end{align*}
To obtain the second line above, by \cite[Theorem 2.2.1 (b)]{Schneider14} we have $\R\normal{\Omega\cap B^{n+1}_{\conerad}(X_0)}{Y}=\R\normal{\Omega}{Y}+\R\normal{B^{n+1}_{\conerad}(X_0)}{Y}$ for $Y\in \partial(\Omega\cap B^{n+1}_{\conerad}(X_0))$. Then simply note that $\max(\inner{X-Y}{\normal{\Omega}{Y}}, \inner{X-Y}{\normal{B^{n+1}_{\conerad}(X_0)}{Y}})\leq 0$ implies that $\inner{X-Y}{(1-t)\normal{\Omega}{Y}+t\normal{B^{n+1}_{\conerad}(X_0)}{Y}}\leq 0$, hence 
\begin{align*}
    &\{ X \mid \inner{X - Y}{\normal{\Omega}{Y}} \leq 0\}\cap \{ X \mid \inner{X - Y}{\normal{B^{n+1}_{\conerad}(X_0)}{Y}} \leq 0\}\\
    &= \bigcap_{N\in \normal{\Omega\cap B^{n+1}_{\conerad}(X_0)}{Y}}\{ X \mid \inner{X - Y}{N} \leq 0\}.
\end{align*}
Now using \eqref{eqn: modulus < 1},
\begin{align*}
&\bigcap_{Y\in \partial\Omega\cap  B^{n+1}_{\conerad}(X_0)}\{ X \mid \inner{X - X_0}{\normal{\Omega}{Y}} \leq 0\}\cap B^{n+1}_{\conerad}(X_0)\notag\\
& \supset \bigcap_{\{V \in \mathbb{S}^n\mid  \abs{V-\mathcal{N}_{\Omega}(X_0)} \leq \sqrt{2-2\Theta}\}}\{ X \mid \inner{X-X_0}{V} \leq 0\} \cap B^{n+1}_{\conerad}(X_0) \notag\\
& = \left\{ X \mid \inner{X-X_0}{\mathcal{N}_{\Omega}(X_0)}\leq - \sqrt{1-\Theta^2}\abs{X-X_0} \right\} \cap B^{n+1}_{\conerad}(X_0),
\end{align*}
proving \eqref{eqn: cone inclusion holds}.
\end{proof}

\begin{rmk}
\label{rmk: uni-sized nbhd in proj}
    Lemma~\ref{lem: interior cone} implies, for example, that
    \begin{equation*}
        B^n_{\frac{\conerad(\frac{1}{2})}{2}}(\proj_{X_0}^\Omega (X_0))\subset\proj_{X_0}^\Omega\left(\mathcal{K}_{\Omega, \frac{1}{2}}(X_0)\right) \subset \Omega_{X_0},
    \end{equation*}
    indeed, the second inclusion follows directly by projection of the inclusion~\eqref{eqn: cone inclusion holds} onto the tangent plane of $\partial\Omega$ at $X_0$. To obtain the first inclusion, first make a rotation and translation to assume $\mathcal{N}_{\Omega}(X_0)=e_{n+1}$, $X_0=0$, then for any $(x, 0)\in B^n_{\frac{\conerad(\frac{1}{2})}{2}}(\proj_{X_0}^\Omega (X_0))$ we can calculate that $(x, -\sqrt{3}\lvert x\rvert)\in \mathcal{K}_{\Omega, \frac{1}{2}}(X_0)$. In particular,
     $\Omega_{X_0}$ always contains a uniform radius ball around $\proj_{X_0}^\Omega (X_0)$. Note that this is not true without $C^1$-ness of $\partial\Omega$.
    
     Note also that by the definition of $\conerad$, we may assume it is decreasing as a function of $\Theta$.
\end{rmk}
Next we show two Lipschitz estimates on $\beta_X$. The first shows that $\beta_X$ has a Lipschitz estimate as long as one of the points involved is close to the ``center'' of $\Omega_X$ (i.e. to $\proj^\Omega_X(X)$).
\begin{lem}
\label{lem: unif size nbhd lip beta}
Suppose $\Omega$ is a convex body with $C^1$ boundary and fix $\Theta\in [\frac{1}{2}, 1)$. Also fix $X_0 \in \partial \Omega$, choose coordinates such that $\normal{\Omega}{X_0}=e_{n+1}$, and write $x_0:=\proj_{X_0}^\Omega(X_0)$. Let $\beta_X: \Omega_{X_0} \to \R$ be the concave function whose graph is $\partial \Omega^+(X_0)$. Then, for any $x_1 \in B^n_{\frac{\conerad(\Theta)}{4}}(x_0)$, we have
\begin{equation}
\label{eqn: Lip eqn beta}
\lvert \beta_{X_0} (x_1) - \beta_{X_0} (x_2) \rvert \leq \frac{4\diam(\Omega)}{\conerad(\Theta)}\vert x_1  -x_2\vert =:L_\Theta \abs{x_1-x_2}
\end{equation}
for any $x_2 \in \Omega_{X_0}$, where $\conerad$ satisfies \eqref{eqn: modulus < 1}.
\end{lem}
\begin{proof}
Let $x_1\in B^n_{\frac{\conerad(\Theta)}{4}}(x_0)$, we first note that by Remark~\ref{rmk: uni-sized nbhd in proj}, we have $B^n_{\frac{\conerad(\Theta)}{4}}(x_0) \subset \Omega_{X_0}$ so that $\beta_{X_0} (x_1)$ is well-defined. Now we claim that for all $x_2\in \Omega_{X_0}$,
\begin{equation} \label{eqn: Lip cone of beta}
\beta_{X_0} (x_1) - \beta_{X_0} (x_2) \leq L_\Theta\vert x_1-x_2\vert.
\end{equation}
 Indeed, otherwise we obtain
\begin{equation*}
\beta_{X_0} (x_1)-\beta_{X_0} (x_2) > \frac{4\diam(\Omega)}{\conerad(\Theta)}\vert x_1 -x_2 \vert
\end{equation*}
for some $x_2 \in \Omega_{X_0}$. If $\vert x_1 - x_2 \vert > \frac{\conerad(\Theta)}{4}$, then $\beta_{X_0} (x_1) - \beta_{X_0} (x_2) > \diam(\Omega)$ which is impossible since $(x_1, \beta_{X_0} (x_1))$, $(x_2, \beta_{X_0} (x_2))\in\partial \Omega$. If $\vert x_1-x_2\rvert \leq \frac{\conerad(\Theta)}{4}$, then there exists a point $x_3 \in \partial B^n_{\frac{\conerad(\Theta)}{4}}(x_1) \subset \Omega_{X_0}$ such that $x_2$ lies on the line segment from $x_1$ to $x_3$. Then by concavity of $\beta_{X_0} $, we obtain 
\begin{align*}
\beta_{X_0} (x_3) & \leq \frac{\beta_{X_0} (x_2)-\frac{\lvert x_3-x_2\rvert}{\lvert x_3-x_1\rvert}\beta_{X_0} (x_1)}{\frac{\lvert x_2-x_1\rvert}{\lvert x_3-x_1\rvert}}\\
&< \frac{-\frac{4\diam(\Omega)}{\conerad(\Theta)}\lvert x_2 -x_1\rvert+\left(1-\frac{\lvert x_3-x_2\rvert}{\lvert x_3-x_1\rvert}\right)\beta_{X_0} (x_1)}{\frac{\lvert x_2-x_1\rvert}{\lvert x_3-x_1\rvert}} \\
& =  \beta_{X_0} (x_1)-\diam(\Omega),
\end{align*}
again a contradiction, thus proving \eqref{eqn: Lip cone of beta}.

Next, we use concavity of $\beta_{X_0} $ to obtain \eqref{eqn: Lip eqn beta}. Since $\beta_{X_0}$ is concave, there exists an affine function $l_{x_1}: \Omega_{X_0} \to \R$ supporting $\beta_{X_0}$ from above at $x_1$; we can write $l_{x_1}(x) = \inner{p}{x-x_1}+\beta_{X_0}(x_1)$ for some $p\in \R^n$. From \eqref{eqn: Lip cone of beta}, we obtain $l_{x_1}(x_2) \geq \beta_{X_0} (x_2) \geq \beta_{X_0} (x_1) - L_\Theta\vert x_1 - x_2 \vert$ for any $x_2\in \Omega_{X_0}$, therefore we have for $x_2\neq x_1$ that
\begin{equation}
\label{eqn: supp hyperplane bounds cone}
L_\Theta  \geq \frac{\beta_{X_0} (x_1) - l_{x_1}(x_2)}{\abs{x_1 - x_2}} = \inner{p}{\frac{x_1-x_2}{\abs{x_1 - x_2}}}.
\end{equation}
Since $x_1 \in B^n_{\frac{\conerad(\Theta)}{4}}(x_0)\subset \Omega_{X_0}$, we can choose $x_2=x_1-rp\in \Omega_{X_0}$ for some small $r>0$ in \eqref{eqn: supp hyperplane bounds cone} to find $\vert p \vert \leq L_\Theta$, and then we obtain
\begin{equation*}
\beta_{X_0}(x_2)-\beta_{X_0}(x_1) \leq l_{x_1}(x_2) - l_{x_1}(x_1) = \inner{p}{x_2 - x_1} \leq L_\Theta \vert x_2 - x_1 \vert.
\end{equation*}
We combine this with \eqref{eqn: Lip cone of beta} to obtain \eqref{eqn: Lip eqn beta}.
\end{proof}

The second estimate shows that $\nabla^n\beta_X$ has small norm close to $\proj^\Omega_X(X)$ in a quantifiable way (which is essentially a Lipschitz bound on small balls around $\proj^\Omega_X(X)$).
\begin{lem}\label{lem: gradient beta est}
   Fix $\Theta\in (0, 1)$ and suppose $X\in \partial \Omega$ and $Y \in \partial \Omega^+_{\Theta}(X)$. Then we have 
    \begin{equation*}
        \lvert \nabla^n \beta_X ( \proj^\Omega_X(Y) ) \rvert < \frac{\sqrt{1-\Theta^2}}{\Theta}.
    \end{equation*}
\end{lem}
\begin{proof}
    We take coordinates such that $\normal{\Omega}{X} = e_{n+1}$ and write $\partial \Omega^+ (X)$ as the graph of the $C^1$, concave function $\beta_X$. Then we obtain
    \begin{equation*}
        \Theta< \inner{\normal{\Omega}{X}}{\normal{\Omega}{Y}}= \inner{e_{n+1}}{ \frac{(-\nabla \beta_X( \proj^\Omega_X(Y) ) , 1)}{\sqrt{1+ \lvert \nabla^n \beta_X( \proj^\Omega_X(Y) ) \rvert^2}}} = \frac{1}{\sqrt{1+\lvert \nabla^n \beta_X (\proj^\Omega_X(Y)) \rvert^2}},
    \end{equation*}
and rearranging the terms yields the lemma.
\end{proof}
We conclude this section by presenting a counterexample to existence of a Monge solution when $\Omega$ is not $C^1$ (but is actually \emph{uniformly} convex), which shows Theorem~\ref{thm: nonstrictly convex body} is sharp. The following counterexample is based on similar considerations to the example~\cite[Example 3.13]{GangboMcCann00} between two isosceles triangles.
\begin{ex}[Counterexample for non-$C^1$ boundary]\label{ex: non C1}
Let $$\Omega:=B^{n+1}_{R+1}(-Re_{n+1})\cap B^{n+1}_{R+1}(Re_{n+1})$$ (which is clearly uniformly convex), for some $R>0$ sufficiently large, which forms a shape like a convex lens. Letting $\partial\Omega^\pm:=\{X\in\partial \Omega\mid \pm\inner{X}{e_{n+1}}\geq 0\}$, define the measure $\mu\ll\H^{n-1}_{\partial \Omega}$ with constant density $C-\delta$ on $\partial \Omega^+$ and $\delta$ on $\partial \Omega^-$, where $\delta>0$ is small and $C$ is normalized so $\mu$ is a probability measure. Now suppose $\mubar\ll\H^{n-1}_{\partial \Omega}$ has density bounded away from zero and infinity and $\gamma$ is the (unique by \cite[Theorem 2.6]{GangboMcCann00}) Kantorovich solution from $\mu$ to $\mubar$. Then we claim that when $R$ is sufficiently large, $\gamma=(\Id\times T)_\#\mu$ for some map $T: \partial \Omega\to \partial \Omega$ that is single-valued $\mu$-a.e. \emph{only} if $T(\partial \Omega^+)\subset \partial \Omega^+$ up to a set of $\mu$-measure zero.

To see this, suppose there is such a $\mu$-a.e.  single-valued map $T$, and let $X\in \partial \Omega^+$ be a point where both $\mathcal{N}_{\Omega}(X)$ and $T(X)$ are single-valued; the set of such $X$ is of full $\mu$-measure in $\partial \Omega^+$. Now if $R>0$ is large enough, then whenever $X^+\in \partial \Omega^+$ and $X^-\in \partial \Omega^-\setminus \partial \Omega^+$, and $q^\pm$ are outward normal vectors to $\Omega$ at $X^\pm$ respectively, we must have
\begin{align}\label{eqn: opposite side normals}
    \inner{q^+}{q^-}<0.
\end{align}
Fix such an $R>0$. Writing $\normal{\Omega}{X}$ for the unit outward normal to $\Omega$ at a point $X$ when it is unique, if $T(X)\not \in \partial \Omega^+$ we see $\mathcal{N}_{\Omega}(T(X))$ is single-valued and $\inner{\mathcal{N}_{\Omega}(X)}{\mathcal{N}_{\Omega}(T(X))}<0$.
Then by \cite[Proposition 3.4, Definition 3.5, Theorem 3.8]{GangboMcCann00}, we see there must be a point $T^+(X)\in \partial \Omega$ such that $T^+(X)-T(X)$ is a nonnegative multiple of $\mathcal{N}_{\Omega}(X)$, and $(X, T^+(X))\in \spt\gamma$, while $\inner{\mathcal{N}_{\Omega}(X)}{q}\geq 0$ for some vector $q$ which is an outward normal vector to $\Omega$ at $T^+(X)$. Thus by \eqref{eqn: opposite side normals}, we see $T^+(X)$ must be a point in $\partial \Omega^+$, in particular $T^+(X)\neq T(X)$. This would contradict that $\gamma$ is induced by a $\mu$-a.e. single valued map, unless $T^+(X)\in \partial \Omega^+$ for $\mu$-a.e. $X\in \partial \Omega^+$.

Now, define $\mubar_k$ to have density $C-(1+\frac{1}{k})\delta$ on $\partial \Omega^+$ and density $(1+\frac{1}{k})\delta$ on $\partial \Omega^-$. The proof of \cite[Lemma 5.4]{KitagawaWarren12} can be easily adapted to general convex bodies $\Omega$ instead of the sphere, hence if $\rho$ and $\rhobar_k$ are the densities of $\mu$ and $\mubar_k$ respectively, we will have
\begin{align*}
    0<\W^2_2(\mu, \mubar_k)&\leq C\sqrt{\Norm[L^\infty(\partial \Omega)]{\rho-\rhobar_k}}\leq \frac{C\sqrt{\delta}}{\sqrt{k}}\to 0
\end{align*}
as $k\to\infty$. However, since $\mubar_k(\partial \Omega^+)<\mu(\partial \Omega^+)$, no single-valued map $T$ pushing $\mu$ forward to $\mubar_k$ can satisfy $T(\partial \Omega^+)\subset \partial \Omega^+$, in particular no Kantorovich solution between $\mu$ and $\mubar_k$ for any $k$ can be induced by a single-valued map, for large $R$.
\end{ex}
\section{Localized properties for dual potential functions}\label{section: qqconv}
Recall we assume $\Omega$ is a $C^1$, compact, convex body containing the origin in its interior. In particular,  $\mathcal{N}_\Omega$ is uniformly continuous. Also  $\conerad$ will be a function satisfying \eqref{eqn: modulus < 1}. 

We will now prove some convexity properties of sections and $c$-subdifferentials associated to $c$-convex functions, when they are sufficiently localized.
\begin{defin}
    For any $X_0$, $\Xbar_0\in \partial\Omega$, $h>0$, and $c$-convex function $u$, we define the section
    \begin{align*}
    S^{u}_{h, X_0, \Xbar_0}:=\{X\in \partial\Omega\mid u(X)\leq -c(X, \Xbar_0)+c(X_0, \Xbar_0)+u(X_0)+h\}.
    \end{align*}
\end{defin}
\begin{rmk}\label{rmk: contact set is dual c-sub}
    It is well known that if $u$ is a $c$-convex function and $X_0\in \csubdiff[\Omega]{u^c}{\Xbar_0}$, then
    \begin{align*}
        \{\Xbar\in \partial \Omega\mid u^c(\Xbar)\leq -c(X_0, \Xbar)+c(X_0, \Xbar_0)+u^c(\Xbar_0)\}=\csubdiff[\Omega]{u}{X_0},
    \end{align*}
    i.e. sections of height $0$ for $u^c$ correspond to the $c$-subdifferentials of $u$. Indeed, 
    \begin{align*}
        &u^c(\Xbar)\leq -c(X_0, \Xbar)+c(X_0, \Xbar_0)+u^c(\Xbar_0)\\
        &\iff -u(X_0)-u^c(\Xbar)=-u(X_0)+c(X_0, \Xbar)-c(X_0, \Xbar_0)-u^c(\Xbar_0)=c(X_0, \Xbar)\\
        &\iff \Xbar\in \csubdiff[\Omega]{u}{X_0}
    \end{align*}
    where \cite[Proposition 2.6.4]{FigalliGlaudo21} is used in both the second and third lines.
\end{rmk}
\begin{defin}
    For $X_0 \in \partial \Omega$ and $\Xbar_0$, $\Xbar_1 \in \partial \Omega^+(X_0)$, the \emph{$c$-segment from $\Xbar_0$ to $\Xbar_1$ with respect to $X_0$} is the curve $[0, 1]\ni t\mapsto \Xbar_t\in \partial \Omega$ defined by 
    \begin{align*}
        \Xbar_t:=(1-t)\proj_{X_0}^\Omega(\Xbar_0)+t\proj_{X_0}^\Omega(\Xbar_1)+ \beta_{X_0}((1-t)\proj_{X_0}^\Omega(\Xbar_0)+t\proj_{X_0}^\Omega(\Xbar_1))\normal{\Omega}{X_0}.
    \end{align*}
    By convexity of $\Omega_{X_0}=\proj^\Omega_{X_0}\left(\partial \Omega^+(X_0)\right)$, $\Xbar_t$ is well-defined for all $t\in [0, 1]$.
\end{defin}
Note that $-\nabla_X^\Omega c(X_0, \Xbar_t)=(1-t)\proj_{X_0}^\Omega(\Xbar_0)+t\proj_{X_0}^\Omega(\Xbar_1)$, which is why we have opted to use the name ``$c$-segment'' despite the fact that $c$ does not satisfy the Twist condition (see \cite[P.234]{Villani09}). For our cost function we can show a version of the condition (QQConv) first introduced in \cite{GuillenKitagawa15}.
\begin{prop}\label{prop: qqconv holds}
Fix $X_0 \in \partial \Omega$ and let $X\in \partial \Omega$, $\Xbar_0$, $\Xbar_1 \in \partial \Omega^+(X_0)$. Let $\Xbar_t$ be the $c$-segment with respect to $X_0$ from $\Xbar_0$ to $\Xbar_1$. Then we have the inequality
\begin{equation}\label{eqn: QQconv}
-c(X,\Xbar_t) + c(X_0, \Xbar_t) + c(X, \Xbar_0) - c(X_0, \Xbar_0) \leq t \left( -c(X,\Xbar_1) + c(X_0, \Xbar_1) + c(X, \Xbar_0) - c(X_0, \Xbar_0)\right). 
\end{equation}
\end{prop}
\begin{proof}
By a rotation of coordinates we can assume $ \mathcal{N}_{\Omega}(X_0) = e_{n+1}$, then $\partial \Omega^+ (X_0)$ can be written as the graph of the concave function $\beta_{X_0}$. Then, for some $x$, $\xbar_0$, $\xbar_1\in \Omega_{X_0}$, we have
\begin{equation*}
    X = (x,\beta_{X_0}(x)),\ \Xbar_i = (\xbar_i, \beta_{X_0}(\xbar_i)),
\end{equation*}
and the $c$-segment $\Xbar_t$ is given by
\begin{equation*}
    \Xbar_t = ((1-t)\xbar_0 + t\xbar_1, \beta_{X_0}((1-t)\xbar_0 + t\xbar_1)) =: (\xbar_t,\beta_{X_0}(\xbar_t)).
\end{equation*}
First suppose $X\in \partial \Omega^+(X_0)$,  
then writing $x_0:=\proj_{X_0}^\Omega(X_0)$,  we compute
\begin{align*}
& -c(X,\Xbar_t) + c(X_0, \Xbar_t) + c(X, \Xbar_0) - c(X_0, \Xbar_0) \\
& = \inner{x-x_0}{\xbar_t-\xbar_0} + (\beta_{X_0}(x)-\beta_{X_0}(x_0))(\beta_{X_0}(\xbar_t)-\beta_{X_0}(\xbar_0)) \\
& = t \inner{x-x_0}{\xbar_1-\xbar_0} + (\beta_{X_0}(x)-\beta_{X_0}(x_0))(\beta_{X_0}(\xbar_t)-\beta_{X_0}(\xbar_0)).
\end{align*}
By concavity of $\beta_{X_0}$, we have
\begin{equation*}
    \beta_{X_0}(\xbar_t) -\beta_{X_0}(\xbar_0) \geq t( \beta_{X_0}(\xbar_1) - \beta_{X_0}(\xbar_0)),
\end{equation*}
noting that $\beta_{X_0}(x) \leq \beta_{X_0}(x_0)$, we obtain 
\begin{align*}
    &t \inner{x-x_0}{\xbar_1-\xbar_0} + (\beta_{X_0}(x)-\beta_{X_0}(x_0))(\beta_{X_0}(\xbar_t)-\beta_{X_0}(\xbar_0))\\
   & \leq t( \inner{x-x_0}{\xbar_1-\xbar_0} + (\beta_{X_0}(x)-\beta_{X_0}(x_0))(\beta_{X_0}(\xbar_1)-\beta_{X_0}(\xbar_0)) ) \\
   & = t \left( -c(X,\Xbar_1) + c(X_0, \Xbar_1) + c(X, \Xbar_0) - c(X_0, \Xbar_0)\right),
\end{align*}
giving the desired inequality in this case.

Now, suppose $X\in \partial \Omega\setminus\partial\Omega^+(X_0)$, and define
\begin{align*}
    \lambda_+:=\sup\{\lambda\geq 0\mid X+\lambda\normal{\Omega}{X_0}\in \Omega\},\qquad X^+:=X+\lambda_+\normal{\Omega}{X_0},
\end{align*}
since $\Omega$ is compact we see $\lambda_+<\infty$ and $X^+\in \partial \Omega$. We claim that $X^+\in \overline{\partial \Omega^+(X_0)}$. 
First, rotate coordinates so that $\normal{\Omega}{X^+}=e_{n+1}$ and $\proj_{X^+}(\normal{\Omega}{X_0})$ is in the positive $e_1$ direction; letting $x^+:=\proj_{X^+}(X^+)$ we can then write a portion of $\partial \Omega^+(X^+)$ as the graph of $\beta_{X^+}$ over $B_{\frac{\conerad(\frac{1}{2})}{2}}(x^+)$. Now, suppose that $\lambda_+=0$, then $X=X^+$ hence $\inner{\normal{\Omega}{X^+}}{\normal{\Omega}{X_0}}\leq 0$. However, if this inequality is strict, that implies the ray starting at $X=X^+$ in the direction of $\normal{\Omega}{X_0}$ has a nontrivial segment lying under the graph of $\beta_{X^+}$. This segment then must lie in $\Omega$, which contradicts the definition of $\lambda_+$, thus we must have $\inner{\normal{\Omega}{X^+}}{\normal{\Omega}{X_0}}=0$. On the other hand, if $\lambda_+>0$, note that
\begin{align*}
    0&\geq \inner{\normal{\Omega}{X^+}}{X-X^+}=-\lambda_+\inner{\normal{\Omega}{X^+}}{\normal{\Omega}{X_0}}.
\end{align*}
If this inequality is strict the claim is clear, hence we may assume in any case that $\inner{\normal{\Omega}{X^+}}{\normal{\Omega}{X_0}}=0$; in particular, this implies that in our choice of coordinates $\normal{\Omega}{X_0}=e_1$. 
Now, writing $X_t:=(x^++te_1, \beta_{X^+}(x^++te_1))$, by concavity we have for all $t\in (0, \conerad)$,
\begin{align*}
    0&=\inner{-\nabla^n\beta_{X^+}(x^+)}{e_1}=-\left.\frac{d}{ds}\beta_{X^+}(x^++se_1)\right \vert_{s=0}\\
    &\leq -\left.\frac{d}{ds}\beta_{X^+}(x^++se_1)\right \vert_{s=t}= \inner{-\nabla^n\beta_{X^+}(x^++te_1)}{e_1}\\
    &=\inner{\normal{\Omega}{X_t}}{\normal{\Omega}{X_0}}\sqrt{1+\abs{\nabla^n\beta_{X^+}(x^++te_1)}^2}.
\end{align*}
However, if $\inner{-\nabla^n\beta_{X^+}(x^++te_1)}{e_1}=0$ the graph of $\beta_{X^+}$ would contain the nontrivial line segment $[X^+, X^++t\normal{\Omega}{X_0}]$ again contradicting the definition of $\lambda_+$, hence the above implies $\inner{\normal{\Omega}{X_t}}{\normal{\Omega}{X_0}}>0$ for all $t\in (0, \conerad)$. In particular, $X^+\in \overline{\partial \Omega^+(X_0)}$, and the claim is proved.

Now since $-c(X^+, \Xbar)=-c(X, \Xbar)-\lambda_+\inner{\normal{\Omega}{X_0}}{X-\Xbar}-\frac{\lambda_+^2}{2}$ for any $\Xbar\in \partial \Omega$, the calculation in the first case above gives 
\begin{align*}
    &-c(X, \Xbar_t)+c(X_0, \Xbar_t)+c(X, \Xbar_0)-c(X_0, \Xbar_0)\\
    &=-c(X^+, \Xbar_t)+c(X_0, \Xbar_t)+c(X^+, \Xbar_0)-c(X_0, \Xbar_0)+\lambda_+\inner{\normal{\Omega}{X_0}}{\Xbar_0-\Xbar_t}\\
    &\leq t \left( -c(X^+,\Xbar_1) + c(X_0, \Xbar_1)+c(X^+, \Xbar_0) - c(X_0, \Xbar_0)\right)
    +\lambda_+\inner{\normal{\Omega}{X_0}}{\Xbar_0-\Xbar_t}\\
    &= t \left( -c(X,\Xbar_1) + c(X_0, \Xbar_1)+c(X, \Xbar_0) - c(X_0, \Xbar_0)\right)
    +\lambda_+\inner{\normal{\Omega}{X_0}}{(1-t)\Xbar_0+t\Xbar_1-\Xbar_t}.
\end{align*}
By our assumptions, $\Xbar_t\in \partial\Omega^+(X_0)$ for all $t\in [0, 1]$, hence we can calculate
\begin{align*}
    &\inner{\normal{\Omega}{X_0}}{(1-t)\Xbar_0+t\Xbar_1-\Xbar_t}\\
    &=(1-t)\beta_{X_0}(\proj^\Omega_{X_0}(\Xbar_0))+t\beta_{X_0}(\proj^\Omega_{X_0}(\Xbar_1))-\beta_{X_0}((1-t)\proj^\Omega_{X_0}(\Xbar_0)+t\proj^\Omega_{X_0}(\Xbar_1))\leq 0
\end{align*}
by concavity of $\beta_{X_0}$, this yields the desired inequality in the general case.
\end{proof}
Using the above property, we can obtain convexity of sections in certain coordinates, as in \cite[Corollary 2.21]{GuillenKitagawa15}.
\begin{cor}\label{cor: csubdiff c-convex}
\begin{enumerate}
Suppose $u$ is a $c$-convex function.
    \item For any $h>0$ and $\hat X$, $\Xbar_0\in \partial \Omega$, if  $X_0$, $X_1\in S^{u}_{h, \hat X, \Xbar_0}\cap \partial\Omega^+(\Xbar_0)$, and $X_t$ is the $c$-segment with respect to $\Xbar_0$ from $X_0$ to $X_1$, then $X_t\in S^{u}_{h, \hat X, \Xbar_0}$ for all $t\in [0, 1]$.
    \item If $\Xbar_0$, $\Xbar_1\in \csubdiff[\Omega]{u}{X_0}\cap\partial \Omega^+(X_0)$ for some $X_0$, and $\Xbar_t$ is the $c$-segment from $\Xbar_0$ to $\Xbar_1$ with respect to $X_0$, then $\Xbar_t\in \csubdiff[\Omega]{u}{X_0}$. 
\end{enumerate}
\end{cor}
\begin{proof}
By Remark~\ref{rmk: contact set is dual c-sub}, we see that (2) follows from (1) applied to $S^{u^c}_{0, X_0, \Xbar_0}$, where the roles of variables with and without bars above them are reversed.

To show (1), let $\Xbar\in \csubdiff[\Omega]{u}{X_t}$ and write $m_0(X):=-c(X, \Xbar_0)+c(\hat X, \Xbar_0)+u(\hat X)+h$. Up to relabeling, we may assume that
\begin{align*}
    -c(X_1, \Xbar)+c(X_1, \Xbar_0)\leq -c(X_0, \Xbar)+c(X_0, \Xbar_0).
\end{align*}
Then \eqref{eqn: QQconv} from Proposition~\ref{prop: qqconv holds} (with the roles of the variables $X$ and $\Xbar$ reversed) yields
\begin{align*}
    u(X_t)-m_0(X_t)&=-c(X_t, \Xbar)+c(X_t, \Xbar_0)+c(X_t, \Xbar)+u(X_t)-c(\hat X, \Xbar_0)-u(\hat X)-h\\
    &\leq -c(X_0, \Xbar)+c(X_0, \Xbar_0)+c(X_t, \Xbar)+u(X_t)-c(\hat X, \Xbar_0)-u(\hat X)-h\\
    &\leq u(X_0)-m_0(X_0)\leq 0.
\end{align*}
\end{proof}
\begin{defin}
    If $U$ is a real valued function defined on a neighborhood of some point $x_0\in \R^n$, the \emph{local subdifferential} of $U$ at $x_0$ is defined by
    \begin{align*}
        \locsubdiff{U}{x_0}:=\{p\in \R^n\mid U(x_0+v)\geq U(x_0)+\inner{p}{v}+o(\lvert v\rvert),\ v\to 0\}.
    \end{align*}
\end{defin}
We are now able to show a ``local-to-global'' property of supporting $c$-functions as in \cite[Corollary 2.20]{GuillenKitagawa15}. Before doing so, we introduce a bit of notation which will also be useful in Section~\ref{section: aleksandrov}.

\begin{defin}
    The map $-\nabla_X^\Omega c(X_0, \cdot) = \proj^\Omega_{X_0}(\cdot - X_0)$ is injective on $\partial \Omega^+(X_0)$. Identifying $\proj^\Omega_{X_0}(\partial \Omega^+(X_0)-X_0)$ with a subset of $T_{X_0}\partial \Omega$, we define the \emph{$c$-exponential map}  $\cExp{X_0}{\cdot} : \proj^\Omega_{X_0}(\partial \Omega^+(X_0)-X_0) \to \partial \Omega^+(X_0)$ with respect to $X_0$ as the inverse function of $\proj^\Omega_{X_0}(\cdot - X_0) : \partial \Omega^+ (X_0) \to T_{X_0}\partial \Omega$.

\end{defin}

\begin{rmk}\label{rmk: cexp lip estimate}
Note that after rotating coordinates so that $\normal{\Omega}{X_0}=e_{n+1}$, we can write
    \begin{align*}
        \cExp{X_0}{p}=(\proj^\Omega_{X_0}(X_0)+p, \beta_{X_0}(\proj^\Omega_{X_0}(X_0)+p)).
    \end{align*}
    Since $\partial \Omega$ is $C^1$, $\beta_{X_0}$ is differentiable, and so is $\cExp{X_0}{\cdot}$. Taking a derivative, we obtain (viewing $D \cExp{X_0}{p}$ for $p\in\proj^\Omega_{X_0}(\partial\Omega^+(X_0)-X_0)$ as a map $T_{X_0} \partial \Omega\cong T_x (T_{X_0} \partial \Omega)\to \R^{n+1}$) for any $v\in \R^n$,
    \begin{equation}
    \label{eqn: derivative cexp}
        D \cExp{X_0}{p}v = \left.\frac{d}{dt}\cExp{X_0}{p+tv}\right\vert_{t=0}=(v, \inner{\nabla^n\beta_{X_0}(\proj^\Omega_{X_0}(X_0)+p)}{v}),
    \end{equation}
    which we note belongs to $T_{\cExp{X_0}{p}}\partial \Omega$.
\end{rmk}
For $X\in \partial \Omega$, $Y \in \partial \Omega^+(X)$, we define $[Y]_X := \proj^\Omega_X (Y-X)$ and we use a similar notation for sets: if $A \subset \partial \Omega^+(X)$, then $[A]_X := \proj^\Omega_X(A-X)$.
\begin{cor}\label{cor: local to global}
Suppose $u$ is a $c$-convex function such that $\csubdiff[\Omega]{u}{X_0}\subset \partial \Omega^+(X_0)$ for some $X_0\in \partial\Omega$, and define $U: [\partial\Omega^+(X_0)]_{X_0}\to \R$ by 
    \begin{align*}
        U(x):=u(\cExp{X_0}{x})
    \end{align*}
    (after identifying $T_{X_0}\partial\Omega$ with $\R^n$). Then $\locsubdiff{U}{0}\subset [\partial\Omega^+(X_0)]_{X_0}$, and $\cExp{X_0}{\locsubdiff{U}{0}}\subset \csubdiff[\Omega]{u}{X_0}$. 
\end{cor}
\begin{proof}
First suppose $u$ is tangentially differentiable at $X_0$, then $U$ is differentiable at $0$; hence $\{\nabla^\Omega u(X_0)\}=\locsubdiff{U}{0}$, where here we identify $T_{X_0}\partial\Omega$ with $\R^n$ by an abuse of notation. 
Also, since $u$ is $c$-convex, we see there exists $\Xbar_0\in \csubdiff[\Omega]{u}{X_0}\subset \partial \Omega^+(X_0)$ such that $[\Xbar_0]_{X_0}=\nabla^\Omega u(X_0)$, in particular $\nabla^\Omega u(X_0)\in [\partial\Omega^+(X_0)]_{X_0}$ and $\cExp{X_0}{\nabla^\Omega u(X_0)}=\Xbar_0\in \csubdiff[\Omega]{u}{X_0}$ so the corollary is proved in this case.

In the general case, we claim that the set of exposed points of $\locsubdiff{U}{0}$ is contained in  the set 
    \begin{align*}
    \locdiff{U}{0}:=\{p\in \R^n\mid \exists x_k\to 0,\ p=\lim_{k\to \infty}\nabla^n U(x_k)\}.
    \end{align*}
    Indeed, suppose that $p_0\in \R^n$ is an exposed point of $\locsubdiff{U}{0}$, this implies there exists a unit vector $v_0\in \R^n$ such that for any $p\in \locsubdiff{U}{0}\setminus \{p_0\}$,
    \begin{align}\label{eqn: supporting at exposed}
        \inner{p-p_0}{v_0}<0.
    \end{align}
    Since $u$ is $c$-convex, $U$ is Lipschitz in a neighborhood of $0$ (recall Remark~\ref{rmk: subdifferentials}), in particular differentiable on a dense subset. Thus by considering a sequence of cones with shrinking openings whose vertices are at $0$ and axial directions are $v_0$, we can find a sequence $\{x_k\}_{k=1}^\infty\subset \R^n\setminus\{0\}$ such that $x_k\to 0$, $x_k/\abs{x_k}\to v_0$ as $k\to\infty$, and $U$ is differentiable at each $x_k$. Since $U$ is Lipschitz in a neighborhood of $0$, we see $\{\nabla^n U(x_k)\}_{k=1}^\infty$ is bounded, hence we may pass to a subsequence to assume $\nabla^n U(x_k)\to p_\infty$ for some $p_\infty\in \R^n$. By the $c$-convexity of $u$ we find for each $k$ there exists some $\Xbar_k\in \csubdiff{u}{\cExp{X_0}{x_k}}$, then since $U$ is differentiable at $x_k$ we can see that $\nabla^n U(x_k)=-\nabla^n_x c(\cExp{X_0}{x}, \Xbar_k)\vert_{x=x_k}$. Now by compactness of $\partial\Omega$, we may pass to another subsequence to assume the $\Xbar_k$ converge to some $\Xbar_\infty\in \partial \Omega$, which we can see belongs to $\csubdiff{u}{X_0}$. In turn, this implies that $-\nabla^n_x c(\cExp{X_0}{x}, \Xbar_\infty)\vert_{x=0}\in \locsubdiff{U}{0}$; since $c\in C^1(\partial\Omega\times\partial\Omega)$ we see $-\nabla^n_x c(\cExp{X_0}{x}, \Xbar_\infty)\vert_{x=0}=p_\infty$, thus
    \begin{align}\label{eqn: limit in local subdiff}
        p_\infty\in \locsubdiff{U}{0}.
    \end{align}
    By the mean value theorem, for each $k$ we can find some $t_k\in [0, 1]$ such that
    \begin{align*}
        -c(\cExp{X_0}{0}, \Xbar_k)+c(\cExp{X_0}{x_k}, \Xbar_k)=\inner{\nabla^n_x c(\cExp{X_0}{x}, \Xbar_k)\vert_{x=t_kx_k}}{x_k},
    \end{align*}
    using this we can calculate,
    \begin{align}
        U(x_k)
        &\geq U(0)+\inner{x_k}{p_0}+o(\abs{x_k})
        =u(\cExp{X_0}{0})+\inner{x_k}{p_0}+o(\abs{x_k})\notag\\
        &\geq -c(\cExp{X_0}{0}, \Xbar_k)+c(\cExp{X_0}{x_k}, \Xbar_k)+u(\cExp{X_0}{x_k})+\inner{x_k}{p_0}+o(\abs{x_k})\notag\\
        &=\inner{x_k}{p_0+\nabla^n_x c(\cExp{X_0}{x}, \Xbar_k)\vert_{x=t_kx_k}}+U(x_k)+o(\abs{x_k})\text{ as }k\to\infty.\label{eqn: local subdiff ineq}
    \end{align}
    Again using that $c\in C^1(\partial\Omega\times\partial\Omega)$, we have
    \begin{align*}
        &\abs{p_\infty+\nabla^n_x c(\cExp{X_0}{x}, \Xbar_k)\vert_{x=t_kx_k}}
        \leq \abs{p_\infty-\nabla^n_x U(x_k)}
        +\abs{\nabla^n_x U(x_k)+\nabla^n_x c(\cExp{X_0}{x}, \Xbar_k)\vert_{x=x_k}}\\
        &+\abs{-\nabla^n_x c(\cExp{X_0}{x}, \Xbar_k)\vert_{x=x_k}+\nabla^n_x c(\cExp{X_0}{x}, \Xbar_k)\vert_{x=t_kx_k}}\to 0,
    \end{align*}
    thus dividing~\eqref{eqn: local subdiff ineq} by $\abs{x_k}$ and taking $k\to \infty$ yields
    \begin{align*}
        0\leq \inner{p_\infty-p_0}{v_0}.
    \end{align*}
    Combining with~\eqref{eqn: limit in local subdiff} and~\eqref{eqn: supporting at exposed}, this implies $p_0=p_\infty$, and the claim is proved.
    
    Hence the closed, convex set $\locsubdiff{U}{0}$ is the closed convex hull of $\locdiff{U}{0}$. Now suppose $x_k\to 0$ where each $x_k$ is a point of differentiability of $U$, and that $\lim_{k\to\infty}\nabla^nU(x_k)=p$ for some $p\in\R^n$.  Then $X_k:=\cExp{X_0}{x_k}$ is a point of tangential differentiability of $u$ for each $k$, thus by the argument above there exist  $\Xbar_k\in \csubdiff[\Omega]{u}{X_k}$ with $\nabla ^\Omega u(X_k)=[\Xbar_k]_{X_k}\in [\partial\Omega^+(X_k)]_{X_k}$. Passing to a subsequence we may assume $\Xbar_k\to \Xbar_\infty$ for some $\Xbar_\infty\in \partial \Omega$, which belongs to $\csubdiff[\Omega]{u}{X_0}$ by continuity of $u$ and $c$. We now claim that 
    \begin{align*}
        [\Xbar_\infty]_{X_0}=p.
    \end{align*}
    Indeed, for any fixed index $1\leq i\leq n$, we have that $(e_i, \partial_i\beta_{X_0}(\proj^\Omega_{X_0}(X_0)+x_k))\in T_{X_k}\partial \Omega$ and
    \begin{align*}
        \partial_iU(x_k)&=\left.\frac{d}{dt}u(\cExp{X_0}{x_k+te_i})\right\vert_{t=0}\\
&=\inner{\nabla^\Omega u(X_k)}{(e_i, \partial_i\beta_{X_0}(\proj^\Omega_{X_0}(X_0)+x_k))}\\
        &=\inner{[\Xbar_k]_{X_k}}{(e_i, \partial_i\beta_{X_0}(\proj^\Omega_{X_0}(X_0)+x_k))}.
    \end{align*}
    Taking a limit on both sides above yields
    \begin{align*}
        \inner{p}{e_i}&=\lim_{k\to\infty}\partial_iU(x_k)\\
        &=\lim_{k\to\infty}\inner{[\Xbar_k]_{X_k}}{(e_i, \partial_i\beta_{X_0}(\proj^\Omega_{X_0}(X_0)+x_k))}
        =\inner{[\Xbar_\infty]_{X_0}}{(e_i, 0)},
    \end{align*}
    proving the claim. Since $\Xbar_k\in \partial \Omega^+(X_k)$ for each $k$, we have $\Xbar_\infty\in \partial \Omega^+(X_0)$ hence the above shows $\locdiff{U}{0}\subset [\partial\Omega^+(X_0)]_{X_0}$. Since $[\partial\Omega^+(X_0)]_{X_0}$ is convex, this implies that $\locsubdiff{U}{0}\subset[\partial\Omega^+(X_0)]_{X_0}$, in particular $\Xbar_0:=\cExp{X_0}{p}$ is well-defined for any $p\in \locsubdiff{U}{0}$. 
    Now we see the above yields
    \begin{align*}
        \locdiff{U}{0}\subset [\csubdiff[\Omega]{u}{X_0}]_{X_0},
    \end{align*}
    then since $\csubdiff[\Omega]{u}{X_0}\subset \partial \Omega^+(X_0)$, Corollary~\ref{cor: csubdiff c-convex} (2) implies $[\csubdiff[\Omega]{u}{X_0}]_{X_0}$ is convex, hence we obtain
    \begin{align*}
        [\Xbar_0]_{X_0}\in \locsubdiff{U}{0}\subset [\csubdiff[\Omega]{u}{X_0}]_{X_0}.
    \end{align*}
    Finally since $\Xbar_0\in \partial \Omega^+(X_0)$ and $\csubdiff[\Omega]{u}{X_0}\subset \partial \Omega^+(X_0)$, this implies $\Xbar_0\in\csubdiff[\Omega]{u}{X_0}$.
\end{proof}

\section{Localized Aleksandrov estimates}\label{section: aleksandrov}
In this section, we prove analogues of the classical upper and lower Aleksandrov estimates in Monge-Amp{\`e}re theory. Since the proofs rely on Proposition~\ref{prop: qqconv holds}, they will only apply to suitably localized sections. We start with the lower Aleksandrov estimate whose proof bears similarities to that of \cite[Lemma 3.7]{GuillenKitagawa15}. Below, dilations of convex sets in Euclidean space are with respect to their center of mass. Throughout this section, we fix
\begin{align*}
    \Xbar_0&\in\partial \Omega, \qquad m_0(X):=-c(X, \Xbar_0)+\tilde h \text{ for some }\tilde h>0,\\
    S :&=\{ X \in \partial \Omega \mid u(X) \leq m_0(X) \}.
\end{align*}

\begin{rmk}\label{rmk: section equivalence}
    Note that for any $h>0$, by taking $\tilde h=c(X_0, \bar X_0)+u(X_0)+h$ we find the section $S$ defined above is equal to $S^u_{h, X_0, \bar X_0}$.
\end{rmk}

\begin{thm}\label{thm: lower aleksandrov}
Let $A\subset S$ be such that $\proj^\Omega_{\Xbar_0}(A)$ is convex and $X_{cm}\in \partial\Omega^+(\Xbar_0)$ be such that $\proj^\Omega_{\Xbar_0}(X_{cm})$ is the center of mass of $\proj^\Omega_{\Xbar_0}(A)$. Also assume $2\proj^\Omega_{\Xbar_0}(A)\subset \proj^\Omega_{\Xbar_0}(S)$, where the dilation of $\proj^\Omega_{\Xbar_0}(A)$ is with respect to $\proj^\Omega_{\Xbar_0}(X_{cm})$. 
Fix $\Theta\in (0, 1)$ and suppose that 
\begin{align}
    S&\subset\partial\Omega^+_{\Theta}(\Xbar_0),\label{eqn: lower est section in same side}\\
    S&\subset B^{n+1}_{\frac{\conerad(\frac{1}{2})}{4}}(X_{cm}),\label{eqn: lower est section in center ball}\\
    \csubdiff[\Omega]{u}{S}&\subset \partial \Omega^+(X_{cm}).\label{eqn: lower est csub in same side}
\end{align}
Then 
\begin{align*}
    \sup_{S}(m_0-u)^n\geq \frac{\Theta}{4}\mathcal{H}^n(A)\mathcal{H}^n(\csubdiff[\Omega]{u}{A}).
\end{align*}
\end{thm}
The following is the key lemma for the estimate.
\begin{lem}\label{lem: bound on c function difference}
Under the same hypotheses and notation as Theorem~\ref{thm: lower aleksandrov}, if $\Ybar\in \csubdiff[\Omega]{u}{Y}$ for some $Y\in A$, then for any $X\in A$,
\begin{align*}
    -c(X, \Ybar)+c(X_{cm}, \Ybar)-(-c(X, \Xbar_0)+c(X_{cm}, \Xbar_0))\leq \sup_S(m_0-u).
\end{align*}
\end{lem}
\begin{proof}
Let us rotate coordinates so that $\normal{\Omega}{\Xbar_0}=e_{n+1}$, write $\beta$ for $\beta_{\Xbar_0}$, let $\ybar:=\proj^\Omega_{\Xbar_0}(\Ybar)$,  $y:=\proj^\Omega_{\Xbar_0}(Y)$, $x_{cm}:=\proj^\Omega_{\Xbar_0}(X_{cm})$, and fix $X=(x, \beta(x))\in A$. Then there exists a point $z_x\in \partial^n (\proj^\Omega_{\Xbar_0}(S))$ where $\partial^n$ is the boundary of a set with respect to the relative topology of $\R^n$, such that $x$ lies on the segment $[x_{cm}, z_x]$. Similarly, there is $z_y\in \partial^n (\proj^\Omega_{\Xbar_0}(S))$ such that $y$ lies on the segment $[x_{cm}, z_y]$; by hypothesis $x=(1-t_x)x_{cm}+t_xz_x$ and $y=(1-t_y)x_{cm}+t_yz_y$ for some $t_x$, $t_y\in [0, \frac{1}{2}]$. Let us write $Z_x:=(z_x, \beta(z_x))$, $Z_y:=(z_y, \beta(z_y))$. By \eqref{eqn: lower est section in same side}, $S\subset\partial\Omega^+(\Xbar_0)$, so we may apply Proposition~\ref{prop: qqconv holds} (with the roles of the $X$ and $\Xbar$ variables switched) to calculate,
\begin{align}
    &-c(X, \Ybar)+c(X_{cm}, \Ybar)-(-c(X, \Xbar_0)+c(X_{cm}, \Xbar_0))\notag\\
    &\leq t_x(-c(Z_x, \Ybar)+c(X_{cm}, \Ybar)-(-c(Z_x, \Xbar_0)+c(X_{cm}, \Xbar_0)))\notag\\
    &=t_x(-c(Z_x, \Ybar)+c(X_{cm}, \Ybar)-(m_0(Z_x)-m_0(X_{cm})))\notag\\
    &\leq t_x(-c(Z_x, \Ybar)+c(X_{cm}, \Ybar)-u(Z_x)+m_0(X_{cm}))\notag\\
    &\leq t_x(-c(Y, \Ybar)+c(X_{cm}, \Ybar)-u(Y)+m_0(X_{cm}))
    \leq \frac{1}{2}(-c(Y, \Ybar)+c(X_{cm}, \Ybar)-u(Y)+m_0(X_{cm})),\label{eqn: first QQConv bound}
\end{align}
where we have used that $-c(X_{cm}, \Ybar)+c(Y, \Ybar)+u(Y)\leq u(X_{cm})\leq m_0(X_{cm})$ and $t_x\leq 1/2$ for the last inequality. Continuing and using Proposition~\ref{prop: qqconv holds} again,
\begin{align*}
    &-c(Y, \Ybar)+c(X_{cm}, \Ybar)-u(Y)+m_0(X_{cm})\notag\\
    &\leq t_y(-c(Z_y, \Ybar)+c(X_{cm}, \Ybar)-(-c(Z_y, \Xbar_0)+c(X_{cm}, \Xbar_0)))+m_0(Y)-m_0(X_{cm})-u(Y)+m_0(X_{cm})\notag\\
    &\leq t_y(-c(Z_y, \Ybar)+c(X_{cm}, \Ybar)-u(Z_y)+m_0(X_{cm}))+\sup_S(m_0-u)\notag\\
    &\leq t_y(-c(Y, \Ybar)+c(X_{cm}, \Ybar)-u(Y)+m_0(X_{cm}))+\sup_S(m_0-u)\notag\\
    &\leq \frac{1}{2}(-c(Y, \Ybar)+c(X_{cm}, \Ybar)-u(Y)+m_0(X_{cm}))+\sup_S(m_0-u),
\end{align*}
thus rearranging and combining with \eqref{eqn: first QQConv bound} yields
\begin{align*}
    &-c(X, \Ybar)+c(X_{cm}, \Ybar)-(-c(X, \Xbar_0)+c(X_{cm}, \Xbar_0))\\
    &\leq \frac{1}{2}(-c(Y, \Ybar)+c(X_{cm}, \Ybar)-u(Y)+m_0(X_{cm}))\leq \sup_S(m_0-u).
\end{align*}
\end{proof}
We now prove the lower Aleksandrov estimate.
\begin{proof}[Proof of Theorem~\ref{thm: lower aleksandrov}]
Make a rotation so that $\normal{\Omega}{X_{cm}}=e_{n+1}$ and write $x_{cm}:=\proj^\Omega_{X_{cm}}(X_{cm})$. Also let $\Ybar\in \csubdiff[\Omega]{u}{Y}$ for some $Y\in A$; by~\eqref{eqn: lower est csub in same side} we have $\Ybar\in \partial\Omega^+(X_{cm})$, while~\eqref{eqn: lower est section in same side} implies $X_{cm}\in \partial\Omega^+_\Theta(\Xbar_0)$, hence $\Xbar_0\in \partial\Omega^+(X_{cm})$, and by \eqref{eqn: lower est section in center ball} we have $A\subset \partial\Omega^+(X_{cm})$. Then we can write $\xbar_0:=\proj^\Omega_{X_{cm}}(\Xbar_0)$, $\ybar:=\proj^\Omega_{X_{cm}}(\Ybar)$, and $\ybar_t:=(1-t)\xbar_0+t\ybar$ for $t\in [0, 1]$. Fix $X:=(x, \beta_{X_{cm}}(x))\in A$, then we may apply Proposition~\ref{prop: qqconv holds} along $\Ybar_t:=(\ybar_t, \beta_{X_{cm}}(\ybar_t))$: dividing \eqref{eqn: QQconv} by $t$ and taking $t\to 0^+$, then using Lemma~\ref{lem: bound on c function difference} yields
\begin{align*}
    \sup_S(m_0-u)&\geq -c(X, \Ybar)+c(X_{cm}, \Ybar)-(-c(X, \Xbar_0)+c(X_{cm}, \Xbar_0))\\
    &\geq \left.\frac{d}{dt}(-c(X, \Ybar_t)+c(X_{cm}, \Ybar_t))\right\vert_{t=0^+}\\
    &=\inner{x+(\beta_{X_{cm}}(x)-\beta_{X_{cm}}(x_{cm}))\nabla^n\beta_{X_{cm}}(\xbar_0)-x_{cm}}{\ybar-\xbar_0}.
\end{align*}
In the notation of \cite[Definition 3.2]{GuillenKitagawa15}, this implies that
\begin{align*}
    \proj^\Omega_{X_{cm}}(\csubdiff[\Omega]{u}{A})\subset \left(\Phi_0(\proj^\Omega_{X_{cm}}(A))\right)^*_{x_{cm}, \xbar_0, \sup_S(m_0-u)}
\end{align*}
where 
\begin{align*}
    \Phi_0(z):=z+(\beta_{X_{cm}}(z)-\beta_{X_{cm}}(x_{cm}))\nabla^n\beta_{X_{cm}}(\xbar_0),\qquad z\in \proj^\Omega_{X_{cm}}(A).
\end{align*}
Then by \cite[Lemma 3.9]{GuillenKitagawa15}, we obtain
\begin{align}\label{eqn: transformed lower estimate}
    \mathcal{H}^n(\Phi_0(\proj^\Omega_{X_{cm}}(A)))\mathcal{H}^n(\proj^\Omega_{X_{cm}}(\csubdiff[\Omega]{u}{A}))\leq C_n\sup_S(m_0-u)^n,
\end{align}
for some $C_n>0$ depending only on $n$.
We claim that $\Phi_0$ is injective. To show this, suppose $\Phi_0(z_0)=\Phi_0(z_1)$, $z_0$, $z_1\in \Omega_{X_{cm}}$. Then 
\begin{align}\label{eqn: Phi_0 injectivity}
    z_0-z_1=(\beta_{X_{cm}}(z_1)-\beta_{X_{cm}}(z_0))\nabla^n\beta_{X_{cm}}(\xbar_0),
\end{align} hence if $\beta_{X_{cm}}(z_1)=\beta_{X_{cm}}(z_0)$ we obtain $z_1=z_0$. Suppose by contradiction that $\beta_{X_{cm}}(z_1)>\beta_{X_{cm}}(z_0)$, and write $Z_i:=(z_i, \beta_{X_{cm}}(z_i))$ for $i=0$, $1$. Then since $A\subset S\subset \partial \Omega^+_{\Theta}(\Xbar_0)$ by \eqref{eqn: lower est section in same side}, we have
\begin{align*}
    0<\inner{\normal{\Omega}{Z_0}}{\normal{\Omega}{\Xbar_0}}=\frac{\inner{\nabla^n\beta_{X_{cm}}(z_0)}{\nabla^n\beta_{X_{cm}}(\xbar_0)}+1}{\sqrt{1+\abs{\nabla^n\beta_{X_{cm}}(z_0)}^2}\sqrt{1+\abs{\nabla^n\beta_{X_{cm}}(\xbar_0)}^2}}
\end{align*}
However by \eqref{eqn: Phi_0 injectivity} and concavity of $\beta_{X_{cm}}$, we see
\begin{align*}
   (\beta_{X_{cm}}(z_1)-\beta_{X_{cm}}(z_0))\inner{\nabla^n\beta_{X_{cm}}(z_0) }{\nabla^n\beta_{X_{cm}}(\xbar_0)}&=\inner{\nabla^n\beta_{X_{cm}}(z_0) }{z_0-z_1}\leq \beta_{X_{cm}}(z_0)-\beta_{X_{cm}}(z_1)
\end{align*}
which upon dividing by $\beta_{X_{cm}}(z_1)-\beta_{X_{cm}}(z_0)$ yields a contradiction. Thus $\beta_{X_{cm}}(z_1)=\beta_{X_{cm}}(z_0)$ and $\Phi_0$ is injective. By Sylvester's determinant theorem we have
\begin{align*}
    \det D\Phi_0(z)=1+\inner{\nabla^n\beta_{X_{cm}}(z) }{\nabla^n\beta_{X_{cm}}(\xbar_0)}\geq \inner{\normal{\Omega}{(z, \beta_{X_{cm}}(z))}}{\normal{\Omega}{\Xbar_0}}>\Theta
\end{align*}
for all $z\in \proj^\Omega_{X_{cm}}(A)$. Using Lemma~\ref{lem: gradient beta est} we can see that
\begin{align*}
A\subset B^{n+1}_{\frac{\conerad(1/2)}{4}}(X_{cm})\cap \partial \Omega\subset \partial\Omega^+_{1/2}(X_{cm}).  
\end{align*}
Thus we may apply  \eqref{eqn: lower est section in center ball} combined with  Lemma~\ref{lem: gradient beta est} to obtain
\begin{align*}
    \mathcal{H}^n(A)&=\int_{\proj^\Omega_{X_{cm}}(A)}\sqrt{1+\abs{\nabla^n\beta_{X_{cm}}(x)}^2}dx\\
    &\leq\int_{\proj^\Omega_{X_{cm}}(A)}\sqrt{1+\sqrt{3}^2}dx= 2\mathcal{H}^n(\proj^\Omega_{X_{cm}}(A)),
\end{align*}
and using \eqref{eqn: lower est csub in same side} the same inequality with $A$ replaced by $\csubdiff[\Omega]{u}{A}$ holds. Hence combining the change of variables formula with \eqref{eqn: transformed lower estimate} gives
\begin{align*}
  C_n\sup_S(m_0-u)^n&\geq \Theta \mathcal{H}^n(\proj^\Omega_{X_{cm}}(A))\mathcal{H}^n(\proj^\Omega_{X_{cm}}(\csubdiff[\Omega]{u}{A}))\\
   &\geq \frac{\Theta}{4} \mathcal{H}^n(A)\mathcal{H}^n(\csubdiff[\Omega]{u}{A}),
\end{align*}
finishing the proof of the theorem.
\end{proof}
The localized version of the upper Aleksandrov estimate \cite[Theorem 2.1]{GuillenKitagawa17} will require some more preliminary lemmas.

Before we state our estimate, we fix some notation. Recall that for $X\in \partial \Omega$, $Y \in \partial \Omega^+(X)$, we denote $[Y]_X := \proj^\Omega_X (Y-X)$ and for $A\subset \partial \Omega^+(X)$ we write $[A]_X := \proj^\Omega_X(A-X)$. Also, if $A$ is a convex set (in any dimension) and $w$ is a unit vector, we define $\Pi^w_A$ to be the hyperplane which supports $A$ at a point at which $w$ is an outward normal vector. Also we define $l(A, w)$ to be the length of the longest segment contained in $A$ which is parallel to $w$, i.e.
    \begin{equation*}
        l(A, w) := \sup\{ \lvert y_1 - y_2 \rvert \mid y_1, y_2 \in A,\ y_1 - y_2 = \lvert y_1 - y_2\rvert w \}.
    \end{equation*}
\begin{thm}\label{thm: upper aleksandrov}
Assume that
\begin{align}
 \csubdiff[\Omega]{u}{X}&\subset B^{n+1}_{\frac{\conerad(\frac{35}{36})}{8}}(X),\quad\forall X\in \partial\Omega,\label{eqn: csub in same side}\\
 S&\subset B^{n+1}_{\frac{\conerad(\frac{35}{36})}{8}}(\Xbar_0).\label{eqn: section in same side as slope}
\end{align}
Then for any unit length $w\in \R^n$ and $X_0\in S$ such that $m_0(X_0)>u(X_0)$ we have the following inequality: 
\begin{equation}
    C \H^n(\csubdiff[\Omega]{u}{S}) \H^n(S) \geq \frac{l([S]_{\Xbar_0},w)}{d\left( [X_0]_{\Xbar_0}, \Pi^w_{[S]_{\Xbar_0}}\right)}(m_0(X_0)-u(X_0))^n
\end{equation}
for some constant $C$ that depends on $\Omega$.
\end{thm}
We will follow the idea of \cite[Lemma 5.5]{GuillenKitagawa17}. In the proof, we use generalized (Clarke) subdifferentials and a ``$c$-cone'' to control the set $\csubdiff[\Omega]{u}{S}$. We first give the definition of a $c$-cone and some lemmas before proving Theorem~\ref{thm: upper aleksandrov}.

\begin{defin}\label{def:c-cone}
Fix $X_0\in S$. We define the \emph{$c$-cone with base $S$ and vertex $X_0$} by the function
    \begin{equation*}
        K_{S, X_0}(X) := \sup\{ -c(X,\Xbar)+c(X_0,\Xbar) + u(X_0) \mid -c(\cdot,\Xbar)+c(X_0,\Xbar)+u(X_0) \leq m_0 \textrm{ on } S \}.
    \end{equation*}
\end{defin}

\begin{rmk}\label{rmk: csubdiff of c cone}
A $c$-cone is a $c$-convex function as it is defined as a supremum of $-c$. Then Corollary~\ref{cor: csubdiff c-convex} implies that $[\csubdiff[\Omega]{K_{S, X_0}}{X_0} \cap \partial \Omega^+(X_0)]_{X_0}$ is convex. Also, it is easy to see that if $\Xbar$ satisfies $-c(\cdot,\Xbar)+c(X_0,\Xbar)+u(X_0) \leq m_0$ on $S$, then $-c(X,\Xbar)+c(X_0,\Xbar)+u(X_0) \leq K_{S, X_0}(X) $ for any $X \in \partial \Omega$ by definition of $K_{S, X_0}$, i.e. $\Xbar \in \csubdiff[\Omega]{K_{S, X_0}}{X_0}$. Note that we did not have to check $\Xbar \in \partial \Omega^+(X_0)$ in contrast to Corollary~\ref{cor: csubdiff c-convex}. 
\end{rmk}

\begin{lem}\label{lem: csubdiff cone in csubdiff u }
Assume $u$ is such that $\csubdiff[\Omega]{u}{X}\subset \partial \Omega^+(X)$ for every $X\in \partial \Omega$, and fix $X_0\in S$ with $m_0(X_0)>u(X_0)$. Then we have
    \begin{equation*}
        \csubdiff[\Omega]{K_{S, X_0}}{X_0} \cap \bigcap_{X \in S} \partial \Omega^+(X) \subset \csubdiff[\Omega]{u}{S}.
    \end{equation*}
\end{lem}
\begin{proof}
Let $\Xbar \in \csubdiff[\Omega]{K_{S, X_0}}{X_0} \cap \bigcap_{X \in S} \partial \Omega^+(X)$, and $m(X) := -c(X, \Xbar) +c(X_0, \Xbar) +u(X_0)$. We define
    \begin{equation*}
        h := \sup_{X \in S}( m(X) - u(X) ) \geq 0,
    \end{equation*}
    then $m - h \leq u$ on $S$. Moreover by compactness of $S$, there exists a $Y \in S$ such that
\begin{equation*}
    h = m(Y)- u(Y),
\end{equation*}
note we also have $\Xbar \in \partial \Omega^+(Y)$. Suppose $m_0(Y) > u(Y)$,  then $m_0>u$ in some small relatively open neighborhood of $Y$ in $\partial \Omega$. By Remark~\ref{rmk: uni-sized nbhd in proj},  the function $m(\cdot, \beta_Y(\cdot))-h-u(\cdot, \beta_Y(\cdot))$ is well-defined in some ball around $\proj^\Omega_Y(Y)$ which belongs to the relative interior of $\proj^\Omega_Y(S)$ by the above, thus we see the function has a local maximum there. Thus we can see that $[\Xbar]_{Y}=\nabla^\Omega m(Y)\in \locsubdiff{u(\cdot, \beta_Y(\cdot))}{\proj^\Omega_Y(Y)}=\locsubdiff{(u\circ\exp^c_Y)}{0}$. Since $\csubdiff[\Omega]{u}{Y} \subset\partial \Omega^+(Y)$ by assumption, we can then apply Corollary~\ref{cor: local to global} to conclude that $\Xbar \in \csubdiff[\Omega]{u}{Y} \subset \csubdiff[\Omega]{u}{S}$.

Otherwise, we have $m_0(Y) = u(Y)$ and
\begin{equation*}
    0\leq h = m(Y)-u(Y) = m(Y) - m_0(Y)\leq K_{S, X_0}(Y)-m_0(Y) \leq 0.
\end{equation*}
Therefore $h=0$, and in particular we have
\begin{align*}
    m(X)& \leq u(X) \textrm{ on } S, \qquad
    m(X_0) = u(X_0).
\end{align*}
Again since $\csubdiff[\Omega]{u}{X_0}\subset \partial \Omega^+(X_0)$ and $m_0(X_0)>u(X_0)$, we can apply the same argument as we did above with $Y$ to obtain that $\Xbar \in \csubdiff[\Omega]{u}{X_0} \subset \csubdiff[\Omega]{u}{S}$. 
\end{proof}
Next we recall a well-known fact about local subdifferentials, whose proof is provided for completeness.
\begin{lem}
\label{lem: Lip bounds subdiff}
Let $U$ be an $\R$-valued function that is Lipschitz on some neighborhood of a point $x\in \R^n$, with Lipschitz constant $L_U$. If $p\in\locsubdiff{U}{x}$, then $\lvert p\rvert \leq L_U$.
\end{lem}
\begin{proof}
    Without loss, assume $p\neq 0$. Since $p\in\locsubdiff{U}{x}$, we have
    \begin{equation*}
        L_U\geq \frac{U(x+v) - U(x)}{\lvert v\rvert} \geq \inner{p}{\frac{v}{\lvert v\rvert}} + \frac{o(\lvert v\rvert)}{\lvert v\rvert},\qquad v\to 0,
    \end{equation*}
    therefore letting $v= \epsilon \frac{p}{\lvert p\rvert}$ and taking $\epsilon \to 0$, we obtain $\lvert p\rvert \leq L_U$.
\end{proof}
The next lemma, which is the analogue of \cite[Lemma 5.5]{GuillenKitagawa17}, associates to each unit vector $w\in T_{\Xbar_0}\partial\Omega$ a point $\Xbar_w^*\in \partial\Omega$ belonging to $\csubdiff[\Omega]{K_{S, X_0}}{X_0}$ and such that $\proj^\Omega_{X_0}(\Xbar_w^*-\Xbar_0)$ retains information about both the original direction of $w$, and $\distop(x_0,\Pi^w_{[S]_{\Xbar_0}})$.
\begin{lem}\label{lem: csubdiff in small section est}
Suppose $u$ satisfies \eqref{eqn: csub in same side}, and \eqref{eqn: section in same side as slope}. Then, for any unit length $w \in T_{\Xbar_0}\partial\Omega$, and $X_0\in S$ with $m_0(X_0)>u(X_0)$, there exists $\Xbar_w^* \in \partial \Omega^+(X_0)$ such that 
\begin{align}
    \Xbar_w^* &\in \csubdiff[\Omega]{K_{S, X_0}}{X_0}\cap\bigcap_{X\in S}B^{n+1}_{\conerad(\frac{1}{2})}(X),\label{eqn: good slope in ball}\\
    \lvert [\Xbar_w^*]_{X_0} - [\Xbar_0]_{X_0} \rvert &\geq \frac{\conerad(\frac{1}{2})^2(m_0(X_0)-u(X_0))}{(\conerad(\frac{1}{2})^2+16\diam(\Omega)^2)\distop(x_0,\Pi^w_{[S]_{\Xbar_0}})},\label{eqn: csubdiff length est}\\
    \inner{[\Xbar_w^*]_{X_0} - [\Xbar_0]_{X_0}}{\proj^\Omega_{X_0}(w)} &> \frac{\conerad(\frac{1}{2})}{3\sqrt{\conerad(\frac{1}{2})^2+16\diam(\Omega)^2}} \lvert [\Xbar_w^*]_{X_0} - [\Xbar_0]_{X_0} \rvert,\label{eqn: csubdiff direction est}
\end{align} 
where the meaning of the last expression above is that we identify $T_{\Xbar_0}\partial \Omega$ with an $n$-dimensional subspace of $\R^{n+1}$, and $w$ with the corresponding vector in $\R^{n+1}$.
\end{lem}
\begin{proof}
By \eqref{eqn: section in same side as slope} we have $S\subset \partial\Omega^+(\Xbar_0)$, hence $\exp^c_{\Xbar_0}$ is well-defined on $[S]_{\Xbar_0}$ (throughout the proof we will freely use the fact that $\conerad$ is taken to be a decreasing function). Fix a unit vector $w \in  T_{\Xbar_0}\partial\Omega$ and take $x_w \in \partial^n [S]_{\Xbar_0}$ such that $w \in \normal{[S]_{\Xbar_0}}{x_w}$. Let us write $X_w := \cExp{\Xbar_0}{x_w}$ and also define $U_0(x) := u(\cExp{\Xbar_0}{x})-m_0(\cExp{\Xbar_0}{x})$. 
 Noting that $\cExp{\Xbar_0}{\cdot}$ is differentiable on $[\partial \Omega^+(\Xbar_0)]_{\Xbar_0}$, the proof of the second part of Corollary~\ref{cor: local to global} and \cite[Theorem 2.5.1]{clarke90} implies that the local subdifferential and generalized (Clarke) subdifferential of $U_0$ at a point are the same. Therefore, by \cite[Corollary 1 to Theorem 2.4.7,  and Proposition 2.4.4]{clarke90}, we obtain that for some $\lambda>0$, 
\begin{equation*}
\lambda w \in \locsubdiff{U_0}{x_w},
\end{equation*}
i.e. we have the inequality
\begin{equation}
\label{eqn: subdiff of U at x_w}
    U_0(x_w + v) \geq U_0(x_w) + \inner{\lambda w}{v} + o(\lvert v\rvert),\qquad v\to 0.
\end{equation}
Now define $v: [\partial \Omega^+(X_w)]_{X_w}\to [\partial \Omega^+(\Xbar_0)]_{\Xbar_0}-x_w$ by 
\begin{align*}
v(p) :&= [\cExp{X_w}{p}]_{\Xbar_0}-x_w
=\proj^\Omega_{\Xbar_0}(p+\proj^\Omega_{X_w}(X_w)+\beta_{X_w}(p+\proj^\Omega_{X_w}(X_w))\normal{\Omega}{X_w}-\Xbar_0)-x_w\\
&=p+\proj^\Omega_{X_w}(X_w)+\beta_{X_w}(p+\proj^\Omega_{X_w}(X_w))\normal{\Omega}{X_w}-\Xbar_0\\
&-\inner{p+\proj^\Omega_{X_w}(X_w)+\beta_{X_w}(p+\proj^\Omega_{X_w}(X_w))\normal{\Omega}{X_w}-\Xbar_0}{\normal{\Omega}{\Xbar_0}}\normal{\Omega}{\Xbar_0}-x_w,   
\end{align*}
which we can see is differentiable in $p$. Let  $p\in T_0(T_{X_w}\partial\Omega)\cong T_{X_w}\partial\Omega$ with $p\neq 0$, then we calculate using $\nabla^n \beta_{X_w}(\proj_{X_w}^\Omega(X_w))=0$,
\begin{align*}
    v(p)&=v(0)+Dv(0)p+o(\lvert p\rvert)
    =\left.\frac{d}{dt}v(tp)\right\vert_{t=0}+o(\lvert p\rvert)\\
    &=p+\inner{\nabla^n\beta_{X_w}(\proj^\Omega_{X_w}(X_w))}{p}\normal{\Omega}{X_w}\\
    &\quad-\inner{p+\inner{\nabla^n\beta_{X_w}(\proj^\Omega_{X_w}(X_w))}{p}\normal{\Omega}{X_w}}{\normal{\Omega}{\Xbar_0}}\normal{\Omega}{\Xbar_0}+o(\lvert p\rvert)\\
    &=p-\inner{p}{\normal{\Omega}{\Xbar_0}}\normal{\Omega}{\Xbar_0}+o(\lvert p\rvert)
\end{align*}
and combined with \eqref{eqn: derivative cexp} we have
\begin{align}
    U_0(x_w + v(p)) &\geq U_0(x_w) + \inner{\lambda w}{v(p)} + o(\lvert p\rvert),\qquad p\to 0\notag\\
    &=U_0(x_w) + \inner{\lambda w}{p-\inner{p}{\normal{\Omega}{\Xbar_0}}\normal{\Omega}{\Xbar_0} } + o(\lvert p\rvert),\qquad p\to 0\notag\\
    &=U_0(x_w) + \inner{\lambda \proj^\Omega_{X_w}(w)}{p} + o(\lvert p\rvert),\qquad p\to 0.\label{eqn: subdiff of U circ v}
\end{align}
Moreover we compute for $p$ close to $0$
\begin{align*}
    &m_0(\cExp{\Xbar_0}{x_w+v(p)})-m_0(\cExp{\Xbar_0}{x_w})  = -c(\cExp{\Xbar_0}{x_w+v(p)},\Xbar_0)+c(\cExp{\Xbar_0}{x_w},\Xbar_0) \\
    & = -c(\cExp{X_w}{p},\Xbar_0)+c(X_w,\Xbar_0)
    =\inner{\cExp{X_w}{p}-X_w}{\Xbar_0}+\frac{\lvert X_w\rvert^2-\lvert\cExp{X_w}{p}\rvert^2}{2}\\
    &=\inner{p+(\beta_{X_w}(\proj_{X_w}^\Omega(X_w)+p)-\beta_{X_w}(\proj_{X_w}^\Omega(X_w)))\normal{\Omega}{X_w}}{\Xbar_0}-\inner{p}{\proj_{X_w}^\Omega(X_w)}\\
    &+\frac{\beta_{X_w}(\proj_{X_w}^\Omega(X_w))^2-\beta_{X_w}(\proj_{X_w}^\Omega(X_w)+p)^2-\lvert p\rvert^2}{2},
\end{align*}
hence we obtain that 
\begin{equation*}
    m_0(\cExp{\Xbar_0}{x_w+v(p)})-m_0(\cExp{\Xbar_0}{x_w}) = \inner{p}{[\Xbar_0]_{X_w}} + o(\lvert p\rvert),\qquad p\to 0.
\end{equation*}
Therefore, from \eqref{eqn: subdiff of U circ v}, we obtain
\begin{equation*}
    u(\cExp{X_w}{p}) \geq u(\cExp{X_w}{0}) + \inner{[\Xbar_0]_{X_w}-\lambda  \proj^\Omega_{X_w}(w)}{p}+o(\lvert p\rvert),
\end{equation*}
in particular we have
\begin{equation*}
    [\Xbar_0]_{X_w}-\lambda \proj^\Omega_{X_w}(w) \in \locsubdiff{(u\circ \exp^c_{X_w})}{0}.
\end{equation*}
By Corollary~\ref{cor: local to global}, this means 
$[\Xbar_0]_{X_w} - \lambda 
\proj^\Omega_{X_w}(w)\in [\partial\Omega^+(X_w)]_{X_w}$, hence we can define
\begin{equation*}
   \Xbar_w := \cExp{X_w}{ [\Xbar_0]_{X_w} - \lambda 
   \proj^\Omega_{X_w}(w) } \in \partial \Omega^+(X_w),
\end{equation*}
and $\Xbar_w \in \csubdiff[\Omega]{u}{X_w}$. Thus 
\begin{equation*}
    u(X) \geq -c(X,\Xbar_w ) + c(X_w, \Xbar_w) +u(X_w),\  \forall X \in \partial \Omega.
\end{equation*}
On the other hand, by definition of $S$ and since $m_0(X_w) = u(X_w)$, we have
\begin{equation*}\label{eqn: m_0 another formula}
    u(X_0) \leq m_0(X_0) = -c(X_0, \Xbar_0) + c(X_w, \Xbar_0) +u(X_w).
\end{equation*}
Since $X_w\in S$ and $\Xbar_w\in \csubdiff[\Omega]{u}{X_w}$, by \eqref{eqn: section in same side as slope} and \eqref{eqn: csub in same side} we have $X_w\in \partial\Omega^+(\Xbar_0)\cap \partial\Omega^+(\Xbar_w)$, hence we can define a $c$-segment with respect to $X_w$ that connects $\Xbar_0$ and $\Xbar_w$, let $\Xbar_w(t)$ be this $c$-segment. By the intermediate value theorem, there exists some $t^* \in [0,1]$ such that 
\begin{equation}\label{eqn: def Xbarw*}
    u(X_0) = -c(X_0, \Xbar_w(t^*)) + c(X_w, \Xbar_w(t^*)) +u(X_w),
\end{equation}
denote $\Xbar_w^* := \Xbar_w(t^*)$. Then for $X\in S$ , using that $u(X_w)=m_0(X_w)$, Proposition~\ref{prop: qqconv holds},   \eqref{eqn: def Xbarw*}, and that $\Xbar_w\in \csubdiff[\Omega]{u}{X_w}$,  we compute
\begin{align*}
    -c(X, \Xbar_w^*)+c(X_0, \Xbar_w^*)
    &\leq -c(X, \Xbar_0)+c(X_w, \Xbar_0)-c(X_w, \Xbar_w^*)+c(X_0, \Xbar_w^*)\\
    &+t^*(-c(X, \Xbar_w)+c(X_w, \Xbar_w)+c(X, \Xbar_0)-c(X_w, \Xbar_0))\\
    &=-c(X, \Xbar_0)+c(X_w, \Xbar_0)+u(X_w)-u(X_0)\\
    &+t^*(-c(X, \Xbar_w)+c(X_w, \Xbar_w)+u(X_w)-m_0(X))\\
    &\leq -c(X, \Xbar_0)+c(X_w, \Xbar_0)+u(X_w)-u(X_0)+t^*(u(X)-m_0(X))\\
    &\leq -c(X, \Xbar_0)+c(X_w, \Xbar_0)+m_0(X_w)-u(X_0)=m_0(X)-u(X_0).
\end{align*}
Thus from Remark~\ref{rmk: csubdiff of c cone} we find that $\Xbar_w^* \in \csubdiff[\Omega]{K_{S, X_0}}{X_0}$.

Now we can see
\begin{align*}
   \lvert \Xbar_w(t)-X_w\rvert^2&=\lvert  [\Xbar_0]_{X_w}-t\lambda 
   \proj^\Omega_{X_w}(w)\rvert^2\\
   &+\left(\beta_{X_w}(\proj_{X_w}(X_w)+[\Xbar_0]_{X_w}-t\lambda 
   \proj^\Omega_{X_w}(w))-\beta_{X_w}(\proj_{X_w}(X_w))\right)^2
\end{align*}
is the sum of a quadratic function and the square of a negative, concave function, hence is convex in $t$. In particular, since 
    $\Xbar_w^* = \cExp{X_w}{[\Xbar_0]_{X_w}-t^*\lambda 
    \proj^\Omega_{X_w}(w)}$, 
we obtain
\begin{align*}
    \lvert \Xbar_w^*-X_w\rvert\leq\max( \lvert \Xbar_w-X_w\rvert, \lvert \Xbar_0-X_w\rvert)<\frac{\conerad(\frac{1}{2})}{4}
\end{align*}
by \eqref{eqn: csub in same side} and \eqref{eqn: section in same side as slope}. 
Thus for any $Y \in S$, using \eqref{eqn: section in same side as slope} we have
\begin{align*}
    \lvert \Xbar_w^* - Y \rvert \leq \lvert \Xbar_w^* - X_w\rvert+\lvert X_w-\Xbar_0 \rvert + \lvert \Xbar_0 - Y \rvert  < \conerad(\frac{1}{2}),
\end{align*}
hence we obtain \eqref{eqn: good slope in ball}.

Now, we estimate the size of $[\Xbar_w^*]_{X_0} - [\Xbar_0]_{X_0}=\proj^\Omega_{X_0}(\Xbar_w^*-\Xbar_0)$. 
Let 
\begin{align*}
    x_0 :&= [X_0]_{\Xbar_0},\qquad
    m_w(X) := -c(X,\Xbar_w^*) + c(X_w, \Xbar_w^*) +u(X_w),\\
    [S_w] :&= \{ x \in [\partial\Omega^+(\Xbar_0)]_{\Xbar_0} \mid m_0(\cExp{\Xbar_0}{x}) \geq m_w(\cExp{\Xbar_0}{x})\}.
\end{align*}
Note that $x_0 \in [S]_{\Xbar_0} \subset [S_w] \neq \emptyset$. Since $m_w$ is a $c$-convex function, by Corollary \ref{cor: csubdiff c-convex} (1) we have that $[S_w]$ is convex.

Since $m_0(X_w)=u(X_w)=m_w(X_w)$ by construction, we have $x_w\in \partial^n[S_w]$. Then by taking the derivative of $m_0(\cExp{\Xbar_0}{x}) - m_w(\cExp{\Xbar_0}{x})$ at $x = x_w$, we observe that $w\in\normal{[S_w]}{x_w}$ and  $\Pi^w_{[S]_{\Xbar_0}}$ is an $(n-1)$-dimensional supporting plane of $[S_w]$ at $x_w$. Then there exists a point $x_1 \in \{ x_0+tw \mid t\geq 0 \} \cap \partial^n [S_w]$ such that $ \inner{x_1-x_w}{w} \leq 0$. First suppose $x_1\in [\partial\Omega^+(\Xbar_0)]_{\Xbar_0}$. Thanks to \eqref{eqn: section in same side as slope}, we have $\lvert x_0-\proj^\Omega_{\Xbar_0}(\Xbar_0)\lvert \leq \vert X_0-\Xbar_0\rvert<\frac{\conerad(\frac{1}{2})}{4}$ 
hence letting $X_1 := \cExp{\Xbar_0}{x_1}$ and using Lemma~\ref{lem: unif size nbhd lip beta} we compute  
\begin{align*}
    -c(X_1, \Xbar_w^*) + c(X_0,\Xbar_w^*) + c(X_1,\Xbar_0) - c(X_0, \Xbar_0) 
   & =  \inner{\cExp{\Xbar_0}{x_1}-\cExp{\Xbar_0}{x_0}}{\Xbar_w^* - \Xbar_0} \\
   & \leq  \sqrt{\lvert x_1 -x_0\rvert^2+(\beta_{\Xbar_0}(x_1)-\beta_{\Xbar_0}(x_0))^2} \lvert \Xbar_w^* - \Xbar_0 \rvert \\
   & \leq  \frac{\sqrt{\conerad(\frac{1}{2})^2+16\diam(\Omega)^2}}{\conerad(\frac{1}{2})} \lvert x_1 -x_0\rvert\lvert \Xbar_w^* - \Xbar_0 \rvert\\
   &\leq  \frac{\sqrt{\conerad(\frac{1}{2})^2+16\diam(\Omega)^2}}{\conerad(\frac{1}{2})}\lvert \Xbar_w^* - \Xbar_0 \rvert \distop(x_0,\Pi^w_{[S]_{\Xbar_0}}).
\end{align*}
On the other hand, since $x_1 \in \partial^n [S_w]$ and $m_w\leq u$ everywhere we obtain
\begin{align*}
    -c(X_1, \Xbar_w^*) + c(X_0,\Xbar_w^*) + c(X_1,\Xbar_0) - c(X_0, \Xbar_0) &= m_w(X_1)-m_w(X_0)-m_0(X_1)+m_0(X_0)\\
    &=m_0(X_0)-m_w(X_0)\geq m_0(X_0)-u(X_0).
\end{align*}
Also by \eqref{eqn: section in same side as slope} we have $\lvert \proj^\Omega_{X_0}(\Xbar_0)-\proj^\Omega_{X_0}(X_0)\lvert <\frac{\conerad(\frac{1}{2})}{4}$, thus by Lemma~\ref{lem: unif size nbhd lip beta} we have
\begin{align}
    \lvert \Xbar_w^* - \Xbar_0 \rvert&= \sqrt{\lvert \proj^\Omega_{X_0}(\Xbar_w^*)-\proj^\Omega_{X_0}(\Xbar_0)\rvert^2+(\beta_{X_0}(\proj^\Omega_{X_0}(\Xbar_w^*))-\beta_{X_0}(\proj^\Omega_{X_0}(\Xbar_0)))^2}  \notag\\
   & \leq  \frac{\sqrt{\conerad(\frac{1}{2})^2+16\diam(\Omega)^2}}{\conerad(\frac{1}{2})} \lvert \proj^\Omega_{X_0}(\Xbar_w^*) -\proj^\Omega_{X_0}(\Xbar_0)\rvert\notag\\
   &=\frac{\sqrt{\conerad(\frac{1}{2})^2+16\diam(\Omega)^2}}{\conerad(\frac{1}{2})}\lvert [\Xbar_w^*]_{X_0}-[\Xbar_0]_{X_0}\rvert.\label{eqn: support difference bound}
\end{align}
Therefore, combining the above yields \eqref{eqn: csubdiff length est}:
\begin{equation*}
   \lvert [\Xbar_w^*]_{X_0} - [\Xbar_0 ]_{X_0} \rvert \geq \frac{\conerad(\frac{1}{2})\lvert \Xbar_w^*  - \Xbar_0 \rvert}{\sqrt{\conerad(\frac{1}{2})^2+16\diam(\Omega)^2}} \geq \frac{\conerad(\frac{1}{2})^2(m_0(X_0)-u(X_0))}{(\conerad(\frac{1}{2})^2+16\diam(\Omega)^2)\distop(x_0,\Pi^w_{[S]_{\Xbar_0}})}.
\end{equation*}
If $x_1=x_0+t_0w\in \partial^n[\partial\Omega^+(\Xbar_0)]_{\Xbar_0}$, take a sequence $t_k\nearrow t_0$, since $x_0$ is in the interior of $[\partial\Omega^+(\Xbar_0)]_{\Xbar_0}$ we have $x_0+t_kw\in [\partial\Omega^+(\Xbar_0)]_{\Xbar_0}$. We can then follow the above proof with this sequence of points replacing $x_1$ and take a limit to obtain \eqref{eqn: csubdiff length est}.

Now we estimate the direction of $[\Xbar_w^*]_{X_0}-[\Xbar_0]_{X_0}$. Here, recall we identify the tangent plane $ T_{\Xbar_0}\partial \Omega$ with an $n$-dimensional subspace of $\R^{n+1}$ and consider $w$ as living in $\R^{n+1}$. Let $\xbar_w(t) := [\Xbar_0]_{X_w}-t\lambda 
\proj^\Omega_{X_w}(w)=[\Xbar_w(t)]_{X_w}$, then we can compute
\begin{align*}
    &\inner{[\Xbar_w^*]_{X_0} - [\Xbar_0]_{X_0}}{\proj^\Omega_{X_0}(w)} \\
    &= \inner{\Xbar_w^* - \Xbar_0}{\proj^\Omega_{X_0}(w)} 
     = \inner{\cExp{X_w}{\xbar_w(t^*)}-\cExp{X_w}{\xbar_w(0)}}{\proj^\Omega_{X_0}(w)} \\
    & = \lambda \int_0^{t^*} \inner {D \cExp{X_w}{\xbar_w(t)} 
    \proj^\Omega_{X_w}(w)
    }{\proj^\Omega_{X_0}(w)} dt\\
     & = \lambda \int_0^{t^*} \inner {\proj^\Omega_{X_w}(w)+\inner{\nabla^n\beta_{X_w}(\proj^\Omega_{X_w}(X_w)+\xbar_w(t))}{\proj^\Omega_{X_w}(w)}\normal{\Omega}{X_w}}{\proj^\Omega_{X_0}(w)} dt,
\end{align*}
where we have used  
\eqref{eqn: derivative cexp} to obtain the final line. Now we find
\begin{align*}
    & \lambda \inner {\inner{\nabla^n\beta_{X_w}(\proj^\Omega_{X_w}(X_w)+\xbar_w(t))}{\proj^\Omega_{X_w}(w)}\normal{\Omega}{X_w}}{\proj^\Omega_{X_0}(w)}\\
     & \geq - \lvert \inner{\nabla^n\beta_{X_w}(\proj^\Omega_{X_w}(X_w)+\xbar_w(t))}{\lambda\proj^\Omega_{X_w}(w)}\rvert \\
     & = -\lvert \frac{d}{dt}\beta_{X_w}(\proj^\Omega_{X_w}(X_w)+\xbar_w(t))\rvert\\
     &\geq -\max\left(\lvert \left.\frac{d}{dt}\beta_{X_w}(\proj^\Omega_{X_w}(X_w)+\xbar_w(t))\right\vert_{t=0}\rvert, \lvert \left.\frac{d}{dt}\beta_{X_w}(\proj^\Omega_{X_w}(X_w)+\xbar_w(t))\right\vert_{t=1}\rvert\right),
\end{align*}
where to obtain the fourth line we use that $t\mapsto \beta_{X_w}(\proj^\Omega_{X_w}(X_w)+\xbar_w(t))$ is  a concave function, hence has decreasing derivative. Continuing the calculation, by \eqref{eqn: csub in same side} and \eqref{eqn: section in same side as slope}, we have $\Xbar_w$, $\Xbar_0\in B^{n+1}_{\conerad(\frac{35}{36})}(X_w)$ 
hence by Lemma~\ref{lem: gradient beta est},
\begin{align*}
    &\max\left(\lvert \left.\frac{d}{dt}\beta_{X_w}(\proj^\Omega_{X_w}(X_w)+\xbar_w(t))\right\vert_{t=0}\rvert, \lvert \left.\frac{d}{dt}\beta_{X_w}(\proj^\Omega_{X_w}(X_w)+\xbar_w(t))\right\vert_{t=1}\rvert\right)\\
    &\leq \lambda\max\left(\lvert \nabla^n\beta_{X_w}(\proj^\Omega_{X_w}(X_w)+\xbar_w(0))\rvert, \lvert \nabla^n\beta_{X_w}(\proj^\Omega_{X_w}(X_w)+\xbar_w(1))\rvert\right)\\
    &\leq \frac{\lambda\sqrt{1-\left(\frac{35}{36}\right)^2}}{\frac{35}{36}}
    \leq \frac{\lambda}{3},
\end{align*}
thus 
\begin{align}\label{eqn: inner product estimate}
    \inner{[\Xbar_w^*]_{X_0} - [\Xbar_0]_{X_0}}{\proj^\Omega_{X_0}(w)}\geq t^*\lambda\left(\inner {\proj^\Omega_{X_w}(w)}{\proj^\Omega_{X_0}(w)}-\frac{1}{3}\right).
\end{align}
Now we calculate 
\begin{align*}
    &\inner {\proj^\Omega_{X_w}(w)}{\proj^\Omega_{X_0}(w)}=\inner{w-\inner{w}{\normal{\Omega}{X_w}}\normal{\Omega}{X_w}}{w-\inner{w}{\normal{\Omega}{X_0}}\normal{\Omega}{X_0}}\\
    &=1-\inner{w}{\normal{\Omega}{X_w}}^2-\inner{w}{\normal{\Omega}{X_0}}^2-\inner{w}{\normal{\Omega}{X_w}}\inner{w}{\normal{\Omega}{X_0}}\inner{\normal{\Omega}{X_w}}{\normal{\Omega}{X_0}}.
\end{align*}
Since $w\in T_{\Xbar_0}\partial \Omega$ and is unit length, we have $\lvert w-\normal{\Omega}{\Xbar_0} \rvert = \sqrt{2}$, also by \eqref{eqn: section in same side as slope} and the definition of $\conerad$, we have $\lvert \normal{\Omega}{X_w}-\normal{\Omega}{\Xbar_0}\rvert\leq \sqrt{2-2\cdot\frac{35}{36}}$. Hence we obtain
\begin{equation*}
    \lvert w-\normal{\Omega}{X_w} \rvert \geq \lvert w-\normal{\Omega}{\Xbar_0} \rvert - \lvert \normal{\Omega}{\Xbar_0} - \normal{\Omega}{X_w} \rvert > \sqrt2 - \sqrt{2-2\cdot\frac{35}{36}} .
\end{equation*}
so that
\begin{align*}
    \inner{w}{\normal{\Omega}{X_w}}<2\sqrt{1-\frac{35}{36}}-\left(1-\frac{35}{36}\right)\leq 2\sqrt{1-\frac{35}{36}}.
\end{align*}
Applying \eqref{eqn: section in same side as slope} to $X_0$ shows $\lvert X_0-\Xbar_0\rvert<\conerad(\frac{35}{36})$, thus the same calculation as above yields $\inner{w}{\normal{\Omega}{X_0}}<2\sqrt{1-\frac{35}{36}}$. Then we compute
\begin{align*}
    \inner {\proj^\Omega_{X_w}(w)}{\proj^\Omega_{X_0}(w)}\geq 1-12\left(1-\frac{35}{36}\right)=\frac{2}{3}
\end{align*}
and combining with \eqref{eqn: inner product estimate} we obtain
\begin{equation}\label{eqn: direction bdd below t*lambda}
    \inner{[\Xbar_w^*]_{X_0} - [\Xbar_0]_{X_0}}{\proj^\Omega_{X_0}(w)} >  \frac{t^*\lambda}{3}.
\end{equation}
On the other hand, again by \eqref{eqn: section in same side as slope} we have $\Xbar_0\in B^{n+1}_{\frac{\conerad(\frac{1}{2})}{4}}(X_w)$, hence by Lemma~\ref{lem: unif size nbhd lip beta} and a calculation analogous to \eqref{eqn: support difference bound} with $\beta_{X_w}$ replacing $\beta_{X_0}$, we obtain
\begin{align*}
    \lvert [\Xbar_w^*]_{X_0} - [\Xbar_0]_{X_0} \rvert & = \lvert \proj^\Omega_{X_0}(\Xbar_w^* - \Xbar_0) \rvert 
    \leq \lvert \Xbar_w^*-\Xbar_0\rvert\\
    &=\frac{\sqrt{\conerad(\frac{1}{2})^2+16\diam(\Omega)^2}}{\conerad(\frac{1}{2})} \lvert \proj^\Omega_{X_w}(\Xbar_w^*) -\proj^\Omega_{X_w}(\Xbar_0)\rvert\\
    & =\frac{\sqrt{\conerad(\frac{1}{2})^2+16\diam(\Omega)^2}}{\conerad(\frac{1}{2})} \lvert -t^*\lambda 
    \proj^\Omega_{X_w}(w)\rvert \\
    & \leq \frac{\sqrt{\conerad(\frac{1}{2})^2+16\diam(\Omega)^2}}{\conerad(\frac{1}{2})}t^*\lambda.
\end{align*}
Finally combining with \eqref{eqn: direction bdd below t*lambda} we obtain \eqref{eqn: csubdiff direction est}.
\end{proof}

Recall the projected section $[S]_{\Xbar_0}$ is convex by Corollary~\ref{cor: csubdiff c-convex}. Then there is an ellipsoid $\mathcal{E}\subset T_{\Xbar_0}\partial\Omega$ called the John ellipsoid such that
\begin{align}\label{eqn: john ellipsoid}
    \mathcal{E} \subset [S]_{\Xbar_0} \subset n \mathcal{E},
\end{align}
where the dilation is with respect to the center of mass of $\mathcal{E}$ (see \cite{deGuzman76}).

Fix a unit length $w_1 \in T_{\Xbar_0}\partial \Omega$. Take an orthonormal basis $\{ e_i \}_1^{n}$ which consist of directions parallel to the axial directions of $\mathcal{E}$ and are such that $\inner{w_1}{e_i} \geq 0$ for all $1\leq i \leq n$. Reordering, we may assume $\max_i \inner{w_1}{e_i} = \inner{w_1}{e_1} \geq \frac{1}{\sqrt{n}}$ and let $w_i = e_i$ for $2 \leq i \leq n$, then $\{ w_i \}_{i=1}^{n}$ forms a basis of $T_{\Xbar_0}\partial\Omega$. We now apply Lemma~\ref{lem: csubdiff in small section est} to this collection of vectors to estimate the $\H^n$-volume of $\csubdiff[\Omega]{u}{S}$ from below.
\begin{lem}\label{lem: csubdiff cone est}
   Suppose $u$ satisfies \eqref{eqn: csub in same side}, and \eqref{eqn: section in same side as slope}, and $X_0\in S$ is such that $m_0(X_0)>u(X_0)$. Then for the above choice of $\{ w_i \}_{i=1}^{n}$, there exists a constant $C_K>0$ that depends on $\Omega$ such that
    \begin{equation*}
        C_K \H^n(\csubdiff[\Omega]{u}{S}) \geq (m_0(X_0)-u(X_0))^{n} \prod_{i=1}^n \frac{1}{\distop([X_0]_{\Xbar_0}, \Pi^{w_i}_{[S]_{\Xbar_0}})}
    \end{equation*}
\end{lem}
\begin{proof}
    Let $\Xbar_i = \Xbar_{w_i}^*$ be the points that we obtain by applying Lemma~\ref{lem: csubdiff in small section est} to $w_i$ for $0\leq i \leq n$ and denote $\xbar_i = \proj_{X_0}(\Xbar_i)$, $1 \leq i \leq n$. Since each $\Xbar_i\in \partial\Omega^+(X_0)$, by Corollary~\ref{cor: csubdiff c-convex} we obtain
    \begin{equation*}
        \ch\{ \xbar_i \mid i = 0, \cdots, n \} \subset [\csubdiff[\Omega]{K_{S, X_0}}{X_0}]_{X_0}.
    \end{equation*}
Now for $\vec{\lambda}:=(\lambda_0, \ldots, \lambda_n)$ with $\sum_{i=0}^n\lambda_i=1$ and $\lambda_i\geq 0$, write
\begin{align*}
    \Xbar(\vec\lambda):=\cExp{X_0}{\sum_{i=0}^n \lambda_i \xbar_i},
\end{align*}
which we see is well-defined by convexity of $\Omega_{X_0}$. Writing $x_0:=[X_0]_{X_0}$,
\begin{align*}
    \lvert\Xbar(\vec\lambda)-X_0\rvert^2=\lvert \sum_{i=0}^n \lambda_i \xbar_i\rvert^2+(\beta_{X_0}(x_0+\sum_{i=0}^n \lambda_i \xbar_i)-\beta_{X_0}(x_0))^2.
\end{align*}
Again the above is  a convex function in $\vec\lambda$ by concavity and negativity of $\beta_{X_0}-\beta_{X_0}(x_0)$, hence it attains its maximum at one of the extremal points of the unit simplex, yielding by \eqref{eqn: good slope in ball},
\begin{align*}
    \lvert\Xbar(\vec\lambda)-X_0\rvert^2\leq \max_{i=0, \ldots, n}\lvert\Xbar_i-X_0\rvert^2\leq \conerad(\frac{1}{2})^2,
\end{align*}
thus $\Xbar(\vec\lambda)\in \partial\Omega^+_{\frac{1}{2}}(X_0)$. For any $Y\in S$, by \eqref{eqn: section in same side as slope} we have $\lvert Y-X_0\rvert\leq \lvert Y-\Xbar_0\rvert+\lvert X_0-\Xbar_0\rvert<\conerad(\frac{35}{36})$ which implies $Y\in \partial\Omega^+_{\frac{35}{36}}(X_0)$, hence
\begin{align*}
    \lvert \normal{\Omega}{\Xbar(\vec\lambda)}-\normal{\Omega}{Y}\rvert^2&\leq (\lvert \normal{\Omega}{\Xbar(\vec\lambda)}-\normal{\Omega}{X_0}\rvert+\lvert \normal{\Omega}{X_0}-\normal{\Omega}{Y}\rvert)^2\\
    &\leq (1+\sqrt{2-\frac{35}{18}})^2
    =\frac{19+6\sqrt{2}}{18}
\end{align*}
thus
\begin{align*}
    \inner{\normal{\Omega}{\Xbar(\vec\lambda)}}{\normal{\Omega}{Y}}
    &\geq 1-\frac{19+6\sqrt{2}}{36}>0,
\end{align*}
meaning $\cExp{X_0}{\ch\{ x_i \mid i = 0, \cdots, n \}} \subset \partial \Omega^+(Y)$ for any $Y \in S$. Therefore, we can apply Lemma~\ref{lem: csubdiff cone in csubdiff u } and obtain $\cExp{X_0}{\ch\{ \xbar_i \mid i = 0, \cdots, n \}} \subset \csubdiff[\Omega]{u}{S}$, in particular, we have
    \begin{equation*}
        \H^n(\cExp{X_0}{\ch\{ \xbar_i \mid i = 0, \cdots, n \}}) \leq \H^n(\csubdiff[\Omega]{u}{S}).
    \end{equation*}
    With our choice of $\{ w_i \}_1^n$, \eqref{eqn: csubdiff direction est} from Lemma~\ref{lem: csubdiff in small section est} implies that $\{ x_i - x_0 \}_1^n$ spans a parallelepiped which is non-degenerate and has $\H^n$-volume comparable to $\prod_{i=1}^n \lvert \xbar_i - \xbar_0 \rvert$, with constant depending on $\Omega$. Then we use \eqref{eqn: csubdiff length est} to obtain the desired inequality.
\end{proof}
A standard argument, which we present here for completeness, gives a lower bound on the $\H^n$-volume of $S$.
\begin{lem}\label{lem: |S| lower est}
   Suppose $u$ satisfies \eqref{eqn: csub in same side}, and \eqref{eqn: section in same side as slope}, and let $\{ w_i \}_1^n$ be as in Lemma~\ref{lem: csubdiff cone est}. Then there exists a constant $C_0$ that depends only on $n$ such that
    \begin{equation*}
        C_0 \H^n(S) \geq l([S]_{\Xbar_0}, w_1) \prod_{i=2}^n d \left( \Pi^{w_i}_{[S]_{\Xbar_0}}, \Pi^{-w_i}_{[S]_{\Xbar_0}}\right).
    \end{equation*}
\end{lem}
\begin{proof}
    By \eqref{eqn: section in same side as slope} we can calculate
    \begin{equation}\label{eqn: |S| bounded below}
    \H^n(S)\geq \int_{[S]_{\Xbar_0}}\sqrt{1+\lvert \nabla^n\beta_{\Xbar_0}\rvert^2}\geq \H^n([S]_{\Xbar_0})\geq \H^n(\mathcal{E})\geq C\prod_{i=1}^n d \left( \Pi^{e_i}_{[S]_{\Xbar_0}}, \Pi^{-e_i}_{[S]_{\Xbar_0}}\right)
    \end{equation}
    for some constant $C$ that only depends on $n$. Here we have used that $e_i$ are the axial directions of the John ellipsoid $\mathcal{E}$ of $[S]_{\Xbar_0}$. Let $[y_1,y_2]$ be the points in $[S]_{\Xbar_0}$ that achieve $l([S]_{\Xbar_0}, w_1)$. Then we compute
    \begin{equation}\label{eqn: dist in e1 bounded below}
        d \left( \Pi^{e_1}_{[S]_{\Xbar_0}}, \Pi^{-e_1}_{[S]_{\Xbar_0}}\right) \geq \inner{y_1 - y_2}{e_1} =|y_1 - y_2 | \inner{w_1}{e_1} \geq \frac{ l([S]_{\Xbar_0}, w_1)}{\sqrt{n}},
    \end{equation}
    and we combine \eqref{eqn: |S| bounded below} and \eqref{eqn: dist in e1 bounded below} to obtain
    \begin{equation*}
    \H^n(S) \geq \frac{C}{\sqrt{n}} l([S]_{\Xbar_0}, w_1) \prod_{i=2}^n d \left( \Pi^{w_i}_{[S]_{\Xbar_0}}, \Pi^{-w_i}_{[S]_{\Xbar_0}}\right).
    \end{equation*}
\end{proof}
Finally combining the lemmas above gives the upper Aleksandrov estimate.
\begin{proof}[Proof of Theorem~\ref{thm: upper aleksandrov}]
    Multiplying the inequalities from Lemma~\ref{lem: csubdiff cone est} and Lemma~\ref{lem: |S| lower est}, we obtain that for some constant $C>0$ that depends only on $\Omega$,
    \begin{align*}
        C \H^n(\csubdiff[\Omega]{u}{X_0})\H^n(S) & \geq \frac{l([S]_{\Xbar_0},w_1)}{d \left( [X_0]_{\Xbar_0}, \Pi^{w_1}_{[S]_{\Xbar_0}} \right) } \prod_{i=2}^n \frac{d\left(\Pi^{w_i}_{[S]_{\Xbar_0}}, \Pi^{-w_i}_{[S]_{\Xbar_0}} \right)}{d \left( [X_0]_{\Xbar_0} , \Pi^{w_i}_{[S]_{\Xbar_0}}\right)} (m_0(X_0)-u(X_0))^n\\
        & \geq \frac{l([S]_{\Xbar_0},w_1)}{d \left( [X_0]_{\Xbar_0}, \Pi^{w_1}_{[S]_{\Xbar_0}} \right) } (m_0(X_0)-u(X_0))^n.
    \end{align*}
\end{proof}

\section{Approximation}\label{section: approximation}
To prove Theorem~\ref{thm: nonstrictly convex body}, we must show that some dual potential for our transport problem satisfies the key conditions \eqref{eqn: lower est section in same side}, \eqref{eqn: lower est section in center ball}, and \eqref{eqn: lower est csub in same side} from Theorem~\ref{thm: lower aleksandrov}, and \eqref{eqn: csub in same side} and \eqref{eqn: section in same side as slope} from Theorem~\ref{thm: upper aleksandrov}, at least for sections of sufficiently small height. To do so, we will need to take a uniformly convex sequence of approximating convex bodies and apply the key estimate Proposition~\ref{prop: stay away} combined with Proposition~\ref{prop: nonsplitting} below. Again the key feature of Proposition~\ref{prop: stay away} is that the constants in the estimate depend only on $C^1$ quantities of the body $\Omega$, which will be stable under approximation in Hausdorff distance due to convexity of the domains.
\begin{prop}[Stay-away estimate with splitting]\label{prop: stay away}
 Suppose we have probability measures  $\mu=\rho d\mathcal{H}^n$ and $\mubar$,  with $\rho$ bounded away from zero, and let $(u, u^c)$ be a maximizing pair in the dual problem \eqref{eqn: kantorovich dual} between $\mu$ and $\mubar$. Also suppose $\Theta\in [\frac{1}{2}, 1)$ and let $\rho_0>0$ satisfy $\rho_0\leq \essinf_{\partial \Omega}\rho$. Then for any $\Xbar_0\in \csubdiff[\Omega]{u}{X_0}\cap \partial \Omega^+(X_0)$, it holds that
\begin{align*}
\abs{\proj^\Omega_{X_0}(X_0-\Xbar_0)}\leq C_{\rho_0, \Omega}\W_2(\mu, \mubar)^{\frac{2}{n+2}}
\end{align*}
where
\begin{align*}
C_{\rho_0, \Omega} 
:&= \left(\frac{2^{n-2}\rho_0 \mathcal{H}^{n-2}(\mathbb{S}^{n-2})}{n(n+1)}\right.\\
&\left.\cdot \min \left( \frac{\conerad(\Theta)^2}{48\diam(\Omega)^2+2\conerad(\Theta)^2}, \frac{\conerad(\Theta)}{8\diam(\Omega)}, \frac{1}{16} \right)^n  \int_0^{\frac{\pi}{2}} \cos^{n+2} \phi \sin^{n-2} \phi d \phi\right)^{-\frac{1}{n+1}}.
\end{align*}
%
\end{prop}

\begin{proof}
Fix some $\Xbar_0\in \csubdiff[\Omega]{u}{X_0}\cap \partial \Omega^+(X_0)$. 
For ease of notation, let us write $\beta$ for $\beta_{X_0}$, then by a rotation and translation of coordinates we may assume that $\mathcal{N}_{\Omega}(X_0)=e_{n+1}$ to identify $\Pi^\Omega_{X_0}$ with $\mathbb{R}^n$, that $X_0$ is the origin so $\beta\leq\beta(0)=0$, and that $\Xbar_0 = (ae_1, \beta(ae_1))$ for some $a \geq 0$.

Now, defining
\begin{align*}
B_a :&= B^{n+1}_{\frac{\sqrt{a^2+\beta(ae_1)^2}}{2}}(\left(\frac{ae_1}{2}, \frac{\beta(ae_1)}{2}\right)),\\
K_{a, \kappa} :&= \{ (x, \beta(x))\mid x \in B^n_{\kappa a}(\kappa ae_1) \} \subset \partial \Omega
\end{align*}
we claim that
\begin{equation*}
K_{a, \kappa} \subset B_a
\end{equation*}
for some constant $0<\kappa<\frac{\conerad(\Theta)}{4\diam(\Omega)}$. By Remark~\ref{rmk: uni-sized nbhd in proj}, for such $\kappa$ we will have $B^n_{\kappa a}(\kappa ae_1)\subset \Omega_{X_0}$, hence $\beta$ is defined on all of $B^n_{\kappa a}(\kappa ae_1)$. Let $(x,\beta(x)) \in K_{a, \kappa}$, then we must verify 
\begin{equation*}
|x-\frac{ae_1}{2}|^2 + |\beta(x)-\frac{\beta(ae_1)}{2}|^2 < \frac{a^2}{4} + \frac{\beta(ae_1)^2}{4},
\end{equation*}
which is equivalent to 
\begin{equation*}
|x|^2 - a\inner{x}{e_1} + \beta(x)^2 - \beta(x)\beta(ae_1) <0.
\end{equation*}
Since $x \in B^n_{\kappa a}(\kappa ae_1)$ we have $|x-\kappa ae_1|^2< \kappa^2a^2$, and therefore the above  is true if 
\begin{align}\label{eqn: K_a subset B_a}
0&>|x|^2-a\inner{x}{e_1} + \beta(x)^2 -\beta(x)\beta(ae_1) +\kappa^2a^2 - |x-\kappa ae_1|^2\notag\\
&=a(2\kappa-1)\inner{x}{e_1}+\beta(x)(\beta(x)-\beta(ae_1)).
\end{align}
Let $x^\perp$ be such that $x = \inner{x}{e_1} e_1 + x^\perp$. As long as we take $\kappa \leq \min\left(\frac{\conerad(\Theta)}{8\diam(\Omega)}, \frac{1}{16}\right)$ it will hold that $B^n_{\kappa a}(\kappa ae_1) \subset B^n_{\frac{\conerad(\Theta)}{4}}(0)$ and $0\leq \inner{x}{e_1}\leq 2\kappa a<a$. By \eqref{eqn: Lip eqn beta} with $x_1=0$, $x_2=ae_1$,  we have
\begin{equation*}
\beta(ae_1) \geq -L_\Theta a,
\end{equation*}
then using \eqref{eqn: Lip eqn beta} this time with $x_1=x$, $x_2=\inner{x}{e_1}e_1$, along with concavity of $\beta$,
\begin{align*}
0&\geq\beta(x)  \geq \beta(\inner{x}{e_1}e_1) - L_\Theta|x^\perp| 
 \geq \frac{\inner{x}{e_1}}{a} \beta(ae_1) -L_\Theta|x^\perp| 
 \geq - L_\Theta\inner{x}{e_1}- L_\Theta|x^\perp|.
\end{align*}
Using the third inequality above, we also find
\begin{align*}
    \beta(x)-\beta(ae_1)&\geq \left(\frac{\inner{x}{e_1}}{a}-1\right)\beta(ae_1)-L_\Theta\abs{x^\perp}\geq -L_\Theta\abs{x^\perp},
\end{align*}
thus combining the above inequalities we see to obtain \eqref{eqn: K_a subset B_a} it is sufficient to show
\begin{equation*}
a(2\kappa-1)\inner{x}{e_1} + L_\Theta|x^\perp| \left( L_\Theta\inner{x}{e_1}+ L_\Theta|x^\perp| \right)<0.
\end{equation*}
Since $x \in B^n_{\kappa a}(\kappa ae_1)$, we have $|x^\perp|^2 < 2\kappa a\inner{x}{e_1}-\inner{x}{e_1}^2$ and $|x^\perp|<\kappa a$, therefore 
\begin{align*}
& a(2\kappa-1)\inner{x}{e_1} + L_\Theta|x^\perp| \left(L_\Theta\inner{x}{e_1} + L_\Theta|x^\perp| \right) 
 < a\inner{x}{e_1} (\kappa(3L_\Theta^2 +2)-1) - L_\Theta^2\inner{x}{e_1}^2.
\end{align*}
Since $\inner{x}{e_1}\geq 0$, the expression on the right is nonpositive for $\inner{x}{e_1} \geq \frac{a(\kappa(3L_\Theta^2+2)-1)}{L_\Theta^2}$, thus we obtain $K_{a, \kappa} \subset B_a$ by choosing $$\kappa := \min \left( \frac{1}{3L_\Theta^2+2}, \frac{\conerad(\Theta)}{8\diam(\Omega)}, \frac{1}{16} \right).$$

Since $\partial_c^\Omega u$ is closed by continuity of $c$, from \cite[Theorem 5.10 (iii)]{Villani09} we see that there exists at least one Kantorovich solution between $\mu$ and $\mubar$ satisfying $\spt\gamma\subset \partial_c^\Omega u$.  Let $(X,\Xbar) \in \spt\gamma$ with $X = (x,\beta(x)) \in K_{a, \kappa}$. By Remark~\ref{rmk: subdifferentials} we see that $\partial^\Omega_c u$ is a subset of the subdifferential of a convex function, hence by Rockafellar's theorem (see \cite[Theorem 24.8]{Rockafellar70}) it is monotone. In particular the point $\Xbar$ must lie in the half space $\{ \Ybar \mid \inner{\Ybar-(ae_1, \beta(ae_1))}{X} \geq 0 \}$, which does not contain $X$ in the interior (as $X \in B_a$), hence the distance from $X$ to $\Xbar$ is bounded from below by the orthogonal distance from $X$ to the boundary of this halfspace. Hence for such $(X, \Xbar)$ we have the estimate
\begin{align*}
|X-\Xbar| & \geq \inner{(ae_1, \beta(ae_1))}{ \frac{X}{|X|}}-|X| \\
& = \frac{1}{|X|} ( a\inner{x}{e_1} + \beta(ae_1) \beta(x)) - |X| \\
& \geq \frac{a \inner{x}{e_1} }{|X|} - |X| 
 \geq \frac{a}{2} \inner{\frac{x}{|x|}}{e_1} - 2|x|,
\end{align*}
 the last inequality is from the (nonsharp) bound $|X| = \sqrt{|x|^2 + \beta(x)^2} \leq \sqrt{|x|^2+ 3|x|^2} = 2|x|$, which follows from  Lemma~\ref{lem: interior cone} \eqref{eqn: cone inclusion holds} and Remark~\ref{rmk: uni-sized nbhd in proj}, noting that since $B^n_{\kappa a}(\kappa ae_1)\subset B^n_{\frac{\conerad(\Theta)}{4}}(0)\subset B^n_{\frac{\conerad(\frac{1}{2})}{2}}(0)$ we have $0\geq \beta(x)\geq -\sqrt{3}$. 
Also note since $x\in B^n_{\kappa a}(\kappa ae_1)$ and $\kappa<1/16$, we have
\begin{equation*}
\frac{a}{2}\inner{\frac{x}{|x|}}{e_1} -2|x|\geq 4\kappa a \inner{ \frac{x}{|x|}}{e_1}  - 2 |x|\geq 0.
\end{equation*}

Now we can calculate (using $\kappa<\frac{1}{16}$ here)
\begin{align*}
\W_2^2(\mu, \bar{\mu}) & \geq \frac{1}{2} \int_{\{(X, \Xbar) \in \mathrm{spt} \gamma \mid X \in K_{a, \kappa} \}} |X-\Xbar|^2 d \gamma(X, \Xbar) \\ 
& \geq \frac{\rho_0}{2} \int_{B^n_{\kappa a}(\kappa ae_1)} \left(\frac{a}{2} \langle \frac{x}{|x|}, e_1 \rangle - 2|x|\right)^2 \sqrt{1+|\nabla^n \beta (x)|^2} dx \\
& \geq \frac{\rho_0 \mathcal{H}^{n-2}( \mathbb{S}^{n-2} )}{2} \int_0^{\frac{\pi}{2}} \int_0^{{2\kappa a \cos \phi}} r^{n-1} (\frac{a}{2} \cos \phi - 2r)^2dr \sin^{n-2} \phi d \phi \\
& =  a^{n+2}\left( \frac{1}{4n} - \frac{2^2\kappa}{n+1} + \frac{2^4\kappa^2}{n+2}\right) 2^n   \rho_0 \mathcal{H}^{n-2}(\mathbb{S}^{n-2})\kappa^n \int_0^{\frac{\pi}{2}} \cos^{n+2} \phi \sin^{n-2} \phi d \phi\\
& \geq  a^{n+2}\left( \frac{1}{4n} - \frac{1}{4(n+1)}\right) 2^n  \rho_0 \mathcal{H}^{n-2}(\mathbb{S}^{n-2}) \kappa^n \int_0^{\frac{\pi}{2}} \cos^{n+2} \phi \sin^{n-2} \phi d \phi\\
& =  a^{n+2}\left( \frac{2^{n-2}  \rho_0 \mathcal{H}^{n-2}(\mathbb{S}^{n-2})}{n(n+1)}\right)  \kappa^n \int_0^{\frac{\pi}{2}} \cos^{n+2} \phi \sin^{n-2} \phi d \phi.
\end{align*}
Since $a=\abs{\proj^\Omega_{X_0}(X_0-\Xbar_0)}$ this finishes the proof.
\end{proof}

Let us write $d_{\partial \Omega}$ for the geodesic distance on $\partial \Omega$, which we recall is defined by $d_{\partial \Omega}(X_1, X_2):=\inf \int_0^1 \abs{\dot\sigma(t)}dt$, where the infimum is over $C^1$ curves $\sigma$ with $\sigma(0)=X_1$, $\sigma(1)=X_2$. Since $\partial \Omega$ is $C^1$ regular, there is a constant $C_{\partial \Omega}\geq 1$ such that 
\begin{align}\label{eqn: geodesic distance bound}
    \lvert X_1-X_2\rvert\leq d_{\partial \Omega}(X_1, X_2)\leq C_{\partial \Omega} \abs{X_1-X_2}
\end{align}
for all $X_1$, $X_2\in \partial \Omega$.

We now show that if the dual potential $u$ has sufficiently small Lipschitz constant, then the $c$-subdifferential image at a point $X$ must remain  close to $X$; in particular this will be applied to obtain \eqref{eqn: csub in same side}.
\begin{prop}\label{prop: nonsplitting}
 Let $\Theta\in [0, 1)$ and $\delta>0$. Assume $u$ is a $c$-convex function with $d_{\partial \Omega}$-Lipschitz constant $L_u>0$  such that 
\begin{align}\label{eqn: grad bound}
    L_u<\frac{\delta\conerad(\Theta)}{2C_{\partial \Omega}}.
\end{align} 
Then for any $X\in \partial \Omega$, we must have  $\csubdiff[\Omega]{u}{X}\subset B^{n+1}_{\delta\conerad(\Theta)}(X)$.
\end{prop}
\begin{proof}
Suppose by contradiction for some $\Xbar\in \csubdiff[\Omega]{u}{X}$,
\begin{align}\label{eqn: not in conerad ball}
    \abs{X-\Xbar}\geq \delta\conerad(\Theta).
\end{align}
Then,
\begin{align*}
    u(\Xbar)-u(X)&\geq -c(\Xbar, \Xbar)+c(X, \Xbar)=\frac{\abs{X-\Xbar}^2}{2}.
\end{align*}
At the same time we have
\begin{align*}
    u(\Xbar)-u(X)&\leq L_ud_{\partial \Omega}(X, \Xbar)\leq C_{\partial \Omega}L_u\abs{X-\Xbar},
\end{align*}
thus combining the above two inequalities and using \eqref{eqn: grad bound} gives
\begin{align*}
    \delta\conerad(\Theta)&> 2C_{\partial \Omega}L_u\geq\abs{X-\Xbar},
\end{align*}
a contradiction with \eqref{eqn: not in conerad ball}, thus no such $\Xbar$ can exist.
\end{proof}

For the remainder of the paper, we fix probability measures  $\mu=\rho d\mathcal{H}^n$, and $\mubar=\rhobar d\mathcal{H}^n$ with $\rho$ and $\rhobar$ bounded away from zero and infinity. We will apply Proposition~\ref{prop: stay away} to an approximating sequence to show there is a dual potential on $\Omega$ whose Lipschitz constant can be controlled by $\W_2(\mu, \mubar)$, which in turn allows us to apply Proposition~\ref{prop: nonsplitting}. First recall: 
\begin{defin}\label{def: hausdorff distance}
    If $\Omega$, $\Omega'\subset \R^{n+1}$, then the \emph{Hausdorff distance} between them is defined by
    \begin{align*}
        \hdist{\Omega}{\Omega'}:=\max\left(\sup_{X\in \Omega}\dist{X}{\Omega'},\ \sup_{X'\in \Omega'}\dist{X'} {\Omega}\right).
    \end{align*}
\end{defin}
By \cite[Main Theorem]{Schneider84}, there exists a sequence $\{\Omega_k\}_{k=1}^\infty$ of (smooth) uniformly convex bodies which converge in Hausdorff distance to $\Omega$ as $k\to\infty$. In particular, we may take uniform constants
\begin{align*}
    \inrad[k]=\inrad, \qquad\outrad[k]=\outrad,\qquad\forall k,
\end{align*}
where $\inrad$ and $\outrad$ are defined in~\eqref{eqn: uniform balls}. 
By \cite[Theorem 14]{Wills07}, the sequence $\{\partial\Omega_k\}_{k=1}^\infty$ also converges in Hausdorff distance to $\partial \Omega$.

Our estimates rely on the quantity $\conerad(\Theta)$ from Lemma~\ref{lem: interior cone}, a simple compactness argument shows this value can be taken to be uniform for $\Omega$ and the approximating sequence $\Omega_k$.
\begin{lem}\label{lem: hausdorff stable}
For any $\Theta\in (0, 1)$, there exists $\conerad(\Theta)>0$ depending only on $\Omega$ and $\Theta$ such that 
    \begin{equation}\label{eqn: conerad for all k}
        \max\left(\omega_{\Omega}(\conerad(\Theta)), \omega_{\Omega_k}(\conerad(\Theta))\right) < \sqrt{2-2\Theta},\qquad\forall k\in \mathbb{N}.
    \end{equation}
\end{lem}
\begin{proof}
Since $\partial\Omega$ compact and $C^1$, it is sufficient to verify there is a function such that\\ $\omega_{\Omega_k}(\conerad(\Theta)) < \sqrt{2-2\Theta}$ for all $k\in \mathbb{N}$. Suppose the claim fails, then we may assume there exist $X_{1, k}$, $X_{2, k}\in \partial \Omega_k\subset B^{n+1}_{R_\Omega}(0)$ such that $\abs{X_{1, k}-X_{2, k}}<\frac{1}{k}$ but $\abs{\normal{\Omega_k}{X_{1, k}}-\normal{\Omega_k}{X_{2, k}}}>\sqrt{2-2\Theta}$. By boundedness we may pass to subsequences to assume that for $i=1$ and $2$,  $X_{i, k}\to X_0$ and $\normal{\Omega_k}{X_{i, k}}\to V_i$ for some $X_0\in \R^{n+1}$ and unit length vectors $V_i$ satisfying $\abs{V_1-V_2}\geq \sqrt{2-2\Theta}$. By \cite[Theorem 14]{Wills07}, we have $\hdist{\partial\Omega_k}{\partial\Omega}=\hdist{\Omega_k}{\Omega}\to 0$, hence we find $X_0\in \partial \Omega$. Now for any $X\in \Omega$, there is a sequence $X_k\in \Omega_k$ converging to $X$, hence
\begin{align*}
    \inner{V_i}{X-X_0}&=\lim_{k\to\infty}\inner{\normal{\Omega_k}{X_{i, k}}}{X_k-X_{i, k}}\leq 0,
\end{align*}
hence $V_1$, $V_2\in \normal{\Omega}{X_0}$. However this is a contradiction as $\Omega$ is $C^1$ and has a unique unit outer normal.
\end{proof}
For the remainder of this paper we assume $\conerad$ is a decreasing function satisfying \eqref{eqn: conerad for all k} as in Lemma~\ref{lem: hausdorff stable}. Next we construct pairs of probability measures $\mu_k$ and $\mubar_k$ on $\partial \Omega_k$ from $\mu$ and $\mubar$ with densities bounded uniformly away from zero and infinity.
\begin{lem}
\label{lem: rad proj bi-Lip}
If $\pi_\Omega : \partial \Omega \to \S^n$ is defined by
\begin{equation}\label{eqn: sphere projection map}
    \pi_\Omega(X) = \frac{X}{|X|},
\end{equation}
then $\pi_\Omega$ satisfies
\begin{equation}\label{eqn: proj lipschitz bound}
    \frac{1}{\tilde L_{\pi_\Omega}}|X-Y| \leq d_{\S^n}(\pi_\Omega(X), \pi_\Omega(Y)) \leq \tilde L_{\pi_\Omega}d_{\partial\Omega}(X,Y),
\end{equation}
where $\tilde L_{\pi_\Omega}>0$ only depends on $\inrad$ and $\outrad$.
\end{lem}
\begin{proof}
We first show that the map $\pi_\Omega : \partial \Omega \to \S^n$ is Lipschitz. Indeed,
\begin{align*}
d_{\partial\Omega}(X,Y)^2&\geq|X-Y|^2  = |X|^2 + |Y|^2 - 2\inner{X}{Y} 
 = |X||Y| \left( \frac{|X|}{|Y|} + \frac{|Y|}{|X|} - 2 \frac{\inner{X}{Y}}{|X||Y|} \right) \\
& \geq \inrad^2 \left( 2\sqrt{\frac{|X|}{|Y|} \cdot \frac{|Y|}{|X|}} - 2 \frac{\inner{X}{Y}}{|X||Y|} \right) 
 = \inrad^2 \left( 2 - 2 \frac{\inner{X}{Y}}{|X||Y|} \right) 
 = \inrad^2 \left| \frac{X}{|X|} - \frac{Y}{|Y|} \right|^2.
\end{align*}
Next we show that the inverse map $\pi_\Omega^{-1}: \S^n \to \partial \Omega$ is Lipschitz. Let $\mathcal{R}: \S^n\to [\inrad, \outrad]$ denote the radial function of $\Omega$, that is,
\begin{align*}
    \mathcal{R}(w):=\sup\left\{r>0\mid rw\in \Omega\right\}.
\end{align*}
Then using \cite[Theorem 2]{Toranzos67} in the final line below, for any $X$, $Y\in \partial \Omega$ we have 
\begin{align*}
    \lvert X-Y\rvert&=\left\lvert \mathcal{R}(\frac{X}{\lvert X\rvert})\frac{X}{\lvert X\rvert}-\mathcal{R}(\frac{Y}{\lvert Y\rvert})\frac{Y}{\lvert Y\rvert}\right\rvert\\
    &\leq \left\lvert \mathcal{R}(\frac{X}{\lvert X\rvert})\frac{X}{\lvert X\rvert}-\mathcal{R}(\frac{Y}{\lvert Y\rvert})\frac{X}{\lvert X\rvert}\right\rvert+\left\lvert \mathcal{R}(\frac{Y}{\lvert Y\rvert})\frac{X}{\lvert X\rvert}-\mathcal{R}(\frac{Y}{\lvert Y\rvert})\frac{Y}{\lvert Y\rvert}\right\rvert\\
    &\leq \left\lvert \mathcal{R}(\frac{X}{\lvert X\rvert})-\mathcal{R}(\frac{Y}{\lvert Y\rvert})\right\rvert+\outrad\left\lvert \frac{X}{\lvert X\rvert}-\frac{Y}{\lvert Y\rvert}\right\rvert\leq C d_{\S^n} \left(\frac{X}{\lvert X\rvert}, \frac{Y}{\lvert Y\rvert}\right)
\end{align*}
for some constant $C>0$ depending only on $\inrad$ and $\outrad$.
\end{proof}
Next we show a quick lemma on the uniform comparability of the intrinsic distances on $\partial \Omega_k$ and the Euclidean distance. 
\begin{lem}\label{lem: uniform comparability}
    There exists a constant $C_\Omega>0$ depending only on $\Omega$ such that $d_{\partial\Omega_k}(X, Y)\leq C_\Omega \lvert X-Y\rvert$ for all $X$, $Y\in \partial \Omega_k$ and all $k$ sufficiently large.
\end{lem}
\begin{proof}
    If $X$, $Y\in \partial \Omega_k$ with $\lvert X-Y\rvert\leq \frac{\conerad(\frac{1}{2})}{4}$, by Remark~\ref{rmk: uni-sized nbhd in proj} we have $Y\in \proj_X^{\Omega_k}(\partial\Omega_k^+(X))$, and we can write a portion of $\partial \Omega_k$ as the graph of a $C^1$ concave function $\beta_k$ over $\proj_X^{\Omega_k}(\partial\Omega_k^+(X))$. Rotate coordinates so that $\normal{\Omega_k}{X}=e_{n+1}$ and write $x:=\proj_X^{\Omega_k}(X)$, $y=\proj_X^{\Omega_k}(Y)$, then for all $t\in [0, 1]$ the segment $x_t:=(1-t)x+ty$ lies in the domain of $\beta_k$ and we can define the $C^1$ curve $\gamma(t):=(x_t, \beta_k(x_t))$ 
    lying in $\partial \Omega_k$.
    Then by Lemma~\ref{lem: unif size nbhd lip beta}, 
    \begin{align*}
        d_{\partial \Omega_k}(X, Y)&\leq \int_0^1 \lvert \dot \gamma(t)\rvert dt=\int_0^1 \sqrt{\lvert x-y\rvert^2+\inner{ \nabla^n\beta_k(x_t)}{y-x}^2}dt\leq \lvert X-Y\rvert\sqrt{1+\frac{64\outrad^2}{\conerad(\frac{1}{2})}}.
    \end{align*}
    Otherwise by \cite[Theorem 1]{Makai73}, we find that
    \begin{align*}
        d_{\partial \Omega_k}(X, Y)\leq \pi\outrad\leq \frac{4\pi\outrad}{\conerad(\frac{1}{2})}\lvert X-Y\rvert.
    \end{align*}
\end{proof}
We are now ready to construct the aforementioned family of measures with density bounds.
\begin{lem}
\label{lem: rad proj push}
Let $\pi_\Omega : \partial \Omega \to \S^n$, $\pi_{\Omega_k} : \partial \Omega_k \to \S^n$ be defined by \eqref{eqn: sphere projection map}, and
\begin{align*}
    \mu_k:=(\pi_{\Omega_k}^{-1}\circ \pi_\Omega)_\sharp \mu, \quad\mubar_k:=(\pi_{\Omega_k}^{-1}\circ \pi_\Omega)_\sharp \mubar.
\end{align*}
Then $\mu_k$ and $\mubar_k$ are probability measures on $\partial\Omega_k$, absolutely continuous with respect to $\mathcal{H}^n$ with densities $\rho_k$ and $\rhobar_k$ where
\begin{align*}
    L_{\pi_\Omega}^{-2n} (\essinf_{\partial\Omega} \rho) &\leq \rho_k \leq L_{\pi_\Omega}^{2n} (\esssup_{\partial\Omega} \rho),\quad L_{\pi_\Omega}^{-2n} (\essinf_{\partial\Omega} \rhobar) \leq \rhobar_k \leq L_{\pi_\Omega}^{2n} (\esssup_{\partial\Omega} \rhobar),
\end{align*}
for some constant $L_{\pi_\Omega}\geq 1$ which depends only on $\Omega$.
\end{lem}
\begin{proof}
    Let us write $B^{\Omega_k}_r(X)$ for an open geodesic ball in $\partial \Omega_k$ of radius $r$ centered at $X$. Fix $X_0\in \partial \Omega_k$ and $r>0$ small. Combining Lemmas~\ref{lem: rad proj bi-Lip} and \ref{lem: uniform comparability}, we can see $\pi_\Omega$ and $\pi_{\Omega_k}$ are all uniformly Lipschitz with respect to the intrinsic distances, with some constant $L_{\pi_\Omega}>0$ depending only on $\Omega$ (which we may assume greater than $1$), thus using \cite[Theorem 7.14.33]{Bogachev07} we obtain
    \begin{align*}
        \mu_k(B^{\Omega_k}_r(X_0))&=\int_{\pi_\Omega^{-1}(\pi_{\Omega_k}(B^{\Omega_k}_r(X_0)))}\rho d\mathcal{H}^n\\
        &\leq (\esssup_{\partial \Omega}\rho) \mathcal{H}^n(\pi_\Omega^{-1}(\pi_{\Omega_k}(B^{\Omega_k}_r(X_0))))\leq L_{\pi_\Omega}^{2n} (\esssup_{\partial \Omega}\rho) \mathcal{H}^n(B^{\Omega_k}_r(X_0))
    \end{align*}
    and 
    \begin{align*}
        \mu_k(B^{\Omega_k}_r(X_0))&=\int_{\pi_\Omega^{-1}(\pi_{\Omega_k}(B^{\Omega_k}_r(X_0)))}\rho d\mathcal{H}^n\\
        &\geq (\essinf_{\partial \Omega}\rho) \mathcal{H}^n(\pi_\Omega^{-1}(\pi_{\Omega_k}(B^{\Omega_k}_r(X_0))))\geq L_{\pi_\Omega}^{-2n} (\essinf_{\partial \Omega}\rho) \mathcal{H}^n(B^{\Omega_k}_r(X_0)).
    \end{align*}
    By the Lebesgue differentiation theorem (for example, \cite[Theorem 5.8.8]{Bogachev07}), dividing the above by $\mathcal{H}^n(B^{\Omega_k}_r(X_0))$ and taking $r\to 0$ yields the lemma (with a similar argument for $\mubar_k$).
\end{proof}

For the remainder of the paper, we take
\begin{align}\label{eqn: density lower bound}
    \rho_0:=L_{\pi_\Omega}^{-2n}\min\left(\essinf_{\partial \Omega}\rho, \essinf_{\partial \Omega}\rhobar\right).
\end{align}
If $\gamma$ is a Kantorovich solution between $\mu$ and $\mubar$ (which exists by \cite[Theorem 4.1]{Villani09}), then $((\pi_{\Omega_k}^{-1}\circ \pi_\Omega)\times (\pi_{\Omega_k}^{-1}\circ \pi_\Omega))_\sharp \gamma\in \Pi(\mu_k, \mubar_k)$, hence using Lemma~\ref{lem: rad proj bi-Lip},
\begin{align}\label{eqn: W2 approximation bound}
    \W_2^2(\mu_k, \mubar_k)&\leq \int_{\partial \Omega_k\times\partial\Omega_k}\lvert x-y\rvert^2 \distop((\pi_{\Omega_k}^{-1}\circ \pi_\Omega)\times (\pi_{\Omega_k}^{-1}\circ \pi_\Omega))_\sharp \gamma(x, y)\notag\\
    &= \int_{\partial \Omega\times\partial\Omega}\lvert \pi_{\Omega_k}^{-1}\circ \pi_\Omega(x)-\pi_{\Omega_k}^{-1}\circ \pi_\Omega(y)\rvert^2 d \gamma(x, y)\notag\\
    &\leq L_{\pi_\Omega}^2\int_{\partial \Omega\times\partial\Omega}\lvert x-y\rvert^2 d \gamma(x, y)= L_{\pi_{\Omega}}^2\W_2^2(\mu, \mubar).
\end{align} 
\begin{lem}\label{lem: measures weakly converge}
The sequences $\mu_k$ and $\mubar_k$ weakly converge to $\mu$ and $\mubar$ respectively.
\end{lem}
\begin{proof}
Fix $X\in \partial \Omega$, then the sequence of points $\{\pi_{\Omega_k}^{-1}(\pi_\Omega(X))\}_{k=1}^\infty$ lies on the ray through $X$ emanating from the origin, hence $\pi_{\Omega_k}^{-1}(\pi_\Omega(X))=t_kX$ for some $t_k>0$. Suppose we consider an arbitrary subsequence of $\{t_k\}_{k=1}^\infty$, not relabeled. Since $\partial\Omega_k\subset B^{n+1}_{\outrad}(0)$ for all $k$, we may pass to a further subsequence to assume $\lim_{k\to\infty}t_k=t_\infty$ for some $t_\infty\in [\inrad/\abs{X}, \outrad/\abs{X}]$. By Hausdorff convergence, there exist $X_k\in \partial \Omega$ such that $\lim_{k\to\infty}\lvert X_k-t_kX\rvert=0$. Again by compactness, passing to a subsequence we can assume $X_k\to X_\infty$ for some $X_\infty\in \partial \Omega$. Thus $X_\infty=t_\infty X$, however since $0$ is in the interior of $\Omega$ and $\Omega$ is convex, this is only possible if $t_\infty=1$. Hence we see any subsequence has a further subsequence for which
\begin{align*}
    \lim_{k\to\infty}\pi_{\Omega_k}^{-1}(\pi_\Omega(X))=X,
\end{align*}
hence the full sequence also exhibits the same convergence, in other words $\pi_{\Omega_k}^{-1}\circ\pi_\Omega$ converges pointwise to the identity map. Now suppose $\phi\in C_b(\R^n)$, then using dominated convergence we have
\begin{align*}
    \int_{\R^n}\phi d\mu_k&=\int_{\partial\Omega}\phi(\pi_{\Omega_k}^{-1}(\pi_\Omega(X)))d\mu(X)\to \int_{\partial\Omega}\phi(X) d\mu(X)=\int_{\R^n}\phi d\mu,
\end{align*}
proving weak convergence of $\mu_k$ to $\mu$; a similar argument shows $\mubar_k$ weakly converges to $\mubar$.
\end{proof}
\begin{rmk}\label{rmk: limiting potential}
    By \cite[Theorem 5.10 (iii)]{Villani09} there exist $c$-convex functions $u_k$ where $(u_k, u_k^c)$ are maximizers in the dual problem \eqref{eqn: kantorovich dual} between $\mu_k$ and $\mubar_k$ with $\partial\Omega$ replaced by $\partial\Omega_k$, translating by constants we may assume $u_k(0)=0$ for all $k$. Since the measures $\mu_k$ and $\mubar_k$ are all supported in $B^{n+1}_{\outrad}(0)$, we may apply \cite[Theorem 1.52]{Santambrogio15} to pass to a subsequence and find that $u_k$ converges uniformly on $B^{n+1}_{\outrad}(0)$ to some $u$ where $(u, u^c)$ is a maximizer in the dual problem \eqref{eqn: kantorovich dual} between $\mu$ and $\mubar$. By \cite[Theorem 5.10 (iii)]{Villani09}, there exists a pair $(\tilde u, \tilde u^c)$ which is a maximizer in the dual problem and $\tilde u$ is $c$-convex, by \cite[Proposition 4.1]{Loeper09} we see $u-\tilde u$ is a constant, hence $u$ is also $c$-convex.
\end{rmk}

We now show we can control the Lipschitz constant of the limiting $u$ in terms of $\W_2(\mu, \mubar)$. This is done by applying the stay-away estimate Proposition~\ref{prop: stay away}. However, we cannot directly apply the estimate on $\partial\Omega$, as the estimate is only applicable to points in $\partial\Omega^+(X)\cap \csubdiff[\Omega]{u}{X}$. The results of \cite{GangboMcCann00} will guarantee the existence of such points for \emph{uniformly convex} domains, thus we must first apply the estimate to the approximating bodies, then pass to the limit, exploiting that the constants in Proposition~\ref{prop: stay away} are stable under Hausdorff convergence.
\begin{lem}\label{lem: limiting potential lip bound}
The potential $u$ is $d_{\partial \Omega}$-Lipschitz with constant $C_{\rho_0, \Omega}L_{\pi_\Omega}^2 \W_2(\mu, \mubar)^{\frac{2}{n+2}}$.
\end{lem}
\begin{proof}
Let us write
\begin{align*}
    \tilde{u}_k:&=u_k+\frac{\lvert \cdot\rvert^2}{2},\qquad
    \tilde{u}:=u+\frac{\lvert \cdot\rvert^2}{2},
\end{align*}
which are convex by Remark~\ref{rmk: subdifferentials}. 
By \cite[Lemma 1.6]{GangboMcCann00}, $u$ is tangentially differentiable for $\mathcal{H}^n$-a.e. $X\in \partial \Omega$ and thus for such $X$, the subdifferential $\subdiff{\tilde u}{X}$ consists of a line segment parallel to $\normal{\Omega}{X}$. Fix such a point $X\in \partial \Omega$ of tangential differentiability, then by Hausdorff convergence there exists a sequence $X_k\in \partial \Omega_k$ such that $\lvert X_k-X\rvert\to 0$ as $k\to\infty$; we also see
\begin{align}\label{eqn: projection is tangential gradient}
    \{\nabla^{\Omega} u(X)+X\}=\proj^{\Omega}_{X}(\subdiff{\tilde u}{X}).
\end{align}
 Possibly passing to a subsequence, we may assume that $\normal{\Omega_k}{X_k}$ converges to some unit vector $V_\infty\in \R^{n+1}$. Let $Z\in \Omega$ be arbitrary, then there is a sequence $Z_k\in \Omega_k$ with $\lvert Z_k-Z\rvert\to 0$ as $k\to\infty$. Then 
\begin{align*}
    \inner{Z-X}{V_\infty}=\lim_{k\to\infty}\inner{Z_k-X_k}{\normal{\Omega_k}{X_k}}\leq 0,
\end{align*}
thus $V_\infty$ is a unit outward normal to $\Omega$ at $X$. Since $\partial\Omega$ is $C^1$, this means
\begin{align}\label{eqn: normals converge}
    \lim_{k\to\infty}\normal{\Omega_k}{X_k}=\normal{\Omega}{X}.
\end{align}

Since each $\Omega_k$ is uniformly convex, by Lemma~\ref{lem: rad proj push} for each $k$, the pair $\mu_k$ and $\mubar_k$ are ``suitable measures'' in the sense of \cite[Definition 3.1]{GangboMcCann00}. Thus by \cite[Theorem 3.8]{GangboMcCann00} followed by \cite[Theorem 5.10 (iii)]{Villani09}, there exists a map $T_k^+: \partial \Omega_k\to \partial \Omega_k$ satisfying 
\begin{align*}
    \inner{X}{T_k^+(X)}&\geq 0,\quad\forall X\in \partial\Omega_k,\\
    \{(X, T_k^+(X))\mid X\in \partial \Omega_k\}&\subset \spt \gamma_k\subset \partial_c^{\Omega_k} u_k=\partial \tilde{u}_k\cap (\partial \Omega_k\times \partial \Omega_k),
\end{align*}
then Proposition~\ref{prop: stay away} combined with \eqref{eqn: W2 approximation bound} implies that (with $\rho_0$ as in \eqref{eqn: density lower bound}) $\lvert \proj^{\Omega_k}_{X_k}(T_k^+(X_k)-X_k)\rvert\leq C_{\rho_0, \Omega}L_{\pi_{\Omega}}^2\W_2(\mu, \mubar)^{\frac{2}{n+2}}$. Passing to another subsequence, we may assume $T_k^+(X_k)$ converges to some  $\Xbar_\infty\in \partial \Omega$, then \eqref{eqn: normals converge} yields
\begin{align}\label{eqn: approx gradient bound}
    \lvert \proj^{\Omega}_{X}(\Xbar_\infty-X)\rvert\leq C_{\rho_0, \Omega}L_{\pi_{\Omega}}^2\W_2(\mu, \mubar)^{\frac{2}{n+2}}.
\end{align}

By \cite[Theorem 24.5]{Rockafellar70}, for any $\epsilon>0$ there is $K_\epsilon$ such that for all $k\geq K_\epsilon$, we have $\distop(T_k(X_k), \subdiff{\tilde u}{X})<\epsilon$, then since $\subdiff{\tilde u}{X}$ is closed, we have $\Xbar_\infty\in \subdiff{\tilde u}{X}$. 
Hence combining with \eqref{eqn: projection is tangential gradient} and \eqref{eqn: approx gradient bound} gives
\begin{align}\label{eqn: limit gradient bound}
    \lvert \nabla^\Omega u(X)\rvert\leq C_{\rho_0, \Omega}L_{\pi_{\Omega}}^2\W_2(\mu, \mubar)^{\frac{2}{n+2}},\qquad \mathcal{H}^n\textnormal{-}a.e.\ X\in \partial \Omega.
\end{align}
Now for $X_0\in \partial \Omega$, we can make a rotation so that $\normal{\Omega}{X_0}=e_{n+1}$ and $\Omega_{X_0}\subset \R^n$, then viewing a portion of $\partial \Omega$ as the graph of $\beta_{X_0}$ over $\Omega_{X_0}$ gives a $C^1$-coordinate chart. Let us write $x_0=\proj^\Omega_{X_0}(X_0)$, then using Lemma~\ref{lem: unif size nbhd lip beta} for $x_1$, $x_2\in B^n_{\frac{\conerad(\frac{1}{2})}{4}}(x_0)$ we see
\begin{align}\label{eqn: bilipschitz}
    \lvert x_1-x_2\rvert&\leq \sqrt{\lvert x_1-x_2\rvert^2+(\beta_{X_0}(x_1)-\beta_{X_0}(x_2))^2}
    \leq \frac{\sqrt{\conerad(\frac{1}{2})^2+16\diam(\Omega)^2}}{\conerad(\frac{1}{2})} \lvert x_1-x_2\rvert.
\end{align}
In these coordinates, we can also calculate the Riemannian metric tensor on $\partial \Omega$ as
\begin{align*}
    g_{ij}(x)=\delta_{ij}+\partial_{x_i}\beta_{X_0}(x)\partial_{x_j}\beta_{X_0}(x).
\end{align*}
Now defining 
\begin{align*}
    U_{X_0}(x):=u(x, \beta_{X_0}(x)),
\end{align*}
by \eqref{eqn: bilipschitz} we see $U_{X_0}$ is differentiable for a.e. $x\in B^n_{\frac{\conerad(\frac{1}{2})}{4}}(x_0)$. For such $x$, we have
\begin{align*}
    \lvert \nabla^n U_{X_0}(x)\rvert&\leq \sqrt{\lvert \nabla^n U_{X_0}(x)\rvert^2+\inner{\nabla^n\beta_{X_0}(x)}{ \nabla^n U_{X_0}(x)}^2}=\lvert \nabla^\Omega u(x, \beta_{X_0}(x))\rvert\\
    &\leq C_{\rho_0, \Omega}L_{\pi_{\Omega}}^2\W_2(\mu, \mubar)^{\frac{2}{n+2}}
\end{align*}
by \eqref{eqn: limit gradient bound}. Since $u$ is uniformly Lipschitz on $\partial \Omega$ and \eqref{eqn: bilipschitz} holds, $U_{X_0}$ is Lipschitz on $B^n_{\frac{\conerad(\frac{1}{2})}{4}}(x_0)$, hence it belongs to $W^{1, \infty}(B^n_{\frac{\conerad(\frac{1}{2})}{4}}(x_0))$ (see \cite[Section 5.8, Theorem 4]{Evans10}). By \cite[Section 5.8, Theorem 5]{Evans10} and  \cite[Section 5.8, Proof of Theorem 4]{Evans10} combined with \eqref{eqn: bilipschitz}, we find that
\begin{align}\label{eqn: coordinate Lipschitz bound}
    &\lvert U_{X_0}(x_1)-U_{X_0}(x_2)\rvert \leq C_{\rho_0, \Omega}L_{\pi_{\Omega}}^2\W_2(\mu, \mubar)^{\frac{2}{n+2}}\lvert x_1-x_2\rvert\notag\\
    &\leq C_{\rho_0, \Omega}L_{\pi_{\Omega}}^2\W_2(\mu, \mubar)^{\frac{2}{n+2}}d_{\partial\Omega}((x_1, \beta_{X_0}(x_1)), (x_2, \beta_{X_0}(x_2))),\quad\forall x_1, x_2\in B^n_{\frac{\conerad(\frac{1}{2})}{4}}(x_0).
\end{align}
Now, fix $X_1, X_2 \in \partial \Omega$, since $d_{\partial\Omega}$ is intrinsic and $\partial\Omega$ is compact, by \cite[Section 2.5]{BuragoBuragoIvanov01} there exists a curve $\gamma: [0, d_{\partial\Omega}(X_1, X_2)]\to \partial\Omega$ such that
\begin{align*}
    d_{\partial\Omega}(\gamma(t), \gamma(s))=\abs{t-s},\qquad \forall 0\leq s\leq t\leq d_{\partial\Omega}(X_1, X_2).
\end{align*}
Let us write
\begin{align*}
    \tau:=\frac{\conerad(\frac{1}{2})}{4C_{\partial\Omega}},\qquad d_{12}:=d_{\partial\Omega}(X_1, X_2),\qquad N:=\lfloor\frac{d_{12}}{\tau}\rfloor,
\end{align*}
then for $i=1, \ldots, N$, by \eqref{eqn: geodesic distance bound}, we have $\gamma([(i-1)\tau, i\tau]) \subset B^{n+1}_{\frac{\conerad(\frac{1}{2})}{4}}(\gamma((i-1)\tau))$ and $\gamma([N\tau, d_{12}]) \subset B^{n+1}_{\frac{\conerad(\frac{1}{2})}{4}}(\gamma(N\tau))$. Thus by \eqref{eqn: coordinate Lipschitz bound},
\begin{align*}
    \abs{u(X_1)-u(X_2)}&\leq \sum_{i=1}^N \abs{u(\gamma(i\tau)-u(\gamma((i-1)\tau))}+\abs{u(\gamma(d_{12})-u(\gamma(N\tau))}\\
    &\leq C_{\rho_0, \Omega}L_{\pi_\Omega}^2 \W_2(\mu, \mubar)^{\frac{2}{n+2}} \left(\sum_{i=1}^Nd_{\partial \Omega}(\gamma(i\tau),\gamma((i-1)\tau))+d_{\partial \Omega}(\gamma(N\tau), \gamma(d_{12}))\right)\\
    &= C_{\rho_0, \Omega}L_{\pi_\Omega}^2 \W_2(\mu, \mubar)^{\frac{2}{n+2}} d_{\partial \Omega}(X_1, X_2).
\end{align*}
\end{proof}
By the above estimate, when $\W(\mu, \mubar)$ is sufficiently small we will be able to ensure that the potential $u$ constructed above satisfies the condition~\eqref{eqn: section in same side as slope} for sections of sufficiently small height.
\begin{lem}\label{lem: small section in small nbhd Xbar}
Fix $\eta>0$, and let $u$ be the limiting potential constructed by approximation above. 
If 
\begin{equation}\label{eqn: W2 small bound 2}
\W_2(\mu,\mubar)^{\frac{2}{n+2}} <\max\left(\frac{\eta}{4C_{\rho_0,\Omega}}, \frac{\conerad(\frac{1}{2})}{2C_{\rho_0,\Omega}C_{\partial\Omega}L_{\pi_\Omega}^2}\right), 
\end{equation}
then there exists $h_{\eta, \Omega}>0$ depending only on $\eta$ and $\Omega$ such that for all $0<h<h_{\eta, \Omega}$, $X_0 \in \partial \Omega$, and $\Xbar_0 \in \csubdiff[\Omega]{u}{X_0}$ we have 
\begin{align*}
    S^{u}_{h, X_0, \Xbar_0} \subset  B^{n+1}_\eta(\Xbar_0) \cap \partial \Omega.
\end{align*}
\end{lem}
\begin{proof}
Suppose the lemma fails, then for $k\in \mathbb{N}$ there exist $X_k\in \partial \Omega$, $\Xbar_k\in \csubdiff[\Omega]{u}{X_k}$, $h_k\searrow 0$, and $Y_k\in S^{u}_{h_k, X_k, \Xbar_k}\cap \left(\partial \Omega\setminus B^{n+1}_\eta(\Xbar_k)\right)$. Let us write $S_k:=S^{u}_{h_k, X_k, \Xbar_k}$ and $m_k(X):= -c(X,\Xbar_k)+c(X_k, \Xbar_k) + u(X_k)$; note $m_k\leq u$ on $\partial \Omega$ for any $k$ as $\Xbar_k\in \csubdiff[\Omega]{u}{X_k}$. By compactness we may pass to subsequences (not relabeled) to assume $X_k$, $\Xbar_k$, $Y_k$ converge to $X_\infty$, $\Xbar_\infty$, $Y_\infty\in \partial \Omega$ respectively; we also write $m_\infty(X):= -c(X,\Xbar_\infty)+c(X_\infty, \Xbar_\infty) + u(X_\infty)$. 

Now since $m_k(X_k)\leq u(Y_k)\leq m_k(Y_k)+h_k$ for all $k$, by continuity we can pass to the limit to find 
\begin{align*}
    u(Y_\infty)= -c(Y_\infty, \Xbar_\infty)+c(X_\infty, \Xbar_\infty)+u(X_\infty)=m_\infty(Y_\infty).
\end{align*}
Since $m_\infty\leq u$ by continuity of $c$ and $u$, this implies $\Xbar_\infty\in \csubdiff[\Omega]{u}{Y_\infty}$.
By \eqref{eqn: W2 small bound 2} combined with Lemma~\ref{lem: limiting potential lip bound}, we see the $d_{\partial\Omega}$-Lipschitz constant of $u$ is less than $\frac{\conerad(\frac{1}{2})}{2C_{\partial\Omega}}$, hence by Proposition~\ref{prop: nonsplitting} we have $\Xbar_\infty \in \partial \Omega_{\frac{1}{2}}^+(Y_\infty)$. 
Then we can use Lemma~\ref{lem: gradient beta est} and Proposition~\ref{prop: stay away} to compute
\begin{equation*}
    \lvert Y_\infty - \Xbar_\infty \rvert \leq \sqrt{1+\sqrt{3}^2} \lvert \proj^\Omega_{Y_\infty} (Y_\infty - \Xbar_\infty) \rvert \leq \frac{\eta}{2},
\end{equation*}
however this contradicts that we must have $Y_\infty \not\in B^{n+1}_\eta(\Xbar_\infty)$, finishing the proof.
\end{proof}

\section{Proof of Main Theorems}\label{section: main proof}
We are now ready to prove our main theorems. Theorem~\ref{thm: monge solution} does not require the use of the estimates from Section~\ref{section: aleksandrov} but uses the dual potential constructed via approximating bodies from the previous section.
\begin{proof}[Proof of Theorem~\ref{thm: monge solution}]
By the smallness condition~\eqref{eqn: W2 small} on $\W_2(\mu, \mubar)$, Lemma~\ref{lem: limiting potential lip bound} implies the $d_{\partial\Omega}$-Lipschitz constant of $u$ is smaller than $\conerad(\frac{35}{36})/(16 C_{\partial\Omega})$.  Then we can apply Proposition~\ref{prop: nonsplitting} to find that  $\csubdiff[\Omega]{u}{X}\subset B^{n+1}_{\frac{\conerad(\frac{35}{36})}{8}}(X)$ for all $X\in \partial\Omega$ (that is, condition~\eqref{eqn: csub in same side} holds).

    By \cite[Lemma 1.6]{GangboMcCann00}, $u$ is tangentially differentiable at $\mathcal{H}^n$-a.e., hence $\mu$-a.e. $X\in \partial\Omega$. Fix such an $X$, then the function $U_X(y):=u(\cExp{X}{y})$ is differentiable at $0$, meaning $\locsubdiff{U_X}{0}$ is a singleton. Suppose now that $\Xbar_1$, $\Xbar_2\in \csubdiff{u}{X}$, this implies $[\Xbar_1]_X$, $[\Xbar_2]_X\in \locsubdiff{U_X}{0}$; in particular $\Xbar_2=\Xbar_1+\lambda\mathcal{N}_\Omega(X)$ for some $\lambda\in \R$. 
     From~\eqref{eqn: csub in same side} we have that $\inner{\mathcal{N}_\Omega(\Xbar_i)}{\mathcal{N}_\Omega(X)}>\frac{35}{36}$ for $i=1$, $2$, thus
     \begin{align*}
         0&\geq \inner{\Xbar_2-\Xbar_1}{\mathcal{N}_\Omega(\Xbar_1)}=\lambda\inner{\mathcal{N}_\Omega(X)}{\mathcal{N}_\Omega(\Xbar_1)},\\
         0&\geq \inner{\Xbar_1-\Xbar_2}{\mathcal{N}_\Omega(\Xbar_2)}=-\lambda\inner{\mathcal{N}_\Omega(X)}{\mathcal{N}_\Omega(\Xbar_2)},
     \end{align*}
     which implies $\lambda=0$. Thus we find $\csubdiff{u}{X}$ is a singleton.

         Now we can consider $T: \partial\Omega\to\partial\Omega$ which is single valued $\mu$-a.e. on $\partial\Omega$, defined by $\{T(X)\}=\csubdiff[\Omega]{u}{X}$. Since $c$ is continuous, $\partial_c^\Omega u$ is closed, hence by \cite[Theorem 5.10 (iii)]{Villani09} we see that \emph{any} Kantorovich solution $\gamma$ between $\mu$ and $\mubar$ satisfies $\spt\gamma\subset \partial^\Omega_cu$. This implies  $\int_{\partial\Omega\times\partial\Omega}\lvert T(X)-Y\rvert d\gamma(X,Y)=0$, thus by \cite[Lemma 2.4]{GangboMcCann00} we have $\gamma=(\Id\times T)_\#\mu$, finishing the proof.
\end{proof}

Next we work toward showing continuity of the optimal transport map $T$. For the remainder of this section, $C>0$ will denote a constant, whose value may change from line to line, which depends possibly only on $n$, $\Omega$, and the bounds on $\rho$ and $\rhobar$ away from zero and infinity (we will refer to such a constant as \emph{universal}); for two quantities $a$ and $b$ we write $a\lesssim b$ to indicate $a\leq Cb$ for a universal $C>0$ and $a\sim b$ to indicate that both $a\lesssim b$ and $b\lesssim a$. We start with a quick corollary to Theorem~\ref{thm: monge solution}.
\begin{cor}
Under the same hypotheses as Theorem~\ref{thm: monge solution}, for any Borel $\hat{A}\subset \partial \Omega$, 
\begin{align}\label{eqn: csubdiff comparable}
    \mathcal{H}^n(\csubdiff{u}{\hat{A}})\sim \mathcal{H}^n(\hat{A}).
\end{align}
\end{cor}
\begin{proof}
By Theorem~\ref{thm: monge solution} and the same theorem with $\mu$ and $\mubar$ reversed, we find a $\mu$-null set $\mathcal{N}$ such that $\csubdiff{u}{X}$ is a singleton for all $X\in \partial\Omega\setminus\mathcal{N}$, and a $\mubar$-null set $\bar{\mathcal{N}}$ such that $\csubdiff{u^c}{\Xbar}$ is a singleton for all $\Xbar\in \partial\Omega\setminus\bar{\mathcal{N}}$. Suppose we fix any Borel $\hat{A}\subset \partial \Omega$. If $X\in T^{-1}(\csubdiff{u}{\hat{A}}\setminus\bar{\mathcal{N}})$, there exists some $\Xbar\in \csubdiff{u}{\hat{A}}\setminus \bar{\mathcal{N}}$ such that $\csubdiff{u}{X}=\{\Xbar\}$, or equivalently $X\in \csubdiff{u^c}{\Xbar}$. This also implies there exists $\hat{X}\in \hat{A}$ such that $\Xbar\in \csubdiff{u}{\hat{X}}$, equivalently $\hat{X}\in \csubdiff{u^c}{\Xbar}$, then by the choice of $\bar{\mathcal{N}}$ this implies $\hat{X}=X$. In particular $X\in \hat{A}$, hence we find
\begin{align*}
    \mathcal{H}^n(\hat{A})
   &\sim \mu(\hat{A})
    \geq\mu(T^{-1}(\csubdiff{u}{\hat{A}}\setminus\bar{\mathcal{N}}))
    =\mubar(\csubdiff{u}{\hat{A}}\setminus\bar{\mathcal{N}})
    \sim \mathcal{H}^n(\csubdiff{u}{\hat{A}}).
\end{align*}
Since $\hat{A}\setminus\mathcal{N}\subset T^{-1}(\csubdiff{u}{\hat{A}})$ trivially, we can also see 
\begin{align*}
  \mathcal{H}^n(\csubdiff{u}{\hat{A}})\sim \mubar(\csubdiff{u}{\hat{A}})
    =\mu(T^{-1}(\csubdiff{u}{\hat{A}}))
    \geq\mu(\hat{A}\setminus\mathcal{N})\sim \mathcal{H}^n(\hat{A}).
\end{align*}
\end{proof}
Next we will use the Aleksandrov estimates from Section~\ref{section: aleksandrov} to prove Theorem~\ref{thm: nonstrictly convex body}.  The proof follows the method of \cite[Section 5]{GuillenKitagawa15} (see also \cite[Section 7]{FigalliKimMcCann13}). 
Roughly speaking, we first show that Theorems~\ref{thm: lower aleksandrov} and~\ref{thm: upper aleksandrov} can be applied to sections with sufficiently small heights of the (unique up to addition of constants) dual Kantorovich potential. Next, we follow the proof of the Caffarelli localization theorem in the interior of $\spt \mu$, \cite[Theorem 5.7]{GuillenKitagawa15}, with the lower and upper Aleksandrov estimates replaced by Theorems~\ref{thm: lower aleksandrov} and~\ref{thm: upper aleksandrov} respectively, this will imply all contact sets between $u$ and its supporting $c$-functions are singletons. Finally, noting that the roles of $\mu$ and $\mubar$ are symmetric, we obtain the same result for contact sets of $u^c$, which imply that all sets of the form $\csubdiff{u}{X}$ are singletons. This allows us to show the transport map is both single valued everywhere and continuous. 

\begin{proof}[Proof of Theorem~\ref{thm: nonstrictly convex body}]
Apply the approximation process in Remark~\ref{rmk: limiting potential} to obtain a pair $(u, u^c)$ which maximizes the dual problem between $\mu$ and $\mubar$ where $u$ is $c$-convex. By Theorem~\ref{thm: monge solution}, the map $T$  defined by $\csubdiff{u}{X}=\{T(X)\}$ is singled valued $\mu$-a.e. on $\partial\Omega$. Thus by the continuity of $u$ and $c$, it is sufficient to show that $\csubdiff{u}{X}$ is a singleton for \emph{every} $X\in \partial\Omega$.
    
    We first verify that conditions~\eqref{eqn: lower est section in same side},~\eqref{eqn: lower est section in center ball}, and~\eqref{eqn: lower est csub in same side} in Theorem~\ref{thm: lower aleksandrov} and conditions~\eqref{eqn: csub in same side} and~\eqref{eqn: section in same side as slope} in Theorem~\ref{thm: upper aleksandrov} hold for appropriate sections of $u$ and choices of sets $A$. We have already verified that~\eqref{eqn: csub in same side} holds in the proof of Theorem~\ref{thm: monge solution} above.  
    By~\eqref{eqn: W2 small}, we find the condition~\eqref{eqn: W2 small bound 2} holds with $\eta=\conerad(\frac{35}{36})/8$, then by Lemma~\ref{lem: small section in small nbhd Xbar} (see also Remark~\ref{rmk: section equivalence}), there is an $h_0>0$ depending only on $\Omega$ such that for any $X\in \partial \Omega$, $\Xbar\in \csubdiff{u}{X}$, and $0<h<h_0$, the section $S=S^u_{h, X, \Xbar}$ satisfies \eqref{eqn: section in same side as slope}. In turn, by Lemma~\ref{lem: interior cone} we see \eqref{eqn: lower est section in same side} holds with $\Theta=\frac{35}{36}$. 
    Now for each such section $S$, let $E$ be the John ellipsoid (recall~\eqref{eqn: john ellipsoid}) associated to 
    $\proj_{\Xbar}^\Omega(S)$; by \eqref{eqn: lower est section in same side} we see the center of $E$ is given by $\proj_{\Xbar}^\Omega(X_{cm})$ for some $X_{cm}\in \partial\Omega^+(\Xbar)$. Let $A:=(\proj_{\Xbar}^\Omega)^{-1}(\frac{1}{2}E)\cap\partial\Omega^+(\Xbar)$, where the dilation is with respect to $\proj_{\Xbar}^\Omega(X_{cm})$, clearly $2\proj_{\Xbar}^\Omega(A)\subset \proj_{\Xbar}^\Omega(S)$. 
    Finally for any $Y\in S^u_{h, X, \Xbar}$ and $\Ybar\in \csubdiff{u}{Y}$, by~\eqref{eqn: section in same side as slope} we have 
    \begin{align*}
    \lvert Y-X_{cm}\rvert&\leq \lvert Y-\Xbar\rvert+\lvert \Xbar-X_{cm}\rvert<\frac{\conerad(\frac{35}{36})}{4}\leq \frac{\conerad(\frac{1}{2})}{4}
    \end{align*}
    so~\eqref{eqn: lower est section in center ball} holds, while also using~\eqref{eqn: csub in same side} yields
    \begin{align*}
    \lvert \Ybar-X_{cm}\rvert&\leq \lvert\Ybar-Y\rvert+\lvert Y-\Xbar\rvert+\lvert \Xbar-X_{cm}\rvert<\frac{3\conerad(\frac{35}{36})}{8} ,   
    \end{align*}
    hence applying Lemma~\ref{lem: interior cone} we have that~\eqref{eqn: lower est csub in same side} also holds.

    Let us now fix $X_0\in \partial\Omega$ and $\Xbar_0\in \csubdiff{u}{X_0}$, and define 
    \begin{align*}
    m_0(X):&=-c(X, \Xbar_0)+c(X_0, \Xbar_0)+u(X_0),\\
        S_0:&=S^u_{0, X_0, \Xbar_0}=\{X\in \partial\Omega\mid u(X)=m_0(X)\};
    \end{align*}
    note by definition, $m_0\leq u$ everywhere. We will show that $S_0=\{X_0\}$; suppose this does not hold, then the set $\proj_{\Xbar_0}^\Omega(S_0)$ is convex, compact, and not a singleton, hence contains at least one exposed point which we call $x_e$. Then (making a rotation to identify the tangent plane to $\partial\Omega$ at $\Xbar_0$ with $\R^n$) there exists a unit vector $v_0\in \R^n$ such that
    \begin{align*}
        \inner{x-x_e}{v_0}\leq 0,\quad \forall x\in \proj_{\Xbar_0}^\Omega(S_0)
    \end{align*}
    and the inequality is strict whenever $x\neq x_e$; by \cite[Lemma 7.4]{GuillenKitagawa15} we may assume there exists a $\tau_0>0$ such that $x_e-\tau v_0\in \proj_{\Xbar_0}^\Omega(S_0)$ for all $\tau\in [0, \tau_0]$. We now construct a family of sections which ``chop off'' a small portion of $S_0$, while decaying in a specified manner. By~\eqref{eqn: lower est section in same side} the following are well-defined for $t>0$ sufficiently small, depending only on $\Xbar_0$:
    \begin{align*}
    \xbar_0:&=\proj_{\Xbar_0}^\Omega(\Xbar_0),\qquad
    \Xbar_t:=(\xbar_0+tv_0, \beta_{\Xbar_0}(\xbar_0+tv_0)),\\
        X_e:&=(x_e, \beta_{\Xbar_0}(x_e)).
    \end{align*}
    Then we can define
    \begin{align*}
        m_t(X):&=-c(X, \Xbar_t)+c(X_e, \Xbar_t)+m_0(X_e)+t,\\
        S_t:&=\{X\in \partial \Omega\mid u(X)\leq m_t(X)\}.
    \end{align*}
    By a compactness argument, for all $t>0$ sufficiently small, we see $S_t$ is contained in a neighborhood (in $\R^{n+1}$) of size $\conerad(\frac{35}{36})/8$ around $S_0$. Then by \eqref{eqn: section in same side as slope}, we have that $S_t\subset B^{n+1}_{\conerad(\frac{35}{36})}(\Xbar_0)$, thus Lemma~\ref{lem: interior cone} implies that $S_t\subset \partial\Omega^+(\Xbar_0)$. 
    Since $m_0\leq u$, we then find that for any $X\in S_t$ (writing $x:=\proj^\Omega_{\Xbar_0}(X)$),
    \begin{align*}
        m_t(X)-u(X)
        &\leq m_t(X)-m_0(X)
        =-c(X, \Xbar_t)+c(X_e, \Xbar_t)-(m_0(X)-m_0(X_e))+t\\
        &=-c(X, \Xbar_t)+c(X_e, \Xbar_t)-(-c(X, \Xbar_0)+c(X_e, \Xbar_0))+t\\
        &=t\left[\left.\frac{d}{ds}\left(-c(X, \Xbar_s)+c(X_e, \Xbar_s)\right)\right\vert_{s=0}+1\right]+o(t),\quad t\searrow 0,\\
        &=t\left(\inner{x-x_e}{v_0}+(\beta_{\Xbar_0}(x)-\beta_{\Xbar_0}(x_e))\inner{\nabla^n \beta_{\Xbar_0}(\xbar_0)}{v_0}+1\right)+o(t),\quad t\searrow 0,\\
        &\leq t\left(\diam(\Omega)+2\diam(\Omega)\abs{\nabla^n \beta_{\Xbar_0}(\xbar_0)}+1\right)+o(t),\quad t\searrow 0.
    \end{align*}
    Since $\abs{\nabla^n \beta_{\Xbar_0}(\xbar_0)}$ is bounded by a universal constant by Lemma~\ref{lem: gradient beta est} (clearly $\Xbar_0\in \partial \Omega^+_{1/2}(\Xbar_0)$), we can see that 
    \begin{align}\label{eqn: tilted section sup}
        0\leq \sup_{S_t}(m_t-u)\leq Ct+o(t).
    \end{align}
    At the same time, by condition~\eqref{eqn: lower est section in same side} and since $x_e\in \partial^n\proj_{\Xbar_0}^\Omega(S_0)$, we find that $m_0(X_e)=u(X_e)$, hence for any $t>0$ we have
    \begin{align}\label{eqn: tilted strictly bigger at exposed}
        m_t(X_e)=m_0(X_e)+t>u(X_e).
    \end{align}
    We can combine this with the fact that $u\equiv m_t$ on the relative boundary of $S_t$ to find that for any $t>0$, there is a translate $\hat{m}_t$ of $m_t$ by a constant such that $u-\hat{m}_t$ has a local minimum relative to $\partial \Omega$ at some $X_t\in \partial\Omega$. In turn this implies that the function $\hat{U}_t:=u\circ \exp^c_{X_t}$ satisfies 
    \begin{align*}
        [\Xbar_t]_{X_t}=-\nabla^\Omega_Xc(X, \Xbar_t)\vert_{X=X_t}\in \locsubdiff{\hat{U}_t}{0},
    \end{align*}
    thus by the ``local-to-global'' result Corollary~\ref{cor: local to global}, we see $\Xbar_t\in \csubdiff{u}{X_t}$. By taking $t>0$ sufficiently small, from~\eqref{eqn: tilted section sup} we can see that $S_t=S^u_{h_t, X_t, \Xbar_t}$
    where $0<h_t<h_0$. Thus defining $A_t:=(\proj_{\Xbar_t}^\Omega)^{-1}(\frac{1}{2}E_t)\cap\partial\Omega^+(\Xbar_t)$ where $E_t$ is the John ellipsoid associated to $\proj_{\Xbar_t}^\Omega(S_t)$, we now have that $A_t$ and $S_t$ satisfies all of the hypotheses of Theorems~\ref{thm: lower aleksandrov} and~\ref{thm: upper aleksandrov}. Note by condition~\eqref{eqn: section in same side as slope}, we can apply Lemma~\ref{lem: gradient beta est} to see that $S_t$ and $A_t$ are images of each other under maps whose Lipschitz constants are universal, thus 
    \begin{align}\label{eqn: section volume comparable}
        \mathcal{H}^n(A_t)\sim \mathcal{H}^n(S_t).
    \end{align}
    We must now choose appropriate directions to apply the upper estimate Theorem~\ref{thm: upper aleksandrov}. Let 
    \begin{align*}
        \hat{X}:=(x_e-\tau v_0, \beta_{\Xbar_0}(x_e-\tau v_0))
    \end{align*}
    where $\tau\in (0, \tau_0)$ is to be determined, and small enough that $\beta_{\Xbar_0}(x_e-\tau v_0)$ is defined.
    By the choice of $v_0$ we have $\hat{X}\in S_0$, then by a calculation similar to the one leading to~\eqref{eqn: tilted section sup} we have
    \begin{align*}
        m_t(\hat{X})-u(\hat{X})
        &=m_t(\hat{X})-m_0(\hat{X})\\
        &=t\left(\inner{-\tau v_0}{v_0}+(\beta_{\Xbar_0}(x_e-\tau v_0)-\beta_{\Xbar_0}(x_e))\inner{\nabla^n \beta_{\Xbar_0}(\xbar_0)}{v_0}+1\right)+o(t),\quad t\searrow 0,\\
        &\geq t\left(-\tau-\tau\sup_{0\leq s\leq \tau}\abs{\nabla^n\beta_{\Xbar_0}(x_e-s v_0)}\abs{\nabla^n \beta_{\Xbar_0}(\xbar_0)}+1\right)+o(t),\quad t\searrow 0.
    \end{align*}
    Since $\tau\leq \tau_0$ and $\proj^\Omega_{\Xbar_0}(S_0)$ is convex, by~\eqref{eqn: lower est section in same side} we have $(x_e-s v_0, \beta_{\Xbar_0}(x_e-s v_0))\in \partial\Omega^+_{\frac{35}{36}}(\Xbar_0)$ for all $s\in [0, \tau]$, thus by Lemma~\ref{lem: gradient beta est} we see that for a fixed $\tau$ small enough, for any $t$ sufficiently small the last expression in the inequality above is nonnegative, meaning $\hat{X}\in S_t$. By~\eqref{eqn: tilted strictly bigger at exposed} we also have $X_e\in S_t$, hence we may define
    \begin{align*}
        v_t:=\frac{\proj^\Omega_{\Xbar_t}(\hat{X})-\proj^\Omega_{\Xbar_t}(X_e)}{\abs{\proj^\Omega_{\Xbar_t}(\hat{X})-\proj^\Omega_{\Xbar_t}(X_e)}}.
    \end{align*}
    Note since $\mathcal{N}_\Omega$ is continuous on $\partial \Omega$, for all $t>0$ sufficiently small we obtain that  $\proj^\Omega_{\Xbar_t}(\hat{X})\neq \proj^\Omega_{\Xbar_t}(X_e)$, and we can also estimate
    \begin{align}\label{eqn: segment uniformly bounded below}
        l([S_t]_{\Xbar_t}, v_t)
        &\geq \abs{\proj^\Omega_{\Xbar_t}(\hat{X})-\proj^\Omega_{\Xbar_t}(X_e)}\geq \frac{1}{2}\abs{\proj^\Omega_{\Xbar_0}(\hat{X})-\proj^\Omega_{\Xbar_0}(X_e)}=\frac{\tau}{2}.
    \end{align}
    We also claim
    \begin{align}\label{eqn: distance to plane decays}
        \lim_{t\searrow 0}\distop([X_e]_{\Xbar_t}, \Pi^{v_t}_{[S_t]_{\Xbar_t}})=0.
    \end{align}
    If this were not the case, by \cite[Remark 4.4] {GuillenKitagawa15} there exists some $\delta_0>0$, sequence $t_k\searrow 0$, and $x_k\in \partial\proj^\Omega_{\Xbar_{t_k}}(S_{t_k})$ such that for all $k$,
    \begin{align*}
        \delta_0&\leq\distop([X_e]_{\Xbar_{t_k}}, \Pi^{v_{t_k}}_{[S_{t_k}]_{\Xbar_{t_k}}})
        =\distop(\proj^\Omega_{\Xbar_{t_k}}(X_e), \Pi^{v_{t_k}}_{\proj^\Omega_{\Xbar_{t_k}}(S_{t_k})})\\
        &=\inner{x_k-\proj^\Omega_{\Xbar_{t_k}}(X_e)}{v_{t_k}}.
    \end{align*}
    By compactness of $\partial\Omega$, we can pass to a subsequence to assume $X_k:=(x_k, \beta_{\Xbar_{t_k}}(x_k))$ converges to some $X_\infty\in \partial \Omega$, by continuity of $u$ and $c$ we have $X_\infty\in S_0$. However, this would imply $0<\delta_0\leq \inner{\proj^\Omega_{\Xbar_0}(X_\infty)-x_e}{v_0}$, contradicting the choice of $v_0$, thus~\eqref{eqn: distance to plane decays} must hold.

    Finally, we may apply Theorems~\ref{thm: lower aleksandrov} and~\ref{thm: upper aleksandrov} to each $S_t$ and $A_t$, then use \eqref{eqn: csubdiff comparable}, \eqref{eqn: tilted section sup}, \eqref{eqn: section volume comparable}, \eqref{eqn: segment uniformly bounded below}, and \eqref{eqn: distance to plane decays} to obtain as $t\searrow 0$,
    \begin{align*}
        \frac{(m_t(X_e)-u(X_e))^n}{\sup_{S_t}(m_t-u)^n}
        &=\left(\frac{t}{\sup_{S_t}(m_t-u)}\right)^n
        \geq \frac{1}{(C+\frac{o(t)}{t})^n}\to C^{-n}>0,\\
        \frac{(m_t(X_e)-u(X_e))^n}{\sup_{S_t}(m_t-u)^n}&\leq \frac{C\mathcal{H}^n(S_t)\mathcal{H}^n(\csubdiff{u}{S_t})\distop([S_t]_{\Xbar_t}, \Pi^{v_t}_{[S_t]_{\Xbar_t}})}{\mathcal{H}^n(A_t)\mathcal{H}^n(\csubdiff{u}{A_t})l([S_t]_{\Xbar_t}, v_t)}\to 0,
    \end{align*}
    which is a contradiction, thus $S_0=\{X_0\}$.

    Now since $W_2(\mubar, \mu)=W_2(\mu, \mubar)$ and $c(X, \Xbar)=c(\Xbar, X)$, we may apply the above proof to sections of $u^c$ instead of $u$ and recall Remark~\ref{rmk: contact set is dual c-sub} to obtain that for any $X$, $\csubdiff{u}{X}$ is a single point, which finishes the proof.
    
\end{proof}
Finally, we can use the method in \cite[Section 8]{GuillenKitagawa17} based on the theory of Forzani and Maldonado in \cite{ForzaniMaldonado04}, this was also adapted by Figalli, Kim, and McCann in \cite{FigalliKimMcCann13}. Since the proof follows what is presented in \cite[Section 8]{GuillenKitagawa17} rather closely, we mainly present some of the key steps and explain differences in notation.
\begin{proof}[Proof of Theorem~\ref{thm: Holder regularity}]
Again let $u$ be the $c$-convex potential obtained from Remark~\ref{rmk: limiting potential}. We will follow the sequence of proofs in \cite[Section 8]{GuillenKitagawa17}, with the choices $G(X, \Xbar, z):=-c(X, \Xbar)-z$ and $H(X, \Xbar, v):-c(X, \Xbar)-v$. Note that $S(X, \Xbar, h)$ is the equivalent of our $S^u_{h, X, \Xbar}$ (where always, $\Xbar\in \csubdiff{u}{X})$, with associated function $m_h(\cdot)=-c(\cdot, \Xbar)+c(X, \Xbar)+h$, $\partial_Gu=\partial_cu$, $p_{\Xbar, z}(X)=[X]_{\Xbar}$ (the variable $z$ can be ignored), and the phrase ``\emph{very nice}'' can be replaced by ``universal.''

First we can replace \cite[Theorem 2.1 and Theorem 2.2]{GuillenKitagawa17} by Theorem~\ref{thm: upper aleksandrov} and Theorem~\ref{thm: lower aleksandrov} respectively, and follow the proofs of \cite[Proposition 8.1 and Lemma 8.2]{GuillenKitagawa17}\footnote{We remark that all instances of $x$ and $\bar{x}$ in the statement and proof of \cite[Lemma 8.2]{GuillenKitagawa17} should actually be $x_0$ and $\bar{x}_0$} to see there exist universal constants $\kappa\in (0, 1)$ and $h_0>0$ such that for any $X_0\in \partial\Omega$ and $\Xbar_0\in \csubdiff{u}{X_0}$, 
\begin{align}\label{eqn: double height dilation}
    \proj^\Omega_{\Xbar_0}(S^u_{h, X_0, \Xbar_0})\subset \kappa \proj^\Omega_{\Xbar_0}(S^u_{2h, X_0, \Xbar_0}),
\end{align}
where the above dilation is with respect to $x_0:=\proj^\Omega_{\Xbar_0}(X_0)$. 
 This allows us to obtain the crucial \emph{engulfing property} as in \cite[Lemma 8.3]{GuillenKitagawa17}: specifically we can show that for any $X_0\in \partial\Omega$, $\Xbar_0\in \csubdiff{u}{X_0}$, $X_1\in S^u_{h, X_0, \Xbar_0}$ with $0<h<h_0$, and $\Xbar_1\in \csubdiff{u}{X_1}$, we have
\begin{align*}
    X_0\in S^u_{2\kappa h, X_1, \Xbar_1}.
\end{align*}
Indeed, suppose $h_0$ is small enough that the conditions for Theorems~\ref{thm: lower aleksandrov} and \ref{thm: upper aleksandrov} hold and $0<h<h_0$. Let $x_1^\partial\in \partial \proj^\Omega_{\Xbar_0}(S^u_{2h, X_0, \Xbar_0})$ be such that $\proj^\Omega_{\Xbar_0}(X_1)=(1-s_1)x_0+s_1x_1^\partial$ for some $s_1\in [0, 1]$; by~\eqref{eqn: double height dilation} we have $s_1\leq \kappa$. Write $X^\partial_1:=(x^\partial_1, \beta_{\Xbar_0}(x^\partial_1))$, and first suppose $-c(X^\partial_1, \Xbar_1)+c(X^\partial_1, \Xbar_0)+c(X_0, \Xbar_1)-c(X_0, \Xbar_0)>0$. Then by~\eqref{eqn: lower est section in same side}, we can apply Proposition~\ref{prop: qqconv holds} with the barred and unbarred variables reversed, then use that $\Xbar_1\in \csubdiff{u}{X_1}$, followed by $x_1^\partial\in \partial \proj^\Omega_{\Xbar_0}(S^u_{2h, X_0, \Xbar_0})$, then $\Xbar_1\in \csubdiff{u}{X_1}$ again to see
\begin{align*}
    &-c(X_1, \Xbar_1)+c(X_1, \Xbar_0)+c(X_0, \Xbar_1)-c(X_0, \Xbar_0)\\
    &\leq s_1(-c(X^\partial_1, \Xbar_1)+c(X^\partial_1, \Xbar_0)+c(X_0, \Xbar_1)-c(X_0, \Xbar_0))\\
    &\leq \kappa(-c(X^\partial_1, \Xbar_1)+c(X^\partial_1, \Xbar_0)+c(X_0, \Xbar_1)-c(X_0, \Xbar_0))\\
    &\leq\kappa(u(X^\partial_1)-c(X_1, \Xbar_1)-u(X_1)+c(X^\partial_1, \Xbar_0)+c(X_0, \Xbar_1)-c(X_0, \Xbar_0))\\
    &= \kappa(-c(X_1, \Xbar_1)-u(X_1)+c(X_0, \Xbar_1)+u(X_0)+2h)
    \leq 2\kappa h;
\end{align*}
the same final upper bound also holds if $-c(X^\partial_1, \Xbar_1)+c(X^\partial_1, \Xbar_0)+c(X_0, \Xbar_1)-c(X_0, \Xbar_0)\leq 0$.
Then since $\Xbar_0\in \csubdiff{u}{X_0}$,
\begin{align*}
    u(X_0)&\leq u(X_1)+c(X_1, \Xbar_0)-c(X_0, \Xbar_0)\leq -c(X_0, \Xbar_1)+c(X_1, \Xbar_1)+2\kappa h.
\end{align*}
    With this engulfing property in hand, we can follow the proof of \cite[Lemma 8.4]{GuillenKitagawa17}, and then the proof of \cite[Theorem 2.4]{GuillenKitagawa17} in the case labeled \textbf{When $\bm{u}$ is $\bm{C^1}$} to obtain that for any $X_0\in\partial\Omega$ and any $X$ in some neighborhood of $X_0$,
    \begin{align*}
        \lvert u(X)+c(X, \Xbar_0)+h\rvert\leq Cd_{\partial\Omega}(X, X_0)^{1+\beta}
    \end{align*}
    for some $\beta\in (0, 1]$. As $\partial\Omega$ is itself $C^{1, \alpha}$ regular for some $\alpha$,  then the above shows local $C^{1, \min(\alpha, \beta)}$ regularity of $u$ on $\partial\Omega$, which in turn shows H\"older regularity of the map $T$.
\end{proof}
\subsection*{Acknowledgments/Competing Interests}
J. Kitagawa's research was supported in part by National Science Foundation grants DMS-2000128 and DMS-2246606. The authors would like to thank the anonymous referee for a careful reading and comments, which have led to significant improvements in the presentation and structure of the paper.

The authors have no other competing interests to declare that are relevant to the content of this article.
\subsection*{Data availability} We do not analyze or generate any datasets.
\bibliography{nontwistedbib.bib}
\bibliographystyle{plain}
\end{document}